\documentclass{birkjour_t2}

\usepackage{amssymb}

\usepackage{graphicx}
\usepackage{latexsym}
\usepackage{amsmath}
\usepackage{amssymb}
\usepackage[colorlinks]{hyperref}
\usepackage{paralist}

 \newtheorem{thm}{Theorem}[section]
 \newtheorem{cor}[thm]{Corollary}
 \newtheorem{lem}[thm]{Lemma}

 \newtheorem{defn}[thm]{Definition}
 \newtheorem{rem}[thm]{Remark}
 
 \numberwithin{equation}{section}

\newcommand{\rd}{\color{red}}

\newcommand{\comment}[1]{}
\allowdisplaybreaks
\begin{document}

%
%
%
%
%
%
%
%
%

\title[Stokes and Darcy-Forchheimer-Brinkman {PDE} systems]{Integral potential method for a transmission problem\\ with Lipschitz interface in ${\mathbb R}^3$  for the Stokes and\\ Darcy-Forchheimer-Brinkman PDE systems}




\author[M. Kohr]{Mirela Kohr}
\address{%
Faculty of Mathematics and Computer Science, Babe\c{s}-Bolyai University,\\
1 M. Kog\u alniceanu Str.,
400084 Cluj-Napoca, Romania}

\email{mkohr@math.ubbcluj.ro}


\author[M. Lanza de Cristoforis]{Massimo Lanza de Cristoforis}
\address{Dipartimento di Matematica, Universit\`{a} degli Studi di Padova,\\
Via Trieste 63,
Padova 35121, Italy}

\email{mldc@math.unipd.it}



\author[S. E. Mikhailov]{Sergey E. Mikhailov}
\address{Department of Mathematics, Brunel University London,\\
             Uxbridge, UB8 3PH, United Kingdom}

\email{sergey.mikhailov@brunel.ac.uk}

\author[W. L. Wendland]{Wolfgang L. Wendland}
\address{Institut f\"ur Angewandte Analysis und Numerische Simulation, Universit\"at Stuttgart,\\
Pfaffenwaldring, 57,
70569 Stuttgart, Germany}

\email{wendland@mathematik.uni-stuttgart.de}



\date{}
\dedicatory{{\em Published in:} Z. Angew. Math. Phys., \rm (2016), {\bf 67}:116, 30 p. DOI: 10.1007/s00033-016-0696-1}

\begin{abstract}
The purpose of this paper is to obtain existence and uniqueness results in weighted Sobolev spaces for transmission problems for the {non-linear} Darcy-Forchheimer-Brinkman system and the {linear} Stokes system in two complementary Lipschitz domains in ${\mathbb R}^3$, one of them is a bounded Lipschitz domain $\Omega $ with connected boundary, and the other one is the exterior Lipschitz domain ${\mathbb R}^3\setminus \overline{\Omega }$. We exploit a layer potential method for the Stokes and Brinkman systems combined with a fixed point theorem in order to show the desired existence and uniqueness results, whenever the given data are suitably small in some weighted Sobolev spaces and boundary Sobolev spaces.
\end{abstract}

\maketitle

\section{Introduction}

Let $\alpha >0$ be a given constant. Then the following equations
\begin{equation}
\label{matrix-pseudo-diff1}
{\boldsymbol{\mathcal L}}_{\alpha }({\bf u},\pi ):=
\left(\triangle -\alpha {\mathbb I}\right){\bf u}-\nabla \pi ={\bf f},\ \
{\rm{div}}\ {\bf u}=0
\end{equation}
determine the {normalized} {\it Brinkman system}.
For $\alpha =0$, the {system} {\eqref{matrix-pseudo-diff1} gives the {normalized} {\it Stokes system},
\begin{equation}
\label{Stokes}
\begin{array}{lll}
{\boldsymbol{\mathcal L}}_{0}({\bf u},\pi ):=\triangle {\bf u} -\nabla \pi={\bf f},\ {\rm{div}}\ {\bf u}=0.
\end{array}
\end{equation}
{Both PDE systems \eqref{matrix-pseudo-diff1} and \eqref{Stokes} are linear.}

The {normalized} {\it Darcy-Forchheimer-Brinkman system}
\begin{equation}
\label{D-B-F-1}
\begin{array}{lll}
\triangle {\bf u}-\alpha {\bf u}-k|{\bf u}|{\bf u}-\beta ({\bf u}\cdot \nabla ){\bf u}-\nabla \pi ={\bf f},\ {\rm{div}}\ {\bf u}=0,
\end{array}
\end{equation}
{is nonlinear and} describes flows through porous media saturated with viscous incompressible fluids, where the inertia of such a fluid is not negligible. The constants $\alpha ,k,\beta >0$ are determined by the physical properties of such a porous medium (for further details we refer the reader to the book
\cite[p. 17]{Ni-Be} and the references therein).
When {$\alpha=k=0$}, \eqref{D-B-F-1} becomes the {normalized} {\it Navier-Stokes system}
\begin{equation}
\label{N-S}
\begin{array}{lll}
\triangle {\bf u} -{\beta}({\bf u}\cdot \nabla ){\bf u}-\nabla \pi={\bf f},\ {\rm{div}}\ {\bf u}=0.
\end{array}
\end{equation}
In this paper we will mainly deal with the first three systems.}

\label{Intro} \setcounter{equation}{0} The layer potential method
has a well known role in the study of elliptic boundary value problems (see, e.g.,
\cite{Ch-Mi-Na,Co,Fa-Ke-Ve,H-W,12,Med,Med1,M-M-Q,M-W,24,Varnhorn}).
Lang and Mendez \cite{Lang-Mendez} have used a layer potential technique to obtain optimal solvability results in weighted Sobolev spaces for the Poisson problem for the Laplace equation with Dirichlet or Neumann boundary conditions in an exterior Lipschitz domain. Escauriaza and Mitrea \cite{Es-Mi} have analyzed the transmission problems of the Laplace operator on Lipschitz domains in ${\mathbb R}^n$ by exploiting a layer potential method (see also \cite{E-F-V, E-K}). Mitrea and Wright \cite{M-W} have used the integral layer potentials in the analysis of the main boundary value problems for the Stokes system in arbitrary Lipschitz domains in ${\mathbb R}^n$ ($n\geq 2$). The authors in \cite{10-new2} have defined the notion of pseudodifferential Brinkman operator on compact Riemannian manifolds as a differentiable matrix type operator with variable coefficients that extends the notion of differential Brinkman operator from the Euclidean setting
to such manifolds, and have obtained well-posedness results for related transmission problems. The
Dirichlet problem for the Navier-Stokes system in bounded or exterior Lipschitz domains has been intensively studied, starting with the valuable work of Leray \cite{Leray} (see also
\cite{Galdi,Shor}). Dindo\u{s} and Mitrea \cite{D-M} have used the integral layer potentials to show well-posedness results in Sobolev and Besov spaces of Poisson problems for the Stokes and Navier-Stokes systems with Dirichlet condition on $C^1$ and Lipschitz domains in compact Riemannian manifolds (see also \cite{M-T}). Choe and Kim \cite{Choe-Kim} proved existence and uniqueness of the Dirichlet problem for the Navier-Stokes system on bounded Lipschitz domains in ${\mathbb R}^3$ with connected boundaries. Russo and Tartaglione \cite{Russo-Tartaglione-2} studied the Robin problem of the Stokes and Navier-Stokes systems in bounded or exterior Lipschitz domains, by exploiting a double-layer potential method (see also \cite{A-R,Choe-Kim}). Amrouche and Nguyen \cite{Amrouche-1} obtained existence and uniqueness results in weighted Sobolev spaces for the Poisson problem for the Navier-Stokes system in exterior Lipschitz domains in ${\mathbb R}^3$ with Dirichlet boundary condition, by using an approach based on a combination of properties of Oseen problems in ${\mathbb R}^3$ and in exterior domains of ${\mathbb R}^3$ (see also \cite{Amrouche-2,Am-Se}). Razafison \cite{Ra} obtained existence and uniqueness results for the three-dimensional exterior Dirichlet problem for the Navier-Stokes equations in anisotropic weighted $L^q$ spaces, $q\in (1,\infty )$. The authors in \cite{Ch-Mi-Na-3} studied direct segregated systems of boundary-domain integral equations for mixed (Dirichlet-Neumann) boundary value problems for a scalar second-order divergent elliptic partial differential equation with a variable coefficient in an exterior domain in ${\mathbb R}^3$. The boundary-domain integral equation system equivalence to the original boundary value problems and the Fredholm properties and invertibility of the corresponding boundary-domain integral operators have been analyzed in weighted Sobolev spaces. The authors in \cite{Ch-Mi-Na-1} used localized direct segregated boundary-domain integral equations for variable coefficient transmission problems with interface crack for scalar second order elliptic partial differential equations in a bounded composite domain consisting of adjacent anisotropic subdomains having a common interface surface. Segregated direct boundary-domain integral equation (BDIE) systems associated with mixed, Dirichlet and Neumann boundary value problems (BVPs) for a scalar partial differential equation with variable coefficients of the Laplace type for domains with interior cuts (cracks) have been investigated in \cite{Ch-Mi-Na-2}. The equivalence of BDIE's to such boundary value problems and the invertibility of the BDIE operators in the $L^2$-based Sobolev spaces have been also established. The author in \cite{Mikh-1,Mikh-2} used direct localized boundary-domain integro-differential formulations in the study of boundary-value problems in nonlinear elasticity and nonlinear boundary value problems with variable coefficients.

The authors in \cite{K-L-W} combined the integral layer potentials of the Stokes and Brinkman systems with a fixed point theorem to show the existence result for a nonlinear Neumann-transmission problem for the Stokes and Brinkman systems in two adjacent bounded Lipschitz domains with data in $L^p$ spaces, Sobolev spaces, and also in Besov spaces (see also \cite{K-L-W2,K-L-W3,K-L-W1,K-M-W4}). Dindo\u{s} and Mitrea \cite{D-M-new} used results in the linear theory for the Poisson problem of the Laplace operator in Sobolev and Besov spaces on Lipschitz domains and studied Dirichlet or Neumann problems for semilinear Poisson equations on Lipschitz domains in compact Riemannian manifolds.

The purpose of this paper is to obtain existence and uniqueness results in $L^2$-weighted Sobolev spaces for transmission problems for the Darcy-Forchheimer-Brinkman and Stokes systems
in two complementary Lipschitz domains in ${\mathbb R}^3$, one of them is a bounded Lipschitz domain $\Omega $ with connected boundary, and the other one is the complementary Lipschitz domain ${\mathbb R}^3\setminus \overline{\Omega }$. We exploit a layer potential method for the Stokes and Brinkman systems combined with a fixed point theorem in order to show the desired existence and uniqueness results, whenever the given data are suitably small in some $L^2$-based Sobolev spaces. The outline of the paper is the following. In the next section we introduce the $L^2$-weighted Sobolev spaces on an exterior Lipschitz domain in ${\mathbb R}^3$, where the existence and uniqueness results for the nonlinear transmission problem have been obtained. {In the third section we introduce the Newtonian and layer potential operators for the Stokes and Brinkman systems}. The fourth section presents a linear Poisson problem of transmission type for the Stokes and Brinkman systems in two complementary Lipschitz domains in ${\mathbb R}^3$ with data in weighted Sobolev spaces and boundary Sobolev spaces. Theorem \ref{well-posedness-Poisson-Besov-Dirichlet} is devoted to the well-posedness of such a problem and has been obtained by using an integral layer potential method. In the last section, we exploit the well-posedness result obtained for the linear problem together with a fixed point theorem in order to show the existence and uniqueness of the solution of a boundary value problem of transmission type for the Stokes and Darcy-Forchheimer-Brinkman systems in two complementary Lipschitz domains in ${\mathbb R}^3$ with data in weighted Sobolev spaces and boundary Sobolev spaces. In the Appendix we obtain various mapping properties of the Newtonian and layer potential operators for the Stokes and Brinkman systems in weighted or standard Sobolev spaces on exterior or bounded Lipschitz domains in ${\mathbb R}^3$.

Transmission problems coupling the Stokes and Darcy-Forchheimer-Brinkman systems appear as the mathematical model in various practical problems such as environmental problems with free air flow interacting with evaporation from soils (see, e.g., \cite{Ba-Mo-Fl,Ja-Ry-He-Gr}). An important application in medicine is the transvascular exchange between blood flow in vessels and the surrounding tissue as porous material and also in capillaries and tissue (see \cite{Ba}).

\section{Preliminary results}
\setcounter{equation}{0}

Let $\Omega _{+}\subset {\mathbb R}^3$ be a bounded Lipschitz domain, i.e., an open connected set whose boundary $\partial {\Omega }$ is locally the graph of a Lipschitz function. Let $\Omega _{-}:={\mathbb R}^3\setminus \overline{{\Omega }}_{+}$ denote the exterior Lipschitz domain.

\subsection{\bf Standard Sobolev spaces and related results}
Let $\Omega'$ be $\Omega_+$, $\Omega_-$ or $\mathbb R^3$. We denote by ${\mathcal E}(\Omega'):=C^{\infty }(\Omega')$ the space of infinitely differentiable functions and by ${\mathcal D}(\Omega '):=C^{\infty }_{{\rm{comp}}}(\Omega ')$ the space of infinitely differentiable functions with compact support in $\Omega '$, equipped with the inductive limit topology. Also, ${\mathcal E}'(\Omega ')$ and ${\mathcal D}'(\Omega ')$ denote the corresponding spaces of distributions on $\Omega$, i.e., the duals of ${\mathcal D}(\Omega ')$ and ${\mathcal E}(\Omega ')$, respectively.

{Let ${\mathcal F}$ and ${\mathcal F}^{-1}$ be the Fourier
transform and its inverse defined on the
$L^1(\mathbb R^3)$ functions as
$$
\hat g({\boldsymbol\xi})=[\mathcal F g]({\boldsymbol\xi}):=
\int_{\mathbb R^3}e^{-2\pi i \mathbf x\cdot\boldsymbol\xi}g({\mathbf x})d{\mathbf x},\quad
g({\mathbf x})=[\mathcal F^{-1} \hat g]({\bf x}):=
\int_{\mathbb R^3}e^{2\pi i \mathbf x\cdot\boldsymbol\xi}\hat g({\boldsymbol\xi})d{\boldsymbol\xi},
$$
and generalised to
the space of tempered distributions}. Note that $L^2({\mathbb R}^3)$ is the Lebesgue space of (equivalence classes of) measurable, square integrable functions on ${\mathbb R}^3$, and $L^{\infty }({\mathbb R}^3)$ is the space of (equivalence
classes of) essentially bounded measurable functions on ${\mathbb R}^3$.
For $s\in {\mathbb R}$, let us consider the $L^2$-based Sobolev (Bessel potential) spaces
\begin{align}
\label{bessel-potential}
&H^s({\mathbb R}^3):=\big\{(I-\triangle )^{-\frac{s}{2}}f: f\in L^2({\mathbb R}^3)\big\}
=\big\{{\mathcal F}^{-1}(1+|{\boldsymbol\xi} |^2)^{-\frac{s}{2}}{\mathcal F}f:f\in L^2({\mathbb R}^3)\big\},\\
&H^s({\mathbb R}^3)^3:=\{f=(f_1,f_2,f_3):f_j\in H^s({\mathbb R}^3),\ j=1,2,3\}, \label{bessel-potential2}\\
\label{spaces-Sobolev-inverse}
&H^s({\Omega }):=\{f\in {\mathcal D}'(\Omega ):\exists \ F\in H^s({\mathbb R}^3)
\mbox{ such that } F|_{\Omega }=f\}.
\end{align}
The space $\widetilde{H}^s({\Omega '})$ is the closure of ${\mathcal D}(\Omega ')$ in $H^1({\mathbb R}^3)$. This space can be also characterized as
\begin{align}
&\widetilde{H}^s({\Omega '}):=\{f\in H^s({\mathbb R}^3):{\rm{supp}}\
f\subseteq \overline{\Omega '}\}.
\end{align}
The vector-functions (distributions) belong to $H^s({\Omega '})^3$ and $\widetilde{H}^s({\Omega '})^3$ if their components belong to the corresponding scalar spaces {$H^s({\Omega '})$ and $\widetilde{H}^s({\Omega '})$}, as in \eqref{bessel-potential2} (see, e.g., \cite{Lean}). For ${s\ge0}$, the spaces (\ref{bessel-potential}) coincide with Sobolev-Slobodetskii spaces $W^{s,2}({\mathbb R}^3)$ (see, e.g., \cite[Chapter 4]{H-W}). In particular, such spaces are Sobolev spaces if $s$ is an {integer.
For} $s>-\frac{1}{2}$ such that $s-\frac{1}{2}$ is non-integer, the Sobolev
space $\widetilde{H}^s({\Omega '})$ can be identified with the closure $\mathring{H}^{s}(\Omega )$ of
${\mathcal D}(\Omega ')$ in the norm of ${H}^s({\Omega '})$ (see, e.g., \cite{Fa-Me-Mi}, \cite[Theorem 3.33]{Lean}). The spaces ${H}^s({\Omega '})$ and {$\mathring{H}^s({\Omega '})$ coincide for
$s\le\frac{1}{2}$, see \cite[Theorem 2.12]{Mikh}.} For any $s\in {\mathbb R}$, ${\mathcal D}(\overline{\Omega '})$ is dense in ${H}^s({\Omega '})$ and the following duality relations hold (see \cite[Proposition 2.9]{J-K1}, \cite[(1.9)]{Fa-Me-Mi}, \cite[(4.14)]{M-T2})
\begin{equation}
\label{duality-spaces} \left({H}^s({\Omega '})\right)'=\widetilde{H}^{-s}({\Omega '}),\ \ {H}^{-s}({\Omega '})=\big(\widetilde{H}^s({\Omega '})\big)'.
\end{equation}

For $s\in [0,1]$, the {Sobolev space ${H}^s(\partial\Omega )$ on the boundary $\partial\Omega$} can be defined by using the space
${H}^s({\mathbb R}^{2})$, a partition of unity and the pull-backs of the local parametrization of $\partial \Omega $. In addition, we note that ${H}^{-s}(\partial\Omega  )=\left({H}^s(\partial\Omega  )\right)'.$ All the above spaces are Hilbert spaces. For further properties of Sobolev spaces on bounded Lipschitz domains and Lipschitz boundaries, we refer to \cite{Gris,J-K1,Lean,M-T2,M-W,Triebel}.

\comment{We denote by $d\sigma $ the surface measure on $\partial\Omega  $, and by $\nu $ the outward unit normal to $\Omega_{+}$, which is defined $d\sigma$-a.e. on $\partial\Omega  $. Let $\kappa =\kappa (\partial\Omega  )>1$ be sufficiently large.
Then we consider the non-tangential approach regions in $\Omega _{\pm }$ with vertex at the point ${\bf x}\in \partial\Omega  $
\begin{equation}
\label{non-tangential} \Gamma_{\pm }({\bf x}):=\{{\bf y}\in
\Omega _{\pm }:{\rm{dist}}({\bf x},{\bf y})<\kappa \
{\rm{dist}}\ ({\bf y},\partial\Omega  )\}.
\end{equation}
Given $f:\Omega _{\pm }\to {\mathbb R}$, the non-tangential boundary trace $\gamma_{\pm }f$ of
$f$ is defined a.e. on ${\partial\Omega  }$ as
\begin{align}
\label{non-tangential-1-trace}
&(\gamma_{\pm }f)({\bf x}):=\lim _{\Gamma _{\pm}({\bf x})\ni {\bf y}\to {\bf x}}f({\bf y}),\ \mbox{ a.e. } {\bf x}\in {\partial\Omega  }.
\end{align}
Note that if $f\in C^{\infty }(\overline{\Omega }_{\pm})$, then $\gamma_{\pm }f$ is the usual restriction of $f$ to $\partial\Omega  $, i.e.,
\begin{align}
\label{trace-G-M}
&\gamma_{\pm }:C^{\infty }(\overline{\Omega }_{\pm})\to C^0({\partial\Omega  }),\ \ \gamma_{\pm }f=f|_{{_{\partial\Omega  }}}.
\end{align}
}
A useful result for the problems we are going to investigate is the following trace lemma (see \cite{Co}, \cite[Proposition 3.3]{J-K1}, \cite[Theorem 2.3, Lemma 2.6]{Mikh}, \cite{Mikh-3}, \cite[Theorem 2.5.2]{M-W}):
\begin{lem}
\label{trace-operator1} Assume that ${\Omega }:=\Omega _{+}\subset {\mathbb R}^3$
is a bounded Lipschitz domain with connected boundary $\partial \Omega $ and denote by $\Omega _{-}:={\mathbb R}^3\setminus \overline{\Omega }$ the corresponding exterior domain. Then there {exist linear and continuous trace operators}\footnote{The trace operator defined on Sobolev spaces of vector fields on the domain ${\Omega }_{\pm }$ is also denoted by $\gamma_{\pm }$.} {$\gamma_{\pm }:{H}^{1}({\Omega }_{\pm })\to H^{\frac{1}{2}}({\partial\Omega  })$ such that $\gamma_{\pm }f=f|_{{{\partial\Omega }}}$ for any $f\in C^{\infty }(\overline{\Omega }_{\pm })$.}
{These operators are surjective and have $($non-unique$)$ linear and continuous right inverse operators} $\gamma^{-1}_{\pm }:H^{\frac{1}{2}}({\partial\Omega })\to {H}^{1}({\Omega }_{\pm }).$
\end{lem}
{Note that the trace operator can be also defined through the non-tangential boundary trace as in, e.g., \cite[Section 2.3]{M-W}}.

\subsection{\bf Weighted Sobolev spaces and related results}\label{S2.2}
We now consider the weight function $\rho ({\bf x}):=\left(1+|{\bf x}|^2\right)^{\frac{1}{2}}$, ${\bf x}=(x_1,x_2,x_3)\in {\mathbb R}^3$.
{Then we define the weighted space {$L^2(\rho^{-1};{\Omega }_{-})$} as}
\begin{align}
\label{Lp-weight}
f\in {L^{2}(\rho^{-1};{\Omega }_{-})} \Longleftrightarrow {\rho }^{-1}f\in L^{2}(\Omega _{-}).
\end{align}
We also define the weighted Sobolev space
\begin{equation}
\label{weight-1}
{\mathcal H}^{1}(\Omega _{-}):=\left\{f\in {\mathcal D}'(\Omega _{-}):{\rho }^{-1}f\in L^2(\Omega _{-}),\ \nabla f\in L^2(\Omega _{-})^3\right\},
\end{equation}
which is a Hilbert space with the norm
\begin{equation}
\label{weight-2}
\|f\|_{{\mathcal H}^{1}(\Omega _{-})}:=\left(\left\|{\rho }^{-1}f\right\|_{L^2(\Omega _{-})}^2+\|\nabla f\|_{L^2(\Omega _{-})^3}^2\right)^{\frac{1}{2}}
\end{equation}
(cf. \cite{Ha}; see also \cite{Amrouche-1}, \cite{Do-Li}). In addition, we consider the space $\mathring{\mathcal H}^{1}(\Omega _{-})\subset {\mathcal H}^{1}(\Omega_-)$, which is the closure of ${\mathcal D}({\Omega  }_{-})$ in ${\mathcal H}^{1}(\Omega _{-})$, and the space
$\widetilde{\mathcal H}^{1}(\Omega _{-})\subset {\mathcal H}^{1}(\mathbb R^3)$,
which is the closure of ${\mathcal D}({\Omega  }_{-})$ in ${\mathcal H}^{1}(\mathbb R^3)$.
Note that $\mathring{\mathcal H}^{1}(\Omega _{-})$ coincides with the space $\big\{v\in {{\mathcal H}^{1}}(\Omega _{-}):\gamma_{-}v=0 \mbox{ on } \partial \Omega \big\}$ (see, e .g., \cite[p. 44]{Amrouche-1}), and 
\begin{equation}
\label{property}
\widetilde{\mathcal H}^{1}(\Omega _{-})=\{u\in {\mathcal H}^{1}(\mathbb R^3):{\rm{supp}}\ {u}\subseteq \overline{\Omega }_{-}\} 
\end{equation}
(cf., e.g. \cite[Theorem 3.29(ii)]{Lean}). Note that $\gamma_{-}$ is the trace operator defined in Lemma \ref{trace-operator1}.
Although, the domain of the functions of $\widetilde{\mathcal H}^{1}(\Omega _{-})$ is $\mathbb R^3$, while the domain of the functions from
$\mathring{\mathcal H}^{1}(\Omega _{-})$ is $\Omega _{-}$, we will often identify such spaces (cf., e.g., \cite[Theorem 3.33]{Lean}), when this does not lead to any confusion. Also the seminorm
\begin{align}
|g|_{{\mathcal H}^{1}(\Omega _{-})}:=\|\nabla g\|_{L^2(\Omega _{-})^3}
\nonumber
\end{align}
is equivalent to {the norm} (\ref{weight-2}) in ${\mathcal H}^{1}(\Omega _{-})$ (see, e.g., \cite{Do-Li}, \cite[(2.3)]{Ch-Mi-Na-3}). We also introduce the spaces
\begin{equation}
\label{weight-3}
{\mathcal H}^{-1}(\Omega _{-})=\Big(\widetilde{\mathcal H}^{1}(\Omega _{-})\Big)',\ \
\widetilde{\mathcal H}^{-1}(\Omega _{-})=\left({\mathcal H}^{1}(\Omega _{-})\right)'.
\end{equation}
When $\Omega _{-}={\mathbb R}^3$, we have $\widetilde{\mathcal H}^{-1}({\mathbb R}^3)={\mathcal H}^{-1}({\mathbb R}^3)$.

We note that ${\mathcal D}(\Omega _{-})$ is dense in $\mathring{\mathcal H}^{1}(\Omega _{-})$ by definition, and also in $\widetilde{\mathcal H}^{1}(\Omega _{-})$ (see, e.g., \cite{Ch-Mi-Na-3}).
The space $\mathcal D(\overline{\Omega}_-)$ is dense in $\mathcal H^1(\Omega_-)$ (see, e.g., \cite[p.136]{Am-Ra}). The weighted Sobolev spaces of vector-valued functions ${\mathcal H}^{\pm 1}(\Omega _{-})^3$ and $\widetilde{\mathcal H}^{\pm 1}(\Omega _{-})^3$ can be defined similarly. Moreover, due to the equivalence of the norm and seminorm in ${\mathcal H}^{1}(\Omega _{-})$, by the Sobolev inequality \cite[Theorem 4.31]{Adams2003} we have the embedding
\begin{equation}
\label{weight-4}
{\mathcal H}^{1}(\Omega _{-})\hookrightarrow L^6(\Omega _{-}).
\end{equation}

{We} can also define the {\it exterior trace operator on the weighted Sobolev space}
${\mathcal H}^{1}({\Omega }_{-})$,\\
$\gamma_{-}:{\mathcal H}^{1}({\Omega }_{-})\to H^{\frac{1}{2}}({\partial\Omega }).$ This operator is surjective (see \cite[Proposition 2.4]{Sa-Se}).
{Since $H^{1}(\Omega _{-})$ is continuously embedded in ${\mathcal H}^{1}(\Omega _{-})$, Lemma~\ref{trace-operator1} implies} the following {assertion} (cf. \cite[Lemma 5.3]{Lang-Mendez}, \cite[Theorem 2.3, Lemma 2.6]{Mikh}).
\begin{lem}
\label{trace-operator1-S}
Let ${\Omega }\subset {\mathbb R}^3$ be a bounded Lipschitz domain with connected boundary $\partial \Omega $. {Then there exists a $($non-unique$)$ linear and continuous right inverse $\gamma^{-1}_-: H^{\frac{1}{2}}({\partial\Omega })\to \mathcal{H}^{1}({\Omega }_-)$ to the trace operator $\gamma_{-}:{\mathcal H}^{1}({\Omega }_{-})\to H^{\frac{1}{2}}({\partial\Omega })$.}
\end{lem}
\comment{
\begin{proof}
In view of \cite[Theorem 2.3, Lemma 2.6]{Mikh} there exists a linear and continuous operator\\ ${\gamma }:H^{1}({\mathbb R}^3)\to H^{\frac{1}{2}}({\partial\Omega  })$ such that ${\gamma }f=f|_{{{\partial\Omega }}}$ for any $f\in C^{\infty }({\mathbb R}^3)$. In addition, there exists a linear bounded extension operator $Z:H^\frac{1}{2}({\partial\Omega })\to H^{1}({\mathbb R}^3)$, which is a right inverse to the trace operator ${\gamma }$, i.e., ${\gamma }(Z\phi )=\phi $ for any $\phi \in H^{\frac{1}{2}}({\partial\Omega })$.
Since the restriction operator is continuous from $H^1({\mathbb R}^3)$ to ${\mathcal H}^{1}(\Omega _{-})$ and $H^{1}(\Omega _{-})$ is continuously embedded into ${\mathcal H}^{1}(\Omega _{-})$, the
operator $Z:H^\frac{1}{2}({\partial\Omega })\to {\mathcal H}^{1}(\Omega _{-})$ is continuous as well, and the statement follows.
\end{proof}
}


The following definition specifies in which sense the conditions at infinity, associated with the transmission problems studied in this paper, are satisfied (cf. \cite[Definition 2.2, Remark 2.3]{Am-Ra}.
\begin{defn}
\label{behavior-infinity-sm}
A function $u$ {\it tends to a constant $u_{\infty }$ at infinity {in the sense of Leray}} if
\begin{align}
\label{behavior-infty-sm}
\lim_{r\to \infty }\int _{S^2}|u(r{\bf y})-u_{\infty }|d\sigma _{\bf y}=0.
\end{align}
\end{defn}
\begin{cor}
{If $u\in {\mathcal H}^{1}(\Omega _{-})$ then {$u$ tends to zero at infinity in the sense of Leray, i.e.,}}
\begin{align}
\label{behavior-infty-sm-ms}
\lim_{r\to \infty }\int _{S^2}|u(r{\bf y})|d\sigma _{\bf y}=0.
\end{align}
\end{cor}
\begin{proof}
Indeed, $u\in {\mathcal H}^{1}(\Omega _{-})$ implies that $\nabla u\in L^2(\Omega _{-})^3$. Then by Lemma \ref{behavior-infinity-s} there exists a constant $u_{\infty }\in {\mathbb R}$ such that $u-u_{\infty }\in L^6(\Omega _{-})$. Also, $u\in L^6(\Omega _{-})$ by embedding  (\ref{weight-4}). Therefore, the constant $u_{\infty }\in L^6(\Omega _{-})$, and accordingly $u_{\infty}=0$. Hence, (\ref{sm-sm}) implies the validity of (\ref{behavior-infty-sm-ms}).
\end{proof}

{Let $X$ be either an open subset or a surface in ${\mathbb R}^3$. Then, all along the paper, we use the
notation $\langle \cdot ,\cdot \rangle _X$ for the duality pairing of two dual Sobolev spaces defined on $X$.}

\subsection{\bf The Brinkman and Stokes systems in standard Sobolev spaces}

Let $\Omega '$ be $\Omega_+$, $\Omega_-$ or $\mathbb R^3$, $\alpha\in\mathbb R$ and let $({\bf u},\pi )\in {H}^1(\Omega ')^3\times L^2(\Omega ')$. Then the Brinkman (and Stokes) system is understood in the distributional sense as
\begin{equation}
\label{distributionalB}
\langle {\boldsymbol{\mathcal L}}_{\alpha }({\bf u},\pi ), {\bf w}\rangle_{\Omega'}=\langle {\bf f}, {\bf w}\rangle_{\Omega'},\quad
\langle {\rm div}\,{\bf u},g\rangle_{\Omega'}=0, \quad \forall\,({\bf w}, g)\in \mathcal D(\Omega')^3\times\mathcal D(\Omega'),
\end{equation}
where
\begin{equation}
\label{variationalB}
\langle {\boldsymbol{\mathcal L}}_{\alpha }({\bf u},\pi ), {\bf w}\rangle_{\Omega'}:=
\langle \triangle {\bf u}-\alpha{\bf u}-\nabla \pi ,{\bf w}\rangle_{\Omega'}=
-\langle \nabla {\bf u},\nabla {\bf w}\rangle _{\Omega'}-\langle \alpha{\bf u},{\bf w}\rangle _{\Omega'}
+\langle \pi ,{\rm{div}}\ {\bf w}\rangle _{\Omega'}.
\end{equation}
Since the space $\mathcal D(\Omega ')$ is dense in $\widetilde{H}^{1}(\Omega ')$ and in $L^2(\Omega ')$, and the bilinear form in the right hand side of \eqref{variationalB} is bounded on $({\bf u},\pi )\in {H}^1(\Omega ')^3\times L^2(\Omega ')$ and
$\mathbf w\in \widetilde{H}^{1}(\Omega ')^3$, expression \eqref{variationalB} defines a bounded linear operator
\begin{equation}\label{Lbound}
{\boldsymbol{\mathcal L}}_{\alpha }: {H}^1(\Omega ')^3\times L^2(\Omega ')\to {H}^{-1}(\Omega ')^3=\big(\widetilde{H}^{1}(\Omega ')^3\big)'.
\end{equation}
Note that the operator ${\rm div}:{H}^1(\Omega ')^3\to L^2(\Omega ')$ in the second equation in \eqref{distributionalB} is evidently bounded.

If $({\bf u},\pi)\in C^1(\overline{\Omega_{\pm}})^3\times C^0(\overline{\Omega_{\pm}})$, we define the interior and exterior conormal derivatives (i.e., {\it boundary tractions}) for the Brinkman {and Stokes systems}, ${\bf t}_{\alpha}^{\pm }({\bf u},\pi )$, {by} the well-known formula
\begin{align}
\label{2.37-}
{\bf t}_{\alpha}^{\pm}({\bf u},\pi ):={\gamma _\pm}\left(-\pi{\mathbb I}+2{\mathbb E}({\bf u})\right)\nu ,
\end{align}
where ${\mathbb E}({\bf u})$ is the symmetric part of $\nabla {\bf u}$, and $\nu{=\nu^+}$ is the outward unit normal to $\Omega_{ +}$, defined a.e. on $\partial {\Omega }$. Then for $\alpha\in\mathbb R$ and any function $\boldsymbol{\varphi} \in {\mathcal D}({\mathbb R}^3)^3$ we obtain the the first Green identity
\begin{align}
\label{special-case-1}
{\pm}\left\langle {\bf t}_{\alpha}^{\pm}({\bf u},\pi ),
\boldsymbol{\varphi} \right\rangle _{_{\!\partial\Omega  }}= &2\langle {\mathbb E}({\bf
u}),{\mathbb E}(\boldsymbol{\varphi} )\rangle _{\Omega_{\pm}}+\alpha \langle {\bf
u},\boldsymbol{\varphi} \rangle _{\Omega_{\pm}} -\langle \pi,{\rm{div}}\ \boldsymbol{\varphi} \rangle _{\Omega_{\pm}}+\left\langle {\boldsymbol{\mathcal L}}_{\alpha }({\bf u},\pi ),\boldsymbol{\varphi}\right\rangle _{{\Omega_{\pm}}}.
\end{align}
This formula suggests the following {weak} definition of the conormal derivative in the setting of $L^2$-based Sobolev spaces (cf. \cite[Lemma 3.2]{Co}, \cite[Lemma 2.2]{K-L-W1}, \cite[Lemma 2.2]{K-M-W4}, \cite[Definition 3.1, Theorem 3.2]{Mikh}, \cite[Proposition 3.6]{M-M-W}, \cite[Theorem 10.4.1]{M-W}, \cite{Mikh-3}).
\begin{lem}
\label{conormal-derivative-generalized-Robin} Let $\Omega_{ +}\subset {\mathbb R}^3$ be a bounded Lipschitz domain {and ${\Omega_-}= {\mathbb R}^3\setminus\overline{\Omega_+}$}. Let $\alpha \geq 0$ and
\begin{multline}
\label{space-Brinkman-generalized-Robin}
{\pmb{H}}^{1}({\Omega_{\pm}},{\boldsymbol{\mathcal L}}_{\alpha }):=\Big\{({\bf u}_{\pm},\pi_{\pm} ,{\tilde{\bf f}}_{\pm})\in
H^{1}({\Omega_{\pm}})^3\times
L^{2}({\Omega_{\pm}})\times
\widetilde{H}^{-1}({\Omega_{\pm}})^3:\\
{\boldsymbol{\mathcal L}}_{\alpha }({\bf u}_{\pm},\pi_{\pm} )={\tilde{\bf f}}_{\pm}|_{\Omega_{\pm}} \mbox{ and } {\rm{div}}\ {\bf u}_{\pm}=0 \mbox{ in } {\Omega_{\pm}}\Big\}.
\end{multline}
Then the conormal derivative {operators} ${\bf t}_{\alpha}^{\pm}:{\pmb{H}}^{1}({\Omega_{\pm}},{\boldsymbol{\mathcal L}}_{\alpha })\to H^{-\frac{1}{2}}(\partial\Omega )^3,$ defined by
\begin{align}
\label{conormal-generalized-1-Robin-new1}
&{\pmb{H}}^{1}({\Omega_{\pm}},{\boldsymbol{\mathcal L}}_{\alpha
})\ni ({\bf u}_{\pm},\pi_{\pm} ,{\tilde{\bf f}}_{\pm})\longmapsto {\bf t}_{\alpha}^{\pm}({\bf u}_{\pm},\pi_{\pm} ;{\tilde{\bf f}}_{\pm})\in H^{-\frac{1}{2}}(\partial\Omega )^3,\\
\label{conormal-generalized-1-Robin} &{\pm}\big\langle {\bf t}_{\alpha}^{\pm}({\bf u}_{\pm},\pi_{\pm} ;{\tilde{\bf f}}_{\pm}),{\boldsymbol{\Phi} }\big\rangle
_{_{\!\partial\Omega  }}\!\!\!:=2\langle {\mathbb E}({\bf u}_{\pm}),{\mathbb E}(\gamma^{-1}_{\pm}{\boldsymbol{\Phi}})\rangle _{\Omega_{\pm}}+\alpha \langle {\bf u}_{\pm},\gamma^{-1}_{\pm}{\boldsymbol{\Phi}}\rangle _{\Omega_{\pm}}\nonumber\\
&\quad \quad \quad \quad \quad \quad \quad \quad \quad \quad \quad -\langle \pi_{\pm},{\rm{div}}(\gamma^{-1}_{\pm}{\boldsymbol{\Phi}})\rangle _{\Omega_{\pm}}+\langle {\tilde{\bf f}}_{\pm},\gamma^{-1}_{\pm}{\boldsymbol{\Phi}}\rangle _{{\Omega_{\pm}}},\ \forall \ \boldsymbol{\Phi} \in H^{\frac{1}{2}}(\partial\Omega )^3
\end{align}
{are} bounded and do not depend on the choice of the right inverse $\gamma^{-1}_{\pm}$ of the trace operator\\ $\gamma_{\pm}:H^{1}({\Omega_{\pm}})^3\to H^{\frac{1}{2}}(\partial\Omega )^3$.
In addition, for all
$({\bf u}_{\pm},\pi_{\pm} ,{\tilde{\bf f}}_{\pm})\in {\pmb{H}}^{1}({\Omega_{\pm}},{\boldsymbol{\mathcal L}}_{\alpha })$
and for any ${{\bf w}_\pm}\in H^{1}({\Omega_{\pm}})^3$, the following Green {identities hold}:
\begin{multline}
\label{conormal-generalized-2}
\!\!\!\!\!{\pm}\left\langle {\bf t}_{\alpha}^{\pm}({\bf u}_{\pm},\pi_{\pm};{\tilde{\bf f}}_{\pm}),\gamma_{\pm}{{\bf w}_\pm}\right\rangle _{_{\!\partial\Omega  }}\!=\!
2\langle {\mathbb E}({\bf u}_{\pm}),{\mathbb E}({{\bf w}_\pm})\rangle _{\Omega_{\pm}}
\!+\!\alpha \langle {\bf u}_{\pm},{{\bf w}_\pm}\rangle _{\Omega_{\pm}}\!
-\!\langle \pi_{\pm},{\rm{div}}\ {{\bf w}_\pm}\rangle _{\Omega_{\pm}}
+\langle {\tilde{\bf f}}_{\pm},{{\bf w}_\pm}\rangle _{{\Omega_{\pm}}}.
\end{multline}
\end{lem}

\begin{rem}
\label{CSD}
By exploiting arguments similar to those of the proof of Theorem 3.10 and the paragraph following it in \cite{Mikh}, one can see that the weak conormal derivative on $\partial\Omega$ can be equivalently defined by the dual form like {\eqref{conormal-generalized-1-Robin}} but only on a Lipschitz subset $\Omega'_{\pm}\subset\Omega_{\pm}$ such that $\partial \Omega \subset \partial \Omega _{\pm }'$ and $\Omega _{\pm }'':=\Omega _{\pm }\setminus \overline{\Omega _{\pm }'}$ has closure equal to $\Omega \setminus \Omega _{\pm }'$ (i.e., on $\Omega'_{\pm}$ laying near the boundary $\partial\Omega$) as
\begin{equation}\label{RCN}
{\bf t}_{\alpha}^{\pm}({\bf u}_{\pm},\pi_{\pm} ;{\tilde{\bf f}}'_{\pm})
=r_{_{\partial\Omega}}{\bf t}_{\alpha}^{\prime\pm}({\bf u}_{\pm},\pi_{\pm} ;{\tilde{\bf f}}'_{\pm}),
\end{equation}
where $\tilde{\mathbf f}'_\pm\in\widetilde{H}^{-1}(\Omega'_\pm)^3$ is such that
$r_{\mathbb R^3\setminus\overline{\Omega''_\pm}}\tilde{\mathbf f}'_\pm
=r_{\mathbb R^3\setminus\overline{\Omega''_\pm}}\tilde{\mathbf f}_\pm$, $\overline{\Omega''_\pm}:=\Omega_\pm\backslash\Omega'_\pm$, and
${\bf t}_{\alpha}^{\prime\pm}:{\pmb{H}}^{1}({\Omega'_{\pm} },{\boldsymbol{\mathcal L}}_{\alpha})\to H^{-\frac{1}{2}}(\partial\Omega'_{\pm} )^3,$
is the continuous operator defined as
\begin{multline}
\label{conormal-SD}
\!\!\!\!\!\!\!\big\langle {\bf t}_{\alpha}^{\prime\pm}({\bf u}_{\pm},\pi_{\pm} ;{\tilde{\bf f}'}_{\pm}),
{\boldsymbol{\Phi} }\big\rangle_{_{\!\partial\Omega'_{\pm}  }}\!\!\!
:=2\langle {\mathbb E}({\bf u}_{\pm}),{\mathbb E}(\gamma^{-1}_{\pm}{\boldsymbol{\Phi}})\rangle _{\Omega'_{\pm}}
+\alpha \langle {\bf u}_{\pm},\gamma^{-1}_{\pm}{\boldsymbol{\Phi}}\rangle _{\Omega'_{\pm} }\\
-\langle \pi_{\pm},{\rm{div}}(\gamma^{-1}_{\pm}{\boldsymbol{\Phi}})\rangle _{\Omega'_{\pm} }
+\langle {\tilde{\bf f}'}_{\pm},\gamma^{-1}_{\pm}{\boldsymbol{\Phi}}\rangle _{{\Omega'_{\pm} }},\quad \forall \ \boldsymbol{\Phi} \in H^{\frac{1}{2}}(\partial\Omega'_{\pm})^3.
\end{multline}
Moreover, definition \eqref{RCN}-\eqref{conormal-SD} is well applicable to the class of functions
$({\bf u}_{\pm},\pi_{\pm} ;{\tilde{\bf f}'}_{\pm})\in {\pmb{H}}^{1}({\Omega'_{\pm} },{\boldsymbol{\mathcal L}}_{\alpha})$ that are not obliged to belong to ${\pmb{H}}^{1}({\Omega_{\pm} },{\boldsymbol{\mathcal L}}_{\alpha})$.
{This} is particularly useful for the functions ${\bf u}_-$ tending to a non-zero constant at infinity and which accordingly do not belong to { ${H}^{1}({\Omega_- })^3$}.
\end{rem}
\begin{rem}
\label{2.3}
Note that the conormal derivative operator is linear in the sense that
\begin{align}
\label{2.40a}
c_1{\bf t}_{\alpha}^{\pm}({\bf u}_1,\pi _1;{\tilde{\bf f}}_1)+c_2{\bf t}_{\alpha}^{\pm}({\bf u}_2,\pi _2;{\tilde{\bf f}}_2)={\bf t}_{\alpha}^{\pm}(c_1{\bf u}_1+c_2{\bf u}_2,c_1\pi _1+c_2\pi _2;c_1{\tilde{\bf f}}_1+c_2{\tilde{\bf f}}_2)
\end{align}
if $({\bf u}_i,\pi _i;{\tilde{\bf f}}_i)\in {\pmb{H}}^{1}({\Omega_\pm},{\boldsymbol{\mathcal L}}_{\alpha })$, $i=1,2$,
and $c_1,c_2\in {\mathbb R}$.
\end{rem}
\begin{rem}\label{Rm2.4}
If ${\tilde{\bf f}}\in \widetilde H^{-\frac{1}{2}}(\Omega_\pm)^3\subset{\widetilde H^{-1}(\Omega_\pm)^3}$, then ${\tilde{\bf f}}$ {is} {\em uniquely} determined by $({\bf u}_{\pm},\pi_{\pm} )$ as the {\em unique} extension of ${\boldsymbol{\mathcal L}}_{\alpha }({\bf u}_{\pm},\pi_{\pm} )\in H^{-1/2}(\Omega_\pm)^3$ to ${\widetilde H^{-\frac{1}{2}}({\Omega_\pm})^3}$. In this case, ${\bf t}_{\alpha}^{\pm}({\bf u},\pi;{\tilde{\bf f}})$ becomes the {\em canonical} conormal derivative (cf. \cite{Mikh}, \cite{Mikh-3}), which we denote as
${\bf t}_{\alpha}^{\pm}({\bf u},\pi)$. One can show that this notation is consistent with the classical definition of the conormal derivative if $({\bf u},\pi )\in C^1(\overline{\Omega_\pm })^3\times C^0(\overline{\Omega_\pm })$ $($cf. \cite[Theorem 3.16]{Mikh}$)$. We will particularly use this property for ${\tilde{\bf f}}=\mathbf 0$.

Note also that by Remark~\ref{CSD}, for the conormal derivative to be canonical, it is sufficient that ${\tilde{\bf f}}\in \widetilde H^{-1/2}_{\rm loc}(\Omega'_\pm)^3$ in a domain $\Omega'_\pm\subset\Omega_\pm$ laying near the boundary $\partial \Omega $ in the sense of Remark \ref{CSD}.
\end{rem}

\subsection{\bf The Stokes system in weighted Sobolev spaces in exterior domain}

{For the exterior domain $\Omega_-$, we need to consider the Stokes system in weighted Sobolev spaces.

Let  $({\bf u},\pi )\in {\mathcal H}^1({\Omega_-})^3\times L^2({\Omega_-})$. Then taking $\alpha=0$ in \eqref{distributionalB} we obtain the Stokes system  in the distribution sense,
\begin{equation}
\label{distributionalS}
\langle {\boldsymbol{\mathcal L}}_{0}({\bf u},\pi ), {\bf w}\rangle_{\Omega_-}=\langle {\tilde{\bf f}}, {\bf w}\rangle_{\Omega_-},\quad
\langle {\rm div}\,{\bf u},g\rangle_{\Omega_-}=0, \quad \forall\,({\bf w}, g)\in \mathcal D({\Omega_-})^3\times\mathcal D({\Omega_-}),
\end{equation}
where
\begin{equation}
\label{variational}
\langle {\boldsymbol{\mathcal L}}_{0}({\bf u},\pi ), {\bf w}\rangle_{\Omega_-}:=
\langle \triangle {\bf u}-\nabla \pi ,{\bf w}\rangle_{\Omega_-}=
-\langle \nabla {\bf u},\nabla {\bf w}\rangle _{\Omega_-}
+\langle \pi ,{\rm{div}}\ {\bf w}\rangle _{\Omega_-}.
\end{equation}
Since the space $\mathcal D({\Omega_-})^3$ is dense in $\widetilde{\mathcal H}^{1}({\Omega_-})^3$ and the bilinear form in the right hand side of \eqref{variational} is bounded on $({\bf u},\pi )\in {\mathcal H}^1({\Omega_-})^3\times L^2({\Omega_-})$ and
$\mathbf w\in \widetilde{\mathcal H}^{1}({\Omega_-})^3$, expression \eqref{variational} defines a bounded linear operator
$${\boldsymbol{\mathcal L}}_{0}: {\mathcal H}^1({\Omega_-})^3\times L^2({\Omega_-})\to {\mathcal H}^{-1}({\Omega_-})^3=\big(\widetilde{\mathcal H}^{1}({\Omega_-})^3\big)'$$
in weighted Sobolev spaces in exterior domains, in addition to \eqref{Lbound} also valid for $\alpha=0$ in standard Sobolev spaces.
Note that the operator ${\rm div}:  {\mathcal H}^1({\Omega_-})^3\to L^2({\Omega_-})$ in the second equation in \eqref{distributionalS} is bounded as well due to the definition of the space ${\mathcal H}^1({\Omega_-})^3$.}

We now show the following version of Lemma \ref{conormal-derivative-generalized-Robin} for the Stokes system in weighted Sobolev spaces on exterior Lipschitz domains.
\begin{lem}
\label{conormal-derivative-generalized-Robin-S}
Let ${\Omega_+ }\subset {\mathbb R}^3$ be a bounded Lipschitz domain with connected boundary $\partial \Omega $. Let $\Omega _{-}:={\mathbb R}^3\setminus \overline{\Omega_+ }$ and
\begin{align}
\label{space-Brinkman-generalized-Robin-S}
\pmb{\mathcal H}^{1}({\Omega }_{-},{\boldsymbol{\mathcal L}}_{0}):=\Big\{({\bf u},\pi ,{\tilde{\bf f}})\in &{\mathcal H}^{1}({\Omega }_{-})^3\times L^{2}({\Omega }_{-})\times
\widetilde{\mathcal H}^{-1}({\Omega }_{-})^3:\nonumber\\
&{\boldsymbol{\mathcal L}}_{0}({\bf u},\pi ):=\triangle {\bf u}-\nabla \pi ={\tilde{\bf f}}|_{\Omega _{-}} \mbox{ and } {\rm{div}}\ {\bf u}=0 \mbox{ in } {\Omega }_{-}\Big\}.
\end{align}
Then the conormal derivative operator ${\bf t}_0^{-}:{\pmb{\mathcal H}}^{1}({\Omega _{-}},{\boldsymbol{\mathcal L}}_{0})\to H^{-\frac{1}{2}}(\partial\Omega )^3$,
\begin{align}
\label{conormal-generalized-1-Robin-new1-S}
&{\pmb{\mathcal H}}^{1}({\Omega }_{-},{\boldsymbol{\mathcal L}}_{0})\ni ({\bf u},\pi ,{\tilde{\bf f}})\longmapsto {\bf t}_0^{-}({\bf u},\pi ;{\tilde{\bf f}})\in H^{-\frac{1}{2}}(\partial\Omega )^3,
\\
\label{conormal-generalized-1-Robin-S}
&\left\langle {\bf t}_0^{-}({\bf u},\pi ;{\tilde{\bf f}}),{\boldsymbol\phi }\right\rangle_{_{\!\partial\Omega  }}\!\!\!:=\!-2\langle {\mathbb E}({\bf u}),{\mathbb E}(\gamma^{-1}_- \phi)\rangle _{\Omega _{-}}\!\!+\!\!\langle \pi,{\rm{div}}(\gamma^{-1}_-\phi)\rangle _{\Omega _{-}}
\!\!-\!\!\langle {\tilde{\bf f}},\gamma^{-1}_-\phi\rangle _{{\Omega _{-}}},\, \forall \ \boldsymbol\phi \in H^{\frac{1}{2}}(\partial\Omega )^3,
\end{align}
is linear and bounded, and does not depend on the choice of the right inverse $\gamma^{-1}_-$
of the trace operator $\gamma_{-}:{\mathcal H}^{1}({\Omega }_{-})^3\to H^{\frac{1}{2}}(\partial\Omega )^3$$)$. In addition, the following Green formula holds:
\begin{align}
\label{conormal-generalized-2-Sigma-S}
{\left\langle {\bf t}_0^{-}({\bf u},\pi ;{\tilde{\bf f}}),\gamma_{-}{\bf w}\right\rangle _{_{\partial\Omega  }}\!\!\!=\!-\!2\langle {\mathbb E}({\bf u}),{\mathbb E}({\bf w})\rangle _{\Omega _{-}}+\langle \pi,{\rm{div}}\ {\bf w}\rangle _{\Omega _{-}}\!\!-\!\langle {\tilde{\bf f}},{\bf w}\rangle _{{\Omega _{-}}}}
\end{align}
for all $({\bf u},\pi ,{\tilde{\bf f}})\in {\pmb{\mathcal H}}^{1}({\Omega }_{-},{\boldsymbol{\mathcal L}}_{0})$ and ${\bf w}\in {\mathcal H}^{1}({\Omega }_{-})^3$.
\end{lem}
\begin{proof}
Note that $\widetilde{\mathcal H}^{-1}({\Omega }_{-})=\big({\mathcal H}^{1}({\Omega }_{-})\big)'$, and that accordingly, the last dual pairing in the right-hand side of (\ref{conormal-generalized-1-Robin-S}) is well-defined. All other dual products in (\ref{conormal-generalized-1-Robin-S}) are also well-defined. Therefore, {operator} (\ref{conormal-generalized-1-Robin-new1-S}) is well-defined. By Lemma \ref{trace-operator1-S}, the operator $\gamma^{-1}_- $ involved in the right-hand side {of formula} (\ref{conormal-generalized-1-Robin-S}) is bounded. All other involved operators are bounded as well, and, hence, {operator} (\ref{conormal-generalized-1-Robin-new1-S}) is bounded. Moreover, the independence of (\ref{conormal-generalized-1-Robin-new1-S}) of the choice of {the operator $\gamma^{-1}_-$, as in Lemma $\ref{trace-operator1-S}$,} follows by arguments similar to those for \cite[Theorem 3.2]{Mikh}.

{Further}, by (\ref{conormal-generalized-1-Robin-S}), and by continuity of the trace operator from ${\mathcal H}^{1}({\Omega }_{-})^3$ to $H^{\frac{1}{2}}(\partial\Omega )^3$, we {have},
\begin{align}
&\left\langle {\bf t}_0^{-}({\bf u},\pi ;{\tilde{\bf f}}),\gamma_{-}{\bf w}\right\rangle
_{_{\!\partial\Omega  }}=
-2\left\langle {\mathbb E}({\bf u}),{\mathbb E}(\gamma^{-1}_-\gamma_{-}{\bf w})\right\rangle _{\Omega _{-}}+\left\langle \pi,{\rm{div}}(\gamma^{-1}_-\gamma_{-}{\bf w})\right\rangle _{\Omega _{-}}
-\left\langle {\tilde{\bf f}},\gamma^{-1}_-\gamma_{-}{\bf w}\right\rangle _{{\Omega _{-}}}\nonumber\\
&=-2\left\langle {\mathbb E}({\bf u}),{\mathbb E}({\bf w})\right\rangle _{\Omega _{-}}
+\left\langle \pi,{\rm{div}\,}{\bf w}\right\rangle _{\Omega _{-}}
-\left\langle {\tilde{\bf f}},{\bf w}\right\rangle _{{\Omega _{-}}}
-2\left\langle {\mathbb E}({\bf u}),{\mathbb E}(\gamma^{-1}_-\gamma_{-}{\bf w}-{\bf w})\right\rangle _{\Omega _{-}}\nonumber\\
&\hspace{1em}+\left\langle \pi,{\rm{div}}(\gamma^{-1}_-\gamma_{-}{\bf w}-{\bf w})\right\rangle _{\Omega _{-}}-\left\langle {\tilde{\bf f}},\gamma^{-1}_-\gamma_{-}{\bf w}-{\bf w}\right\rangle _{{\Omega _{-}}},
\ \forall \ {\bf w}\in {\mathcal H}^{1}({\Omega }_{-})^3.\label{2.34}
\end{align}
Since $\gamma_-\big(\gamma^{-1}_-\gamma_{-}{\bf w}-{\bf w}\big)={\bf 0}$ on $\partial \Omega $, {we obtain that
$\gamma^{-1}_-\gamma_{-}{\bf w}-{\bf w}\in \mathring{\mathcal H}^{1}({\Omega }_{-})^3,$
where
\begin{equation}\label{ringH}
\mathring{\mathcal H}^{1}(\Omega _{-}):=\overline{{\mathcal D}(\Omega _{-})}^{{\mathcal H}^{1}(\Omega _{-})}.
\end{equation}
$\mathring{\mathcal H}^{1}(\Omega _{-})$ can be characterized as the subspace of ${\mathcal H}^{1}(\Omega _{-})$ with null traces, i.e.,
\begin{align}
\label{null-trace}
\mathring{\mathcal H}^{1}(\Omega _{-})=\left\{{\bf v}_{0}\in {{\mathcal H}^{1}}(\Omega _{-}):\gamma_{-}{{\bf v}}_0={0} \mbox{ on } \partial \Omega \right\}
\end{align}
(see, e.g., \cite[p. 44]{Amrouche-1}, \cite[Theorem 3.29(ii),Theorem 3.33]{Lean}). Then Green's identity (\ref{conormal-generalized-2-Sigma-S}) follows from \eqref{2.34} provided that we can show that
\begin{align}
\label{add-Green}
-2\langle {\mathbb E}({\bf u}),{\mathbb E}({\bf v})\rangle _{\Omega _{-}}+\langle \pi,{\rm{div}}\ {\bf v}\rangle _{\Omega _{-}}\!\!-\!\langle {\tilde{\bf f}},{\bf v}\rangle _{{\Omega _{-}}}=0
\end{align}
for any ${\bf v}\in \mathring{\mathcal H}^{1}({\Omega }_{-})^3.$
By definition \eqref{ringH},} ${\mathcal D}(\Omega _{-})$ is dense in $\mathring{\mathcal H}^{1}(\Omega _{-})$ (cf., e.g., \cite[p. 44]{Amrouche-1}, \cite[Lemma 2.3]{Lang-Mendez}). Therefore, it suffices to show {identity} (\ref{add-Green}) only for functions ${\bf v}\in {\mathcal D}(\Omega _{-})^3$. However, in this case, {formula} (\ref{add-Green}) follows by the membership of $({\bf u},\pi ,{\tilde{\bf f}})$ in the space {${\pmb{ H}}^{1}({\Omega' },{\boldsymbol{\mathcal L}}_{0})$, where ${\Omega'}$ is a bounded Lipschitz domain such that ${\rm supp}({\bf v})\subset \Omega '\subset \Omega _{-}$, and by the Green identity \eqref{conormal-generalized-2}}.
\end{proof}

\section{{Newtonian and layer potential operators for the Stokes and Brinkman systems}}

Let $\alpha >0$ and $\left({\mathcal G}^{\alpha }(\cdot ,\cdot ),{\Pi }^{\alpha }(\cdot ,\cdot )\right)\in {\mathcal D}'({\mathbb R}^3\times {\mathbb R}^3)^{3\times 3}\times {\mathcal D}'({\mathbb R}^3\times {\mathbb R}^3)^3$ be the fundamental solution of the Brinkman system in ${\mathbb R}^3$. Hence,
\begin{equation}
\label{Brinkman-operator4}
(\triangle _{\bf x}-\alpha {\mathbb I}){\mathcal G}^{\alpha }({\bf x},{\bf y}) -\nabla_{\bf x}{\Pi
}^{\alpha }({\bf x},{\bf y}) =-\delta_{\bf y}({\bf x}){\mathbb I},\ \
{\rm{div}}_{\bf x}{\mathcal G}^{\alpha }({\bf x},{\bf y})=0,
\end{equation}
where $\delta _{\bf y}$ is the Dirac distribution with mass at ${\bf y}$. Note that the subscript ${\bf x}$ added to a differential operator indicates that we are differentiating with respect to ${\bf x}$. Let ${\mathcal G}_{ij}^{\alpha }(\cdot ,\cdot )$ be the components of the fundamental tensor ${\mathcal G}^{\alpha }(\cdot ,\cdot )$. Let $\Pi _{j}^{\alpha }(\cdot ,\cdot )$ be the components of the fundamental vector ${\Pi }^{\alpha }(\cdot ,\cdot )$. Then
\begin{equation}
\label{5.4.21}
\begin{array}{c}
{\mathcal G}_{jk}^{\alpha }({\bf x})=
\displaystyle\frac{1}{8\pi }\left\{\displaystyle\frac{\delta_{jk}}{|{\bf x}|}A_1\left(\alpha ^{\frac{1}{2}}|{\bf x}|\right )+\displaystyle\frac{x_j x_k}{|{\bf x}|^3}A_2\left(\alpha ^{\frac{1}{2}}|{\bf x}|\right )\right\},\
\Pi _j^{\alpha }({\bf x})={\Pi_j({\bf x})}=\displaystyle\frac{1}{4\pi }\displaystyle\frac{x_j}{|{\bf x}|^3},
\end{array}
\end{equation}
where
\begin{equation}
\label{5.4.22}
A_1(z)=2e^{-z}(1+z^{-1}+z^{-2})-2z^{-2},\
A_2(z)=-2e^{-z}(1+3z^{-1}+3z^{-2})+6z^{-2}
\end{equation}
(cf.. e.g., \cite{McCracken1981}, \cite[Chapter 2]{24} and \cite[Section 3.2.1]{12}).

In addition, the components of the stress tensor ${\bf S}^{\alpha }(\cdot ,\cdot )$ are given in view of (\ref{5.4.21}) and (\ref{5.4.22}) by
\begin{align}
\label{fundamental-solution-Brinkman2-new}
S^{\alpha }_{jk\ell }({\bf x})&:=-{\Pi} _k({\bf x})\delta _{j\ell } +\frac{\partial {\mathcal G}^{\alpha }_{jk}({\bf x})}{\partial x_{_{\!\ell }}}+\frac{\partial {\mathcal G}^{\alpha }_{\ell k}({\bf x})}{\partial x_j}\nonumber\\
&=-\displaystyle\frac{1}{4\pi }\left\{\delta_{k\ell }
\frac{x_j}{|{\bf x}|^3}D_1\left(\alpha ^{\frac{1}{2}}|{\bf x}|\right )\!+\!
\left(\delta_{jk}\frac{x_\ell }{|{\bf x}|^3}\!+\!\delta_{j\ell }\frac{x_k}{|{\bf x}|^3}\right)D_2\left(\alpha ^{\frac{1}{2}}|{\bf x}|\right )\!+\!\frac{x_j x_k x_\ell }{|{\bf x}|^5}D_3(\alpha ^{\frac{1}{2}}|{\bf x}|)\right\},
\end{align}
where
\begin{align*}
D_1(z)&=2e^{-z}(1+3z^{-1}+3z^{-2})-6z^{-2}+1,\quad
D_2(z)=e^{-z}(z+3+6z^{-1}+6z^{-2})-6z^{-2},\\
D_3(z)&=e^{-z}(-2z-12-30z^{-1}-30z^{-2})+30z^{-2}
\end{align*}
(see also, e.g., \cite{McCracken1981}, \cite[Chapter 2]{24}, \cite{Varnhorn} and \cite[Section 3.2.1]{12}). Let ${\Lambda }^{\alpha }$ denote the fundamental pressure tensor {whose components ${\Lambda }^{\alpha }_{jk}$ are given by
\begin{equation}
\label{pressure-tensor-B}
{\Lambda }^{\alpha }_{jk}({\bf x})=\displaystyle\frac{1}{4\pi }\left\{\frac{\delta_{jk}}{|{\bf x}|^3}\left(\alpha |{\bf x}|^2-2\right ) +6\frac{x_jx_k}{|{\bf x}|^5}\right\}
\end{equation}
(see \cite[(3.6.15)]{12}).} Note that, for ${\bf x}\neq {\bf y}$,
\begin{equation}
\label{Brinkman-operator6}
\triangle _{\bf x}S_{jk\ell}^{\alpha
}({\bf y},{\bf x}) -\alpha S_{jk\ell}^{\alpha }({\bf y},{\bf
x})-\frac{\partial \Lambda _{j\ell }^{\alpha }({\bf x},{\bf
y})}{\partial x_k}=0,\ \frac{\partial S_{jk\ell}^{\alpha }({\bf
y},{\bf x})}{\partial x_k}=0.
\end{equation}



For $\alpha =0$ we obtain the fundamental solution of the Stokes
system. {Henceforth}, we use the notation $\left({\mathcal G}(\cdot ,\cdot
),{\Pi }(\cdot ,\cdot )\right)\in {\mathcal D}'({\mathbb R}^3\times {\mathbb R}^3)^{3\times 3}\times {\mathcal D}'({\mathbb R}^3\times {\mathbb R}^3)^3$ for such a
fundamental solution, which satisfies {equations} (\ref{Brinkman-operator4}) with $\alpha =0$. In addition, the components of ${\mathcal G}(\cdot ,\cdot )$ and ${\Pi }(\cdot ,\cdot )$ are given by (see, e.g., \cite[p. 38, 39]{12}):
\begin{equation}
\begin{array}{lll}
\label{fundamental-solution-Stokes} {\mathcal G}_{jk}({\bf x})=
\displaystyle\frac{1}{8\pi }\left\{\frac{\delta
_{jk}}{|{\bf x}|}+ \frac{x_jx_k}{|{\bf
x}|^3}\right\},\quad \
\Pi _{j}({\bf x})=\frac{1}{4\pi }\frac{x_j}{|{\bf x}|^3}.
\end{array}
\end{equation}
The stress and pressure tensors ${\bf S}$ and ${\Lambda }$ have the components (see, e.g., \cite[pp. 39, 132]{12}):
\begin{align}
\label{fundamental-solution-Brinkman2}
S_{jk\ell }({\bf x})=-\displaystyle\frac{3}{4\pi }\displaystyle\frac{{x}_j{x}_k{x}_{\ell }}{|{\bf x}|^{5}},\ \Lambda _{jk}({\bf x})=\frac{1}{2\pi }\left(-\displaystyle\frac{\delta _{jk}}{|{\bf x}|^3}
+3\displaystyle\frac{{x}_j{x}_k}{|{\bf x}|^{5}}\right).
\end{align}
{Let further on, $ {\mathcal G}^{\alpha }({\bf x}, {\bf y})={\mathcal G}^{\alpha }({\bf x}-{\bf y})$ and  $ {\Pi}^{\alpha }({\bf x}, {\bf y})={\Pi}({\bf x}-{\bf y})$.}
For ${\boldsymbol{\varphi}}\in \mathcal D(\mathbb R^3)^3$, $\alpha\ge 0$, the Newtonian velocity and pressure potentials, for the Brinkman system, are, respectively, defined as
\begin{align}
\label{NoT}
\big({\boldsymbol{\mathcal N}}_{{\alpha ; \mathbb R^3 }}{\boldsymbol{\varphi}}\big)({\bf x}):=
-\big\langle {\mathcal G}^{\alpha }({\bf x}, \cdot),{\boldsymbol{\varphi}}
\big\rangle_{_{\!{ \mathbb R^3 } }}=
-\int_{\mathbb R^3} {\mathcal G}^{\alpha }({\bf x}, {\bf y}){\boldsymbol{\varphi}}({\bf y})\, d{\bf y},\\
\label{QoT}
\big({\mathcal Q}_{{\alpha ;\mathbb R^3}}{\boldsymbol{\varphi}}\big)({\bf x})
{=\big({\mathcal Q}_{\mathbb R^3}{\boldsymbol{\varphi}}\big)({\bf x})}
 :=
-\big\langle{\Pi }({\bf x},\cdot),{\boldsymbol{\varphi}}\big\rangle _{_{\!{ \mathbb R^3 } }}=
-\int_{\mathbb R^3} {\Pi}({\bf x}, {\bf y}){\boldsymbol{\varphi}}({\bf y})\, d{\bf y},\ {\bf x} \in {\mathbb R}^3
\end{align}
and the operators
${\boldsymbol{\mathcal N}}_{\alpha ;\mathbb R^3}:{\mathcal D}(\mathbb R^3)^3\to {\mathcal E}(\mathbb R^3)^3$,
${\mathcal Q}_{{\alpha ;\mathbb R^3}}:{\mathcal D}(\mathbb R^3)^3\to {\mathcal E}(\mathbb R^3)$ are evidently continuous.
{Note that due to \eqref{Brinkman-operator4}, we have the relations
\begin{equation}
\label{Newtonian-s1-sm}
\triangle ({\boldsymbol{\mathcal N}}_{\alpha ;\mathbb R^3}{\bf f})-\alpha {\boldsymbol{\mathcal N}}_{\alpha ;\mathbb R^3}{\bf f}-\nabla ({\mathcal Q}_{\alpha ;\mathbb R^3}{\bf f})={\bf f},\ \
{\rm{div}}{\boldsymbol{\mathcal N}}_{\alpha ;\mathbb R^3}{\bf f}=0 \mbox{ in } \mathbb R^3.
\end{equation}
Let $r_{\Omega_{\pm}}$ be the operators {restricting} vector-valued or scalar-valued distributions in ${\mathbb R}^3$ to $\Omega_{\pm}$.
{When} ${\boldsymbol{\mathcal N}}_{\alpha ;{\mathbb R}^3}{\bf f}$ and ${\mathcal Q}_{\alpha ;{\mathbb R}^3}{\bf f}$ are well defined on $\mathbb R^3$, we can also define their restrictions to $\Omega_\pm$ as
\begin{align}
\label{Newtonian-D-B-F-S-M}
{\boldsymbol{\mathcal N}}_{\alpha ;\Omega _{\pm}}{\bf f}:=r_{\Omega _{\pm}}\big({\boldsymbol{\mathcal N}}_{\alpha ;{\mathbb R}^3}{\bf f}\big),\ \
{\mathcal Q}_{\alpha ;{\Omega }_{\pm}}{\bf f}{={\mathcal Q}_{{\Omega }_{\pm}}{\bf f}:=r_{\Omega _{\pm}}\big({\mathcal Q}_{{\mathbb R}^3}{\bf f}\big)}.
\end{align}

We use the notation ${\boldsymbol{\mathcal N}}_{\mathbb R^3}:={\boldsymbol{\mathcal N}}_{0;\mathbb R^3}$
{and} similar ones when $\mathbb R^3$ is replaced by $\Omega _{\pm }$.

{Definitions \eqref{NoT} and \eqref{QoT} can be extended to Sobolev spaces, and the mapping properties of the corresponding operators are proved in  Lemmas \ref{Newtonian-weight} and \ref{Newtonian-B}.}

Now let ${\bf g}\in H^{-\frac{1}{2}}({ \partial\Omega })^3$. Then the single-layer potential ${\bf V}_{{\alpha ;\partial\Omega }}{\bf g}$ and the pressure potential ${\mathcal Q}_{{\alpha ;\partial\Omega }}^s{\bf g}$ for the Brinkman system are given by
\begin{equation}
\label{58} \big({\bf V}_{{\alpha ; \partial\Omega }}{\bf g}\big)({\bf x})
:=\big\langle {\mathcal G}^{\alpha }({\bf x}, \cdot),{\bf g}
\big\rangle_{_{\!{ \partial\Omega } }},\
\big({\mathcal Q}_{{\alpha ; \partial\Omega}}^s{\bf g}\big)({\bf x})
{=\big({\mathcal Q}_{{\partial\Omega}}^s{\bf g}\big)({\bf x}) }
:=\big\langle{\Pi }({\bf x},\cdot),{\bf g}\big\rangle _{_{\!{ \partial\Omega } }},\
{\bf x} \in {\mathbb R}^3 \setminus { \partial\Omega } ,
\end{equation}
{and the corresponding non-tangential limiting values satisfy the relations (see \eqref{68} {in Lemma~\ref{layer-potential-properties}})
\begin{align}
&\gamma_+\big({\bf V}_{\alpha ;\partial\Omega }{\bf g}\big)
=\gamma_{-}\big({\bf V}_{\alpha
;{ \partial\Omega } }{\bf g}\big)
=:{\boldsymbol{\mathcal V}}_{\alpha ; \partial\Omega }{\bf g}.\nonumber
\end{align}}
Let $\nu _{_{\ell }}$, $\ell =1,\ldots ,n$, be the components of the
outward unit normal $\nu $ to $\Omega $, which is defined a.e. on ${ \partial\Omega } $. Let
${\bf h}\in H^{\frac{1}{2}}(\partial\Omega )^3$. Then the double-layer
potential ${\bf W}_{\alpha ;\partial\Omega }{\bf h}$ and its associated
pressure potential ${\mathcal Q}_{\alpha ;{ \partial\Omega } }^d{\bf h}$ for the Brinkman system are
given at any ${\bf x} \in {\mathbb R}^3 \setminus { \partial\Omega } $ by\footnote{The repeated index summation convention is adopted everywhere in the paper.}
\begin{equation}
\label{59}
\!\!\!\!\big(\!{\bf W}_{\alpha ;\partial\Omega }{\bf h}\big)_k({\bf x})
\!:=\!\!\!\displaystyle\int_{\partial\Omega }\!\!\!\!\!S_{jk{\ell}}^{\alpha }({\bf y},{\bf x})\nu
_{_{\!{\ell }}}({\bf y})h_j({\bf y})\!d\sigma_{{\bf y}},\,
\big(\!{\mathcal Q}_{\alpha ;\partial\Omega }^d{\bf h}\big)({\bf x})
\!:=\!\!\!\displaystyle\int_{\partial\Omega }\!\!\!\!\Lambda _{j{\ell}}^{\alpha }({\bf x},{\bf y})\nu
_{_{{\ell }}}({\bf y})h_j({\bf y})\!d\sigma_{{\bf y}}.
\end{equation}
In addition, the principal value of ${\bf W}_{\alpha ;{ \partial\Omega } }{\bf h}$ is denoted by
\begin{equation}
\label{double-layer-principal-value} ({\bf K}_{\alpha ;{ \partial\Omega } }{\bf
h})_k({\bf x}):= {\rm{p.v.}}\int_{\partial\Omega }S_{jk{\ell}}^{\alpha
}({\bf y},{\bf x})\nu _{_{\!{\ell }}}({\bf y})h_j({\bf
y})d\sigma_{{\bf y}}\ \mbox { a.e. } {\bf x}\in { \partial\Omega },
\end{equation}
{and the corresponding conormal derivatives are related by the formula (see \eqref{70aaaa} {in Lemma~\ref{layer-potential-properties}})
\begin{align}
&{\bf t}_{\alpha }^{+}\big({\bf W}_{\alpha ;\partial\Omega }{\bf h},{\mathcal Q}_{\alpha ; \partial\Omega }^d{\bf h}\big)
={\bf t}_{\alpha }^{-}\big({\bf W}_{\alpha ;\partial\Omega }{\bf h},{\mathcal Q}_{\alpha ;\partial\Omega }^d{\bf h}\big)
=:{\bf D}_{\alpha ;\partial\Omega }{\bf h}.\nonumber
\end{align}}
For $\alpha =0$, i.e., in {case} of the Stokes system, we
use the following abbreviations: $${\bf V}_{{ \partial\Omega } }\!:=\!{\bf V}_{0;{ \partial\Omega } },\ {\boldsymbol{\mathcal V}}_{{ \partial\Omega } }\!:=\!{\boldsymbol{\mathcal V}}_{0;{ \partial\Omega } },\ {\bf W}_{{ \partial\Omega }
}\!:=\!{\bf W}_{0;{ \partial\Omega } },\ {\bf K}_{{ \partial\Omega } }\!:=\!{\bf
K}_{0;{ \partial\Omega } },\ {\bf D}_{{ \partial\Omega } }\!:=\!{\bf D}_{0;{ \partial\Omega } }.$$
By (\ref{Brinkman-operator4}) and (\ref{Brinkman-operator6}), $({\bf V}_{\alpha ;\partial\Omega }{\bf
g},{\mathcal Q}_{\partial\Omega }^s{\bf g})$ and $({\bf W}_{\alpha ;\partial\Omega }{\bf h},{\mathcal Q}_{\alpha ;{ \partial\Omega }}^d{\bf h})$ satisfy the Brinkman system in ${\mathbb R}^3\setminus \partial\Omega $.

{The main properties of the layer potential operators for the Stokes system ($\alpha=0$), as well as for the Brinkman system ($\alpha >0$), are {provided in Lemmas  \ref{layer-potential-properties-Stokes} and} \ref{layer-potential-properties}.}

\section{Poisson problem of transmission type for the Stokes and Brinkman systems in complementary Lipschitz domains in weighted Sobolev spaces}
Let $\Omega _{+}:={\Omega }\subset {\mathbb R}^3$ be a bounded Lipschitz domain with connected boundary $\partial \Omega $ and let ${\Omega }_{-}:={\mathbb R}^3\setminus \overline{\Omega }$. Let $\nu $ be the outward unit normal to $\partial {\Omega }$. Consider the spaces
\begin{align}
\label{div-s1}
&H_{{\rm{div}}}^{1}({\Omega }_{+})^3:=\left\{{\bf w}\in H^{1}({\Omega }_{+})^3:{\rm{div}}\ {\bf w}=0 \mbox{ in } \Omega _{+}\right\},
\\
\label{div-2}
&{\mathcal H}_{{\rm{div}}}^{1}({\Omega }_{-})^3:=\left\{{\bf w}\in {\mathcal H}^{1}({\Omega }_{-})^3:{\rm{div}}\ {\bf w}=0 \mbox{ in } \Omega _{-}\right\},\\
\label{space}
&{\mathcal X}:=H^{1}_{\rm div}({\Omega }_{+})^3\times L^2(\Omega _{+})\times {\mathcal H}^{1}_{\rm div}({\Omega }_{-})^3\times L^2(\Omega _{-}),\\
\label{space-data}
&{\mathcal Y}:=\widetilde{H}^{-1}({\Omega }_{+})^3\times \widetilde{\mathcal H}^{-1}({\Omega }_{-})^3\times H^{\frac{1}{2}}(\partial\Omega )^3\times H^{-\frac{1}{2}}(\partial\Omega )^3,\\
\label{space-data-inf}
&{\mathcal Y}_\infty:=\widetilde{H}^{-1}({\Omega }_{+})^3\times \widetilde{\mathcal H}^{-1}({\Omega }_{-})^3\times H^{\frac{1}{2}}(\partial\Omega )^3\times H^{-\frac{1}{2}}(\partial\Omega )^3
\times\mathbb R^3,
\end{align}
{where the norms on the first two spaces are inherited from their parent spaces, $H^{1}({\Omega }_{+})^3$ and $\mathcal H^{1}({\Omega }_{-})^3$, respectively, while the norms in the last three spaces are defined as the sum of the corresponding norms of their components.}


Next we consider the Poisson problem of transmission type for the {incompressible} Stokes and Brinkman systems in the complementary Lipschitz domains $\Omega _{\pm }$,
\begin{equation}
\label{Poisson-Brinkman-a-Besov-Dirichlet}
\left\{\begin{array}{lll}
\triangle {\bf u}_{+}-\alpha {\bf u}_{+}-\nabla \pi _{+}={{\tilde{\bf f}}_{+}|_{\Omega _{+}}}\
\mbox{ in }\ {\Omega }_{+},\\
\triangle {\bf u}_{-}-\nabla \pi _{-}={\tilde{\bf f}}_{-}|_{\Omega _{-}}\
\mbox{ in }\ {\Omega }_{-},\\
\gamma_{+}{\bf u}_{+}-\gamma_{-}{\bf u}_{-}={\bf h}_0 \mbox{ on } \partial\Omega ,\\
{\bf t}_{\alpha }^{+}({\bf u}_{+},\pi _{+};{\tilde{\bf f}}_{+})-
\mu {\bf t}_{0}^{-}({\bf u}_{-},\pi _{-};{\tilde{\bf f}}_{-})+\frac{1}{2}{\mathcal P}\, \left(\gamma_{-}{\bf u}_{-}+\gamma_{+}{\bf u}_{+}\right)
={\bf g}_{0} \mbox{ on } \partial\Omega ,
\end{array}\right.
\end{equation}
where $\mu >0$ is a constant and ${\mathcal P}\in L^{\infty }({\partial\Omega })^{3\times 3}$ is a symmetric matrix-valued function which satisfies the positivity condition
\begin{equation}
\label{positive-definite-lambda}
\langle {\mathcal P}{\bf v},{\bf v}\rangle _{_{\partial\Omega }}\geq 0,\ \forall \ {\bf v}\in L^2(\partial \Omega )^3.
\end{equation}
The transmission conditions on $\partial\Omega$ are a generalization of the ones considered, e.g., in \cite{Ochoa_Tapia-Whitacker1995-1}, \cite{Ochoa_Tapia-Whitacker1995-2}, and
the constant $\mu$ is the ratio of viscosity coefficients in $\Omega_+$ and $\Omega_-$, by which the equations are normalized. {\it Note that the conormal derivative ${\bf t}_{0}^{-}({\bf u}_{-},\pi _{-};{\tilde{\bf f}}_{-})$ in the last transmission condition in \eqref{Poisson-Brinkman-a-Besov-Dirichlet} is well defined due to Remark $\ref{CSD}$}.

We will look for a solution $\left({\bf u}_{+},\pi _{+},{\bf u}_{-},\pi _{-}\right)$ of {the transmission problem} \eqref{Poisson-Brinkman-a-Besov-Dirichlet} satisfying
\begin{align}\label{udiff}
\left({\bf u}_{+},\pi _{+},{\bf u}_{-}-{\bf u}_{\infty },\pi _{-}\right)&\in {\mathcal X}
\end{align}
for a given constant vector ${\bf u}_{\infty }\in {\mathbb R}^3$, and show the well-posedness of {the transmission problem} \eqref{Poisson-Brinkman-a-Besov-Dirichlet}, \eqref{udiff}.
Let us start from uniqueness.
\begin{lem}
\label{well-posedness-Poisson-Besov-Dirichlet-L}
Let $\Omega _{+}:={\Omega }\subset {\mathbb R}^3$ be a bounded Lipschitz domain with connected boundary $\partial \Omega $. Let ${\Omega }_{-}:={\mathbb R}^3\setminus \overline{\Omega }$ denote the corresponding exterior Lipschitz domain. Let $\alpha ,\mu >0$ and condition \eqref{positive-definite-lambda} hold.
Then for $\big({\tilde{\bf f}}_{+},{\tilde{\bf f}}_{-},{\bf h}_0,{\bf g}_{0},{\bf u}_{\infty }\big)\in {\mathcal Y}_\infty$, {the transmission problem} \eqref{Poisson-Brinkman-a-Besov-Dirichlet}
has at most one solution $\left({\bf u}_{+},\pi _{+},{\bf u}_{-},\pi _{-}\right)$ satisfying condition \eqref{udiff}.
\end{lem}
\begin{proof}
Let us assume that $\left({\bf u}_{+}^0,\pi _{+}^0,{\bf u}_{-}^0,\pi _{-}^0\right)$ is the difference of two solutions of (\ref{Poisson-Brinkman-a-Besov-Dirichlet}) satisfying \eqref{udiff}.
Then $\left({\bf u}_{+}^0,\pi _{+}^0,{\bf u}_{-}^0,\pi _{-}^0\right)$ belongs to ${\mathcal X}$ and satisfies the homogeneous version of \eqref{Poisson-Brinkman-a-Besov-Dirichlet}.
We now apply the Green formula \eqref{conormal-generalized-2} to
$\left({\bf u}_{+}^0,\pi _{+}^0\right)\in H^{1}_{\rm{div}}({\Omega }_{+})^3\times L^2(\Omega _{+})$ in the Lipschitz domain $\Omega _{+}$, and obtain
\begin{align}
\label{Green-sum-S}
\langle {\bf t}_{\alpha }^+({\bf u}_{+}^0,\pi _{+}^0),\gamma_{+}{\bf u}_{+}^0\rangle_{\partial \Omega }=2\|{\mathbb E}({\bf u}_{+}^0)\|_{L^2(\Omega _+)^{3\times 3}}^2+ \alpha \|{\bf u}_{+}^0\|_{L^2(\Omega _{+})^3}^2,
\end{align}
where ${\mathbb E}({\bf u}_{+}^0)$ is the symmetric part of $\nabla {\bf u}_{+}^0$.
By applying the Green formula (\ref{conormal-generalized-2-Sigma-S}) to the pair $\left({\bf u}_{-}^0,\pi _{-}^0\right)\in {\mathcal H}^{1}_{\rm{div}}({\Omega}_{-})^3\times L^2(\Omega _{-})$ in the exterior domain ${\Omega }_{-}$, we obtain that
\begin{align}
\label{Green-sum-1-S}
-\mu \langle {\bf t}_{0}^{-}({\bf u}_{-}^0,\pi _{-}^0),\gamma_{-}{\bf u}_{-}^0\rangle _{\partial \Omega }=2\mu \|{\mathbb E}({\bf u}_{-}^0)\|_{L^2(\Omega _{-})^{3\times 3}}^2.
\end{align}
By adding {formulas} (\ref{Green-sum-S}) and (\ref{Green-sum-1-S}) and by exploiting the homogeneous transmission conditions satisfied by the pairs $\left({\bf u}_{+}^0,\pi _{+}^0\right)$ and $\left({\bf u}_{-}^0,\pi _{-}^0\right)$, we obtain the equality
\begin{align}
\label{Green-sum-2-S}
2\|{\mathbb E}({\bf u}_{+}^0)\|_{L^2(\Omega _{+})^{3\times 3}}^2&+\alpha \|{\bf u}_{+}^0\|_{L^2(\Omega _{+})^3}^2+2\mu \|{\mathbb E}({\bf u}_{-}^0)\|_{L^2({\Omega }_{-})^{3\times 3}}^2\nonumber\\
&=\left\langle {\bf t}_{\alpha }^+({\bf u}_{+}^0,\pi _{+}^0)-\mu {\bf t}_{0}^{-}({\bf u}_{-}^0,\pi _{-}^0), \gamma_{-}{\bf u}_{-}^0\right \rangle _{\partial \Omega }
=-\left\langle {\mathcal P}\, \gamma_{-}{\bf u}_{-}^0, \gamma_{-}{\bf u}_{-}^0\right \rangle _{\partial \Omega },
\end{align}
where the left-hand side is non-negative, while the right-hand side is {non-positive due to
{the positivity condition} (\ref{positive-definite-lambda}) satisfied by the matrix-valued function ${\mathcal P}$}. Therefore, both hand sides vanish, and hence
\begin{equation}
\label{S-1}
{\bf u}_{+}^0={\bf 0} \mbox{ in } \Omega _{+},\ \ {\mathbb E}({\bf u}_{-}^0)=0 \mbox{ in } {\Omega }_{-}.
\end{equation}
Thus, $\pi _{+}^0=c\in {\mathbb R}$ in $\Omega _{+}$. In addition, the first relation in (\ref{S-1}) and the transmission condition $\gamma_{-}{\bf u}_{-}^0=\gamma_{+}{\bf u}_{+}^0$ on $\partial \Omega $ yield
\begin{equation}
\label{S-2}
\gamma_{-}{\bf u}_{-}^0={\bf 0} \mbox{ on } \partial\Omega .
\end{equation}
Hence, $({\bf u}_{-}^0,\pi _{-}^0)$ is a solution in the space ${\mathcal H}^{1}(\Omega _{-})^3\times L^2(\Omega _{-})$ of the exterior Dirichlet problem for the homogeneous Stokes system with homogeneous Dirichlet boundary condition. {By} {\cite[Theorem 3.4]{Gi-Se} (see also \cite[Theorem 2.1]{Al-Am}), the solution of such a problem is unique, i.e.,} {${\bf u}_{-}^0={\bf 0}$, $\pi _{-}^0=0$ in $\Omega _{-}$}. Then
$$
\gamma_{-}{\bf u}_{-}^0=\gamma_{+}{\bf u}_{+}^0,\, {\bf t}_{\alpha }^{+}({\bf u}_{+}^0,\pi _{+}^0)-\mu {\bf t}_{0}^{-}({\bf u}_{-}^0,\pi _{-}^0)+\frac{1}{2}{\mathcal P}\, \left(\gamma_{-}{\bf u}_{-}^0+\gamma_{+}{\bf u}_{+}^0\right)={\bf 0}
$$
and (\ref{S-2}) yield that $\pi ^0_{+}=0$ in $\Omega _{+}$. Consequently,
${{\bf u}_{\pm }^0={\bf 0},\ \pi _{\pm }^0=0\ \mbox{ in } \ {\Omega }_{\pm }}.$
\end{proof}
Next we show the existence of solution of transmission problem (\ref{Poisson-Brinkman-a-Besov-Dirichlet}), \eqref{udiff}
in the case ${\bf u}_{\infty }={\mathbf 0}$.
\begin{thm}
\label{well-posedness-Poisson-Besov-Dirichlet0}
Let $\Omega _{+}:={\Omega }\subset {\mathbb R}^3$ be a bounded Lipschitz domain with connected boundary $\partial \Omega $. Let ${\Omega }_{-}:={\mathbb R}^3\setminus \overline{\Omega }$ be the exterior Lipschitz domain. Let $\alpha ,\mu >0$ and condition \eqref{positive-definite-lambda} hold.
Then for $\big({\tilde{\bf f}}_{+},{\tilde{\bf f}}_{-},{\bf h}_0,{\bf g}_{0}\big)\in {\mathcal Y}$, {the transmission problem} $\eqref{Poisson-Brinkman-a-Besov-Dirichlet}$
has a unique solution $\left({\bf u}_{+},\pi _{+},{\bf u}_{-},\pi _{-}\right)\in {\mathcal X}$ and
there exists a linear continuous operator
\begin{align}
\label{T-1}
{\mathcal T}:{\mathcal Y}\to{\mathcal X}
\end{align}
delivering such a solution.
Hence there exists a constant $C\equiv C({\Omega }_{+},{\Omega }_{-},{\mathcal P},\alpha ,\mu)>0$ such that
\begin{align}
\label{estimate-Poisson-Besov0}
\|\left({\bf u}_{+},\pi _{+},{\bf u}_{-},\pi _{-}\right)\|_{{\mathcal X}}\leq
C\big\|(\tilde{\bf f}_{+},\tilde{\bf f}_{-},{\bf h}_0,{\bf g}_{0})\big\|_{\mathcal Y}\ .
\end{align}
Moreover, ${\bf u}_{-}$ vanishes at infinity in the sense of Leray, i.e.,
\begin{align}
\label{behavior-infinity-linear0}
\lim_{r\to \infty }\int _{S^2}|{\bf u}_{-}(r{\bf y})|d\sigma _{\bf y}=0.
\end{align}
\end{thm}
\begin{proof}
We construct a solution $\left({\bf u}_{+},\pi _{+},{\bf u}_{-},\pi _{-}\right)\in {\mathcal X}$ of problem \eqref{Poisson-Brinkman-a-Besov-Dirichlet} in the form
\begin{align}
\label{3.3-new-s1}
&{\bf u}_{+}={\boldsymbol{\mathcal N}}_{\alpha ;{\Omega }_{+}}{\tilde{\bf f}}_{+}+{\bf w}_{+},\ \ \pi _{+}={\mathcal Q}_{\alpha ;{\Omega }_{+}}{\tilde{\bf f}}_{+}+p_{+} \mbox{ in } {\Omega }_{+},
\\
\label{3.3-new-s-m}
&{\bf u}_{-}={\boldsymbol{\mathcal N}}_{{\Omega }_{-}}{\tilde{\bf f}}_{-}+{\bf w}_{-},\ \ \pi _{-}={\mathcal Q}_{{\Omega }_{-}}{\tilde{\bf f}}_{-}+p_{-} \mbox{ in } {\Omega }_{-}.
\end{align}
{Here} ${\boldsymbol{\mathcal N}}_{\alpha ;{\Omega }_{+}}{\tilde{\bf f}}_{+}$ and ${\mathcal Q}_{\alpha ;{\Omega }_{+}}{\tilde{\bf f}}_{+}$ are the Newtonian velocity and pressure potentials with the density ${\tilde{\bf f}}_{+}\in \widetilde{H}^{-1}({\Omega }_{+})^3$, which correspond to the Brinkman system in ${\Omega }_{+}$, see (\ref{Newtonian-D-B-F-S-M}). Therefore,
\begin{align}
\label{Newtonian-s1}
\triangle {\boldsymbol{\mathcal N}}_{\alpha ;{\Omega }_{+}}{\tilde{\bf f}}_{+}-\alpha {\boldsymbol{\mathcal N}}_{\alpha ;{\Omega }_{+ }}{\tilde{\bf f}}_{+}-\nabla {\mathcal Q}_{\alpha ;{\Omega }_{+}}{\tilde{\bf f}}_{+}={\tilde{\bf f}}_{+},\
{\rm{div}}\left({\boldsymbol{\mathcal N}}_{\alpha ;{\Omega }_{+}}{\tilde{\bf f}}_{+}\right)=0 \mbox{ in } {\Omega }_{+},
\end{align}
{and by (\ref{Newtonian-Brinkman-SM}) and (\ref{Newtonian-Brinkman+}),
\begin{align}\label{4.21}
{\boldsymbol{\mathcal N}}_{\alpha ;{\Omega }_{+}}{\tilde{\bf f}}_{+}\in H^{1}_{\rm div}({\Omega }_{+})^3,\quad {\mathcal Q}_{\alpha ;{\Omega }_{+}}{\tilde{\bf f}}_{+}\in L^2({\Omega }_{+}).
\end{align}
}
Similarly, the Newtonian {velocity} potential ${\boldsymbol{\mathcal N}}_{{\Omega }_{-}}{\tilde{\bf f}}_{-}$ for the Stokes system in $\Omega _{-}$ and its corresponding pressure potential ${\mathcal Q}_{{\Omega }_{-}}{\tilde{\bf f}}_{-}$ satisfy the equations
\begin{equation}
\label{Newtonian-Stokes-s1}
\triangle {\boldsymbol{\mathcal N}}_{{\Omega }_{-}}{\tilde{\bf f}}_{-}-\nabla {\mathcal Q}_{{\Omega }_{-}}{\tilde{\bf f}}_{-}={\tilde{\bf f}}_{-},\
{\rm{div}}\left({\boldsymbol{\mathcal N}}_{{\Omega }_{-}}{\tilde{\bf f}}_{-}\right)=0\ \mbox{ in }\ {\Omega }_{-},
\end{equation}
and {by (\ref{Newtonian-Stokes-S}),
\begin{align}\label{4.23}
{\boldsymbol{\mathcal N}}_{{\Omega }_{-}}{\tilde{\bf f}}_{-}\in {\mathcal H}^{1}_{\rm div}({\Omega }_{-})^3,\quad {\mathcal Q}_{{\Omega }_{-}}{\tilde{\bf f}}_{-}\in L^2({\Omega }_{-})
\end{align}}

Thus, $\left({\bf v}_{+},\pi _{+},{\bf v}_{-},\pi _{-}\right)$ given by (\ref{3.3-new-s1}) and (\ref{3.3-new-s-m}) is a solution of the Poisson problem of transmission type (\ref{Poisson-Brinkman-a-Besov-Dirichlet}) in the space ${\mathcal X}$ if and only if
$\left({\bf w}_{+},p_{+},{\bf w}_{-},p_{-}\right)\in {\mathcal X}$
satisfies the following transmission problem for the Stokes and Brinkman systems
\begin{equation}
\label{Dirichlet-transmission-ms}
\left\{\begin{array}{lll} \triangle
{\bf w}_{+}-\alpha {\bf w}_{+}-\nabla p_{+}={\bf 0}\
\mbox{ in }\ {\Omega }_{+},\\
\triangle {\bf w}_{-}-\nabla p_{-}={\bf 0}\
\mbox{ in }\ {\Omega }_{-},\\
\gamma_{+}{\bf w}_{+}-\gamma_{-}{\bf w}_{-}
={{\bf h}_{00}} \mbox{ on } \partial\Omega ,\\
{\bf t}_{\alpha }^{+}({\bf w}_{+},p_{+})-\mu {\bf t}_{0}^{-}({\bf w}_{-},p_{-})
+\frac{1}{2}{\mathcal P}\, \left(\gamma_{-}{\bf w}_{-}+\gamma_{+}{\bf w}_{+}\right)
={{\bf g}_{00}} \mbox{ on } { \partial\Omega },
\end{array}\right.
\end{equation}
where
\begin{align}
\label{H0}
&{{\bf h}_{00}}:={\bf h}_0-\gamma_{+}\left({\boldsymbol{\mathcal N}}_{\alpha ;{\Omega }_{+}}{\tilde{\bf f}}_{+}\right)+\gamma_{-}\left({\boldsymbol{\mathcal N}}_{{\Omega }_{-}}{\tilde{\bf f}}_{-}\right),\\
\label{G0}
&{{\bf g}_{00}}:={\bf g}_0-{\bf t}_{\alpha }^{+}\left({\boldsymbol{\mathcal N}}_{\alpha ;{\Omega }_{+}}{\tilde{\bf f}}_{+},
{\mathcal Q}_{\alpha ;{\Omega }_{+}}{\tilde{\bf f}}_{+};{\tilde{\bf f}}_{+}\right)+\mu {\bf t}_{0}^{-}\left({\boldsymbol{\mathcal N}}_{{\Omega }_{-}}{\tilde{\bf f}}_{-},{\mathcal Q}_{{\Omega }_{-}}{\tilde{\bf f}}_{-};{\tilde{\bf f}}_{-}\right)\nonumber
\\
&\quad \quad \quad \quad \quad \quad \quad \ \ -\frac{1}{2}{\mathcal P}\, \left(\gamma_{+}\left({\boldsymbol{\mathcal N}}_{\alpha ;{\Omega }_{+}}{\tilde{\bf f}}_{+}\right)+\gamma_{-}\left({\boldsymbol{\mathcal N}}_{{\Omega }_{-}}{\tilde{\bf f}}_{-}\right)\right).
\end{align}

{By} Lemma \ref{trace-operator1} and {(\ref{4.21})}, $\gamma_{+}\left({\boldsymbol{\mathcal N}}_{\alpha ;{\Omega }_{+}}{\tilde{\bf f}}_{+}\right)\in H^{\frac{1}{2}}(\partial \Omega )^3$, {while by \eqref{4.23} and continuity} of the exterior trace operator $\gamma_{-}:{\mathcal H} ^{1}({\Omega }_{-})^3\to H^{\frac{1}{2}}({\partial\Omega  })^3$ we obtain that $\gamma_{-}\left({\boldsymbol{\mathcal N}}_{{\Omega }_{-}}{\tilde{\bf f}}_{-}\right)\in H^{\frac{1}{2}}(\partial \Omega )^3$. Then the assumption ${\bf h}_{0}\in H^{\frac{1}{2}}(\partial \Omega )^3$ {implies}
{${\bf h}_{00}\in H^{\frac{1}{2}}(\partial \Omega )^3.$}
Since ${\bf g}_{0}\in H^{-\frac{1}{2}}(\partial \Omega )^3$ and ${\mathcal P}\in L^{\infty }({\partial\Omega })^{3\times 3}$, Lemmas \ref{conormal-derivative-generalized-Robin} and \ref{conormal-derivative-generalized-Robin-S} imply
${\bf g}_{00}\in H^{-\frac{1}{2}}(\partial \Omega )^3$.

Next, we look for unknown fields ${\bf w}_{\pm }$ and $p_{\pm }$ in terms
of the following layer potentials
\begin{align}
\label{3.3-new}
&{\bf w}_{+}={\bf W}_{_{\alpha ;\partial \Omega }}\boldsymbol{\Phi} + {\bf V}_{_{\alpha ;\partial \Omega }}\boldsymbol{\varphi} ,\ \ p_{+}={\mathcal Q}_{_{\alpha ;\partial \Omega }}^d\boldsymbol{\Phi} + {\mathcal Q}_{_{\alpha ;\partial \Omega }}^s\boldsymbol{\varphi} \ \mbox{ in }\ {\Omega }_{+},
\\
\label{3.1-new}
&{\bf w}_{-}={\bf W}_{_{\partial \Omega }}\boldsymbol{\Phi} + {\bf V}_{_{\partial \Omega }}\boldsymbol{\varphi} ,\ \ p_{-}={\mathcal Q}_{_{\partial \Omega }}^d\boldsymbol{\Phi} +{\mathcal Q}_{_{\partial \Omega }}^s\boldsymbol{\varphi} \ \mbox{ in }\ {\Omega }_{-},
\end{align}
with unknown densities $\left(\boldsymbol{\Phi} ,\boldsymbol{\varphi} \right)^{\top }\!\!\in \!H^{\frac{1}{2}}(\partial \Omega )^3\!\times H^{-\frac{1}{2}}(\partial \Omega )^3$. Note that the function set
$\left({\bf w}_{+},p_{+},{\bf w}_{-},p_{-}\right)$ satisfies the domain equations in the first two lines of {transmission problem} \eqref{Dirichlet-transmission-ms} and belongs to ${\mathcal X}$
due to {the mapping properties} (\ref{ss-s1-ms}), (\ref{ds-s1-ms}), (\ref{ss-s1}) and (\ref{ds-s1}).

By (\ref{68-Stokes}), (\ref{68-s1-Stokes}), (\ref{68}) and (\ref{68-s1}), and the first transmission
condition in (\ref{Dirichlet-transmission-ms}) we obtain 
\begin{align}
\label{3.7a}
\left(-{\mathbb I}+{\bf K}_{\alpha ,0;\partial\Omega }\right)\boldsymbol{\Phi} +
{\boldsymbol{\mathcal V}}_{\alpha ,0;\partial\Omega }\boldsymbol{\varphi}
={{\bf h}_{00}} \mbox{ on } \partial\Omega ,
\end{align}
where the complementary single- and double-layer potential operators
\begin{align}
\label{2.2tporousmedia-new}
{\boldsymbol{\mathcal V}}_{\alpha ,0;\partial\Omega }&:= {\boldsymbol{\mathcal V}}_{\alpha ;\partial\Omega }- {\boldsymbol{\mathcal V}}_{\partial\Omega }:H^{-\frac{1}{2}}(\partial \Omega )^3\to H^{\frac{1}{2}}(\partial \Omega )^3
\\
\label{2.2tporousmedia-new1}
{\bf K}_{\alpha ,0;\partial\Omega }&:={\bf K}_{\alpha ;\partial\Omega }-{\bf K}_{\partial\Omega }:
H^{\frac{1}{2}}(\partial \Omega )^3\to H^{\frac{1}{2}}(\partial \Omega )^3
\end{align}
are compact (see \cite[Theorem 3.1]{K-L-W}).

Now {the jump relations} (\ref{70aaa-Stokes}), (\ref{70aaaa-Stokes}), (\ref{70aaa}) and (\ref{70aaaa}), and the second
transmission condition in (\ref{Dirichlet-transmission-ms}) lead to the equation
\begin{align}
\label{3.10}
\Big(\frac{1}{2}(1+\mu ){\mathbb I}&+(1-\mu ){\bf K}_{\partial\Omega }^*+{\bf K}_{\alpha ,0;{ \partial\Omega } }^*\Big)\boldsymbol{\varphi} +\Big((1-\mu ){\bf D}_{\partial\Omega }+{\bf D}_{\alpha ,0;\partial\Omega }\Big)\boldsymbol{\Phi} \nonumber\\
&+\frac{1}{2}{\mathcal P}\Big(\left({\boldsymbol{\mathcal V}}_{{\alpha ;\partial\Omega }}+{\boldsymbol{\mathcal V}}_{{\partial\Omega }}\right)\boldsymbol{\varphi} +\left({\bf K}_{\alpha ;\partial\Omega }+{\bf K}_{\partial\Omega }\right)\boldsymbol{\Phi} \Big)={{\bf g}_{00}} \mbox{ on } \partial\Omega ,
\end{align}
where the complementary layer potential operators
\begin{align}
\label{2.2tporousmedia-new-s}
&{\bf D}_{\alpha ,0;\partial \Omega }:={\bf D}_{\alpha ;\partial \Omega }-{\bf D}_{\partial \Omega }:H^{\frac{1}{2}}(\partial \Omega )^3\to H^{-\frac{1}{2}}(\partial \Omega )^3,\\
\label{2.2tporousmedia-new-s1}
&{\bf K}_{\alpha ,0;\partial\Omega }^*:={\bf K}_{\alpha ;\partial\Omega }^*-{\bf K}_{\partial\Omega }^*:H^{-\frac{1}{2}}(\partial \Omega )^3\to H^{-\frac{1}{2}}(\partial \Omega )^3,
\end{align}
are linear and compact, {and \eqref{2.2tporousmedia-new-s1} is} the adjoint of the complementary double-layer potential operator ${\bf K}_{\alpha ,0;\partial\Omega }:H^{\frac{1}{2}}(\partial \Omega )^3
\to H^{\frac{1}{2}}(\partial \Omega )^3$ (see again \cite[Theorem 3.1]{K-L-W}).

Now we set
$
{\mathbb X}:=H^{\frac{1}{2}}(\partial \Omega )^3\times H^{-\frac{1}{2}}(\partial \Omega )^3.
$

Then the {transmission problem}
(\ref{Dirichlet-transmission-ms}) reduces to {the system of equations}
(\ref{3.7a}) and (\ref{3.10}), which can be written in the following matrix form
\begin{equation}
\label{3.7operator}
{\mathcal T}_{\alpha; \mu }(\boldsymbol{\Phi} ,\boldsymbol{\varphi} )^{\top }={\mathfrak D}\ \mbox{ in }\ {\mathbb X},
\end{equation}
with unknown $(\boldsymbol{\Phi} ,\boldsymbol{\varphi} )^{\top }\in {\mathbb X}$, where
${{\mathcal T}_{\alpha;\mu}}:{\mathbb X}\to {\mathbb X}$ is the {matrix} operator
\begin{equation}
\label{3.7operator-new1}
{\mathcal T}_{\alpha; \mu }:=\left(\!
\begin{array}{cc}
-{\mathbb I}+{\bf K}_{{\alpha ,0;{ \partial\Omega } }} \ & \ {\boldsymbol{\mathcal V}}_{{\alpha ,0;\partial\Omega }} \\
(1-\mu ){\bf D}_{\partial\Omega }+{\bf D}_{\alpha ,0;\partial\Omega }+\frac{1}{2}{\mathcal P}\!\left({\bf K}_{\alpha ;\partial\Omega }\!+\!{\bf K}_{\partial\Omega }\right) \ & \ {{\mathcal K}_{\mu ; \partial \Omega }} \!+\!{\bf K}_{\alpha ,0;{ \partial\Omega } }^*\!+\!\frac{1}{2}{\mathcal P}\left({\boldsymbol{\mathcal V}}_{{\alpha ;\partial\Omega }}\!+\!{\boldsymbol{\mathcal V}}_{{\partial\Omega }}\right)
\end{array}
\!\!\right)
\end{equation}
and ${\mathcal K}_{\mu ; \partial \Omega }:H^{-\frac{1}{2}}(\partial \Omega )^3\to H^{-\frac{1}{2}}(\partial \Omega )^3$ is the operator given by
\begin{align}
\label{new-oper}
{\mathcal K}_{\mu ; \partial \Omega }:=\frac{1}{2}(1+\mu ){\mathbb I}+(1-\mu ){\bf K}_{\partial\Omega }^*.
\end{align}
In addition,
{
$
{\mathfrak D}:=\left({{\bf h}_{00}},{{\bf g}_{00}}\right)^{\top }\in {\mathbb X}.
$
}
The operator
${\mathcal T}_{\alpha ;\mu }:{\mathbb X}\to {\mathbb X}$ can be written as
$
{\mathcal T}_{\alpha ;\mu }={\mathcal T}_{\mu }+{\mathcal C}_{\alpha ;0},
$
where ${\mathcal T}_{\mu }:{\mathbb X}\to {\mathbb X}$ and ${\mathcal C}_{\alpha ;0}:{\mathbb X}\to {\mathbb X}$ are the operators defined by
\begin{equation}
\label{Fredholm-compact2} {\mathcal T}_{\mu }:=\left(
\begin{array}{ccc}
-{\mathbb I}\ & \ {\bf 0} \\
(1-\mu ){\bf D}_{\partial\Omega }\ & \ {\mathcal K}_{\mu ; \partial \Omega } \\
\end{array}
\right),
\end{equation}
\begin{equation}
\label{Fredholm-compact2-T}
\!\!\!\!{\mathcal C}_{\alpha ;0}:= \left(
\begin{array}{ccc}
{\bf K}_{{\alpha ,0;{ \partial\Omega } }} \ & \ {\boldsymbol{\mathcal V}}_{{\alpha ,0;{ \partial\Omega } }} \\
{\bf D}_{\alpha ,0;\partial\Omega } +\frac{1}{2}{\mathcal P}\!\left({\bf K}_{\alpha ;\partial\Omega }\!+\!{\bf K}_{\partial\Omega }\right) \ & \
{\bf K}_{\alpha ,0; \partial\Omega }^* +\frac{1}{2}{\mathcal P}\left({\boldsymbol{\mathcal V}}_{{\alpha ;\partial\Omega }}+{\boldsymbol{\mathcal V}}_{{\partial\Omega }}\right) \\
\end{array}
\right).
\end{equation}
We now show that the operator ${\mathcal T}_{\mu }:{\mathbb X}\to {\mathbb X}$ is an isomorphism for any $\mu >0$.

(i) If $\mu =1$ then ${\mathcal T}_{\mu }$ reduces to the isomorphism
$
\left(
\begin{array}{ccc}
-{\mathbb I}\ & \ {\bf 0} \\
{\bf 0}\ & \ {\mathbb I} \\
\end{array}
\right).
$

(ii)
If $\mu \in (0,+\infty )\setminus \{1\}$, then
$
{\mathcal K}_{\mu ; \partial \Omega }=(1-\mu )\left(\dfrac{1}{2}\dfrac{1+\mu }{1-\mu }{\mathbb I}+{\bf K}_{\partial\Omega }^*\right)
$
and this operator is an isomorphism whenever $\mu \in (0,1)$ (cf., e.g., \cite[Theorem 5.3.6, Lemma 11.9.21]{M-W}, \cite[Proposition 10.6]{Agr-1}, and an interpolation argument, as in the proof of \cite[Theorem 10.5.3]{M-W}). If $\mu \in (1,+\infty )$, such a property is still valid. Indeed, we have $\mu ^{-1}\in (0,1)$, and
$
{\mathcal K}_{\mu ; \partial \Omega }=(1-\mu )\left(-\dfrac{1}{2}\dfrac{1+\mu ^{-1}}{1-\mu ^{-1}}{\mathbb I}+{\bf K}_{\partial\Omega }^*\right),
$
and again \cite[Theorem {5.3.6}]{M-W} implies that ${\mathcal K}_{\mu ; \partial \Omega }$ is an isomorphism.

Consequently, for any $\mu \in (0,+\infty )$ the operator ${\mathcal K}_{\mu ; \partial \Omega }:H^{-\frac{1}{2}}(\partial \Omega )^3\to H^{-\frac{1}{2}}(\partial \Omega )^3$ given by (\ref{new-oper}) is an isomorphism, and then the operator ${\mathcal T}_{\mu }:{\mathbb X}\to {\mathbb X}$ given by (\ref{Fredholm-compact2}) is an isomorphism as well. The operator ${\mathcal C}_{\alpha ;0}:{\mathbb X}\to {\mathbb X}$ is linear and compact {due to} the compactness of {operators} (\ref{2.2tporousmedia-new-s}) and (\ref{2.2tporousmedia-new-s1}) and the compactness of the embedding $L^2(\partial \Omega )^3\hookrightarrow H^{-\frac{1}{2}}(\partial \Omega )^3$. Hence the operator ${\mathcal T}_{\alpha ;\mu }:{\mathbb X}\to {\mathbb X}$ given by (\ref{3.7operator-new1}) is a Fredholm operator with index zero. We now show that ${\mathcal T}_{\alpha ;\mu } $ is also one-to-one, i.e.,
\begin{equation*}
{\rm{Ker}}\left\{{\mathcal T}_{\alpha ;\mu }:{\mathbb X}\to
{\mathbb X}\right\}=\{0\}.
\end{equation*}
Let $\left(\boldsymbol{\Phi} ^0,\boldsymbol{\varphi} ^0\right)^{\top }\in {\rm{Ker}}\left\{{\mathcal T}_{\alpha ;\mu }:{\mathbb X}\to {\mathbb X}\right\}$, and consider the layer potentials
\begin{align}
\label{3.3-new-aaa}
{\bf u}^{0}={\bf W}_{\alpha ;\partial\Omega }\boldsymbol{\Phi} ^0+{\bf V}_{\alpha ;\partial\Omega }\boldsymbol{\varphi} ^0,\quad
&\pi^{0}={\mathcal Q}_{\alpha ;\partial\Omega }^d\boldsymbol{\Phi} ^0+{\mathcal Q}_{\alpha ;\partial\Omega }^s\boldsymbol{\varphi} ^0
\mbox{ in } {\mathbb R}^3\setminus \partial\Omega ,
\\
\label{3.1-new-aaa}
{\bf v}^{0}={\bf W}_{\partial\Omega }\boldsymbol{\Phi} ^0+{\bf V}_{\partial\Omega }\boldsymbol{\varphi} ^0,\quad
&p^{0}={\mathcal Q}_{\partial\Omega }^d\boldsymbol{\Phi} ^0+{\mathcal Q}_{\partial\Omega }^s\boldsymbol{\varphi} ^0
\mbox{ in } {\mathbb R}^3\setminus \partial\Omega .
\end{align}

By Lemma \ref{layer-potential-properties-Stokes}, Lemma \ref{layer-potential-properties} and the embedding $L^2(\Omega _{-})\subset \mathfrak M(\Omega _{-})$, we have the following inclusions for the restrictions to $\Omega_\pm$,
\begin{align}
\label{4-43}
&({\bf u}^{0},\pi^{0})\in H_{{\rm{div}}}^{1}({\Omega }_{+})^3\times L^2({\Omega }_{+}),\quad
({\bf u}^{0},\pi^{0})\in H_{{\rm{div}}}^{1}({\Omega }_{-})^3\times \mathfrak M({\Omega }_{-}),\\
\label{4-44}
&({\bf v}^{0},p^{0})\in H_{{\rm{div}}}^{1}({\Omega }_{+})^3\times L^2({\Omega }_{+}),\quad
({\bf v}^{0},p^{0})\in {\mathcal H}_{{\rm{div}}}^{1}(\Omega _{-})^3\times L^2({\Omega }_{-}).
\end{align}
In addition, the functions $\left({\bf u}^{0},\pi^{0},{\bf v}^{0},p^{0}\right)$ determine a solution of the homogeneous transmission problem associated to (\ref{Poisson-Brinkman-a-Besov-Dirichlet}) in the space ${\mathcal X}$. Then by Lemma \ref{well-posedness-Poisson-Besov-Dirichlet-L} we obtain that
\begin{equation}
\label{3.3-new1g} {\bf u}^{0}={\bf 0},\ \pi^{0}=0\ \mbox{ in }\ {\Omega }_{+},\ \ {\bf v}^{0}={\bf 0},\ p^{0}=0\ \mbox{ in }\ {\Omega }_{-}.
\end{equation}

Now {formulas} (\ref{68-Stokes}), (\ref{70aaa-Stokes}), (\ref{68}) and (\ref{70aaa}) applied to the layer potentials (\ref{3.3-new-aaa}) and (\ref{3.1-new-aaa}), together with (\ref{3.3-new1g}) yield that
\begin{align}
\label{3.3-new1f}
&\gamma_{-}{\bf u}^{0}=\boldsymbol{\Phi} ^0,\ \gamma_{+}{\bf v}^{0}=-\boldsymbol{\Phi} ^0\ \mbox{ on }\ \partial\Omega ,
\\
\label{3.3-new2a}
&{\bf t}_{\alpha }^{-}({\bf u}^{0},\pi^{0})=-\boldsymbol{\varphi} ^0,\
{\bf t}_{0}^{+}({\bf v}^{0},p_{0})=\boldsymbol{\varphi} ^0\ \mbox{ on }\ \partial\Omega .
\end{align}

{Due to the second membership in \eqref{4-43} and formula \eqref{conormal-generalized-2c} we can apply the  Green identity (\ref{conormal-generalized-2}) with ${\tilde{\bf f}}_-=\mathbf 0$ for the exterior domain ${\Omega }_{-}$ to obtain by the first relations in (\ref{3.3-new1f})
and (\ref{3.3-new2a}),}
\begin{align}
\label{3.3-new2b}
2\langle {\mathbb E}({\bf u}^{0}),{\mathbb E}({\bf u}^{0})\rangle _{{\Omega }_{-}}
+\alpha \langle {\bf u}^{0},{\bf u}^{0}\rangle _{{\Omega }_{-}}=
-\langle {\bf t}_{\alpha }^{-}({\bf u}^{0},\pi^{0}), \gamma_{-}{\bf u}^{0}\rangle
_{_{{\partial\Omega }}}=\langle \boldsymbol{\varphi} ^0,\boldsymbol{\Phi} ^0\rangle _{_{{ \partial\Omega } }}.
\end{align}
In addition, {Green's formula} (\ref{conormal-generalized-2}) for $\alpha =0$ and the second relations in (\ref{3.3-new1f}) and (\ref{3.3-new2a}), imply that
\begin{align}
\label{3.3-new2c}
2\langle {\mathbb E}({\bf v}^{0}),{\mathbb E}({\bf v}^{0})\rangle _{{\Omega }_{+}}&=\langle {\bf
t}_{0}^{+}({\bf v}^{0},p^{0}),
\gamma_{+}{\bf v}^{0}\rangle _{_{\partial\Omega }}
=-\langle \boldsymbol{\varphi} ^0,\boldsymbol{\Phi} ^0\rangle _{_{{ \partial\Omega } }}.
\end{align}
Adding {formulas} (\ref{3.3-new2b}) and (\ref{3.3-new2c}) we deduce that
${\bf u}^{0}={\bf 0},\ \pi^{0}=0\ \mbox{ in }\ {\Omega }_{-}.$
Then the first relations in (\ref{3.3-new1f}) and (\ref{3.3-new2a}) show that $\boldsymbol{\Phi} ^0={\bf 0},\ \boldsymbol{\varphi} ^0={\bf 0}.$
Therefore, the Fredholm operator of index zero ${\mathcal T}_{\alpha; \mu }:{\mathbb X}\to {\mathbb X}$ is one-to-one and hence an isomorphism, as asserted.
Then equation (\ref{3.7operator}) has a unique solution $(\boldsymbol{\Phi} ,\boldsymbol{\varphi} )^{\top }\in {\mathbb X}$ and the layer potentials (\ref{3.3-new}) and (\ref{3.1-new}), together with relations (\ref{3.3-new-s1}) and (\ref{3.3-new-s-m}), determine a solution $\left(({\bf u}_{-},\pi _{-}),({\bf u}_{+},\pi _{+})\right)\in {\mathcal X}$ of the Poisson problem of transmission type (\ref{Poisson-Brinkman-a-Besov-Dirichlet}). By Lemma~\ref{well-posedness-Poisson-Besov-Dirichlet-L}, this solution is unique.
{Moreover, the linearity and} boundedness of the potential {operators} involved in (\ref{3.3-new-s1}), (\ref{3.3-new-s-m}), (\ref{3.3-new}) and (\ref{3.1-new}) and {of} the inverse of the isomorphism ${\mathcal T}_{\alpha ;\mu }:{\mathbb X}\to {\mathbb X}$ implies that there exists a linear continuous operator \eqref{T-1} delivering the solution. In addition, the conditions ${\bf u}_{-}\in {\mathcal H}_{\rm{div}}^{1}(\Omega _{-})^3$ and (\ref{behavior-infty-sm-ms}) imply that ${\bf u}_{-}$ vanishes at infinity in the sense of Leray (\ref{behavior-infinity-linear0}).
\end{proof}
\begin{rem}
\label{isom}
Because of the involvement of the given function $\tilde{\bf f}_-$ in the conormal derivative, the left hand side operator of transmission problem \eqref{Poisson-Brinkman-a-Besov-Dirichlet}, as written is generally nonlinear with respect to $\left(({\bf u}_{+},\pi _{+}),({\bf u}_{-},\pi _{-})\right)$.
In spite of this, the problem can be equivalently reduced to the weak form for a linear operator $($similar, e.g., to the weak formulations for Dirichlet, Neumann and mixed problems in \cite[Section 3.2]{Mikh}, \cite[Section 5.3]{Mikh-3}$)$. Moreover, by Theorem $\ref{well-posedness-Poisson-Besov-Dirichlet0}$ the operator ${\mathcal T}:{\mathcal Y}\to{\mathcal X}$ delivering the problem solution $\left(({\bf u}_{+},\pi _{+}),({\bf u}_{-},\pi _{-})\right)\in {\mathcal X}$ is linear.
\end{rem}

From Theorem~\ref{well-posedness-Poisson-Besov-Dirichlet0} we can easily deduce the more general assertion, where ${\bf u}_{\infty }$ may be nonzero.
\begin{thm}
\label{well-posedness-Poisson-Besov-Dirichlet}
Let $\Omega _{+}:={\Omega }\subset {\mathbb R}^3$ be a bounded Lipschitz domain with connected boundary $\partial \Omega $. Let ${\Omega }_{-}:={\mathbb R}^3\setminus \overline{\Omega }$ be the exterior Lipschitz domain. Let $\alpha >0, \mu >0$ and condition \eqref{positive-definite-lambda} hold.
Then for $\big({\tilde{\bf f}}_{+},{\tilde{\bf f}}_{-},{\bf h}_0,{\bf g}_{0},{\bf u}_{\infty }\big)\in {\mathcal Y}_\infty$, the transmission problem \eqref{Poisson-Brinkman-a-Besov-Dirichlet}
has a unique solution $\left({\bf u}_{+},\pi _{+},{\bf u}_{-},\pi _{-}\right)$ satisfying {condition} \eqref{udiff}.
Moreover, there exists a constant $C\equiv C({\Omega }_{+},{\Omega }_{-},{\mathcal P},\!\alpha ,{\mu)}>0$ such that
\begin{align}
\label{estimate-Poisson-Besov}
\|\left({\bf u}_{+},\pi _{+},{\bf u}_{-}-{\bf u}_{\infty },\pi _{-}\right)\|_{{\mathcal X}}&\leq C\big\|\big({\tilde{\bf f}}_{+},{\tilde{\bf f}}_{-},{\bf h}_0,{\bf g}_{0},{\bf u}_{\infty }\big)\big\|_{\mathcal Y_\infty},
\end{align}
and ${\bf u}_{-}-{\bf u}_{\infty }$ vanishes at infinity in the sense of Leray \eqref{behavior-infty-sm}.
\end{thm}
\begin{proof}
{Let us introduce} the new variables
$
{\bf v}_{+}={\bf u}_{+},\ \ {\bf v}_{-}={\bf u}_{-}-{\bf u}_{\infty},
$
and write {problem} (\ref{Poisson-Brinkman-a-Besov-Dirichlet}), \eqref{udiff} in the equivalent form
\begin{equation}
\label{Poisson-Brinkman-a-Besov-Dirichlet-sm}
\!\!\!\!\!\!\!\left\{\begin{array}{lll}
\triangle {\bf v}_{+}-\alpha {\bf v}_{+}-\nabla \pi _{+}={\tilde{\bf f}}_{+}|_{\Omega _{+}}\
\mbox{ in }\ {\Omega }_{+},\\
\triangle {\bf v}_{-}-\nabla \pi _{-}={\tilde{\bf f}}_{-}|_{\Omega _{-}}\
\mbox{ in }\ {\Omega }_{-},\\
\gamma_{+}{\bf v}_{+}-\gamma_{-}{\bf v}_{-}={\bf h}_0+{\bf u}_{\infty} \mbox{ on } \partial\Omega ,\\
{\bf t}_{\alpha }^{+}({\bf v}_{+},\pi _{+};{\tilde{\bf f}}_{+})-
\mu {\bf t}_{0}^{-}({\bf v}_{-},\pi _{-};{\tilde{\bf f}}_{-})+\frac{1}{2}{\mathcal P}\, \left(\gamma_{-}{\bf v}_{-}+\gamma_{+}{\bf v}_{+}\right)
={\bf g}_{0}-\frac{1}{2}{\mathcal P}\, {\bf u}_{\infty } \mbox{ on } \partial\Omega
\end{array}\right.
\end{equation}
for
$
\left({\bf v}_{-},\pi _{-},{\bf v}_{+},\pi _{+}\right)\in {\mathcal X},
$
already considered in Theorem~\ref{well-posedness-Poisson-Besov-Dirichlet0}. Since its right hand side includes the constant vector ${\bf u}_{\infty }$, it appears in the right-hand side of estimate \eqref{estimate-Poisson-Besov}.
\end{proof}

\begin{rem}
\label{transmission-infty}
If ${\tilde{\bf f}}_{-}={\bf 0}$ then by Lemma $\ref{infty}$ the unique solution of transmission problem \eqref{Poisson-Brinkman-a-Besov-Dirichlet} satisfies
\begin{align}
\label{behavior-infinity-linear-ms1}
({\bf u}_{-}-{\bf u}_{\infty })({\bf x})=O(|{\bf x}|^{-1}), \nabla {\bf u}_{-}({\bf x})=O(|{\bf x}|^{-2}), \pi _{-}({\bf x})=O(|{\bf x}|^{-2}) \mbox{ as } |{\bf x}|\to \infty .
\end{align}
\end{rem}

\section{Transmission problem for the Stokes and Darcy-Forchheimer-Brinkman systems in complementary Lipschitz domains {in} weighted Sobolev spaces}\label{S5}

We now consider the Poisson problem of transmission type for the Stokes and Darcy-Forchheimer-Brinkman systems in the complementary Lipschitz domains $\Omega _{\pm }$ in ${\mathbb R}^3$,
\begin{equation}
\label{Poisson-Brinkman}
\left\{\begin{array}{lll}
\triangle {\bf u}_{+}-\alpha {\bf u}_{+}-k|{\bf u}_{+}|{\bf u}_{+}-\beta ({\bf u}_{+}\cdot \nabla ){\bf u}_{+}-\nabla \pi _{+}={\tilde{\bf f}}_{+}|_{\Omega _{+}}\
\mbox{ in }\ {\Omega }_{+},\\
\triangle {\bf u}_{-}-\nabla \pi _{-}={\tilde{\bf f}}_{-}|_{\Omega _{-}}\
\mbox{ in }\ {\Omega }_{-},\\
\gamma_{+}{\bf u}_{+}-\gamma_{-}{\bf u}_{-}={\bf h}_0 \mbox{ on } \partial\Omega  ,\\
{\bf t}_{\alpha }^{+}\left({\bf u}_{+},\pi _{+};
{\tilde{\bf f}}_{+}+\mathring E_+\left(k|{\bf u}_{+}|{\bf u}_{+}+\beta ({\bf u}_{+}\cdot \nabla ){\bf u}_{+}\right)\right)
-\mu {\bf t}_{0}^{-}\left({\bf u}_{-},\pi _{-};{\tilde{\bf f}}_{-}\right)
\\
\hspace{10em}
+\frac{1}{2}{\mathcal P}\, \left(\gamma_{-}{\bf u}_{-}+\gamma_{+}{\bf u}_{+}\right)
={\bf g}_{0} \mbox{ on } \partial\Omega ,
\end{array} \right.
\end{equation}
where $\mathring E_\pm$ are the operators of extension of functions defined in $\Omega_\pm$ by zero to $\overline\Omega_\mp$. For the definition of ${\bf t}_{0}^{-}\left({\bf u}_{-},\pi _{-};{\tilde{\bf f}}_{-}\right)$ see also comment below \eqref{positive-definite-lambda}.

We assume that the given data in (\ref{Poisson-Brinkman}) {and a prescribed constant ${\bf u}_{\infty }$} belong to the space {$\mathcal Y_\infty$ defined in (\ref{space-data-inf})}, and that they are sufficiently small in a sense that will be described below. {Then we show the existence and uniqueness of the solution $\left({\bf u}_{+},\pi _{+},{\bf u}_{-},\pi _{-}\right)$ of {transmission problem} (\ref{Poisson-Brinkman}), such that $\left({\bf u}_{+},\pi _{+},{\bf u}_{-}-{\bf u}_{\infty },\pi _{-}\right)$ belongs to the weighted Sobolev space ${\mathcal X}$ introduced in \eqref{space}. Such a solution tends to the constant ${\bf u}_{\infty }\in {\mathbb R}^3$ at infinity in the sense of Leray (\ref{behavior-infty-sm}).} The proof of the existence and uniqueness result is based on the well-posedness result in Theorem \ref{well-posedness-Poisson-Besov-Dirichlet} and on a fixed point theorem.

First we show the following result that plays a main role in the proof of {the well-posedness} (cf. also \cite{K-L-W2}-\cite{K-L-W1}).
\begin{lem}
\label{embeddings}
Let $\Omega _{+}\subset {\mathbb R}^3$ be a bounded Lipschitz domain with connected boundary. Let $k,\beta \in {\mathbb R}$ be nonzero constants and
\begin{align}
\label{Newtonian-oper-Brinkman1-s}
{\mathcal I}_{k;\beta ;{\Omega }_{+}}({\bf v}):=\mathring E_+\left(k|{\bf v}|{\bf v}+\beta ({\bf v}\cdot \nabla ){\bf v}\right).
\end{align}
Then the nonlinear {operators}
\begin{align}
{{\mathcal I}_{k;\beta ;{\Omega }_{+}}:H_{\rm{div}}^{1}({\Omega }_{+})^3\to
L^\frac{3}{2}({\Omega }_{+})^3
,\quad }
{\mathcal I}_{k;\beta ;{\Omega }_{+}}:H_{\rm{div}}^{1}({\Omega }_{+})^3\to
\widetilde{H}^{-1}({\Omega }_{+})^3\nonumber
\end{align}
{are} continuous, positively homogeneous of the order $2$, and bounded, in the sense that there {exist two constants $c'_1\equiv c'_1({\Omega }_{+}{,k,\beta})>0$ and} $c_1\equiv c_1({\Omega }_{+}{,k,\beta})>0$
such that
\begin{align}\label{Ibound}
\|{\mathcal I}_{k;\beta ;{\Omega }_{+}}({\bf u})\|_{{L^\frac{3}{2}({\Omega }_{+})^3}}\leq c'_1 \|{\bf u}\|_{H^{1}({\Omega }_{+})^3}^2, \quad
\|{\mathcal I}_{k;\beta ;{\Omega }_{+}}({\bf u})\|_{{\widetilde{H}^{-1}({\Omega }_{+})^3}}&\leq c_1 \|{\bf u}\|_{H^{1}({\Omega }_{+})^3}^2.
\end{align}
Moreover, 
\begin{align}
\label{estimate-IL}
\|{\mathcal I}_{k;\beta ;{\Omega }_{+}}({\bf v})&-{\mathcal I_{k;\beta ;{\Omega }_{+}}({\bf w})}\|_{\widetilde{H}^{-1}({\Omega }_{+})^3}
\nonumber\\
&\leq {c_1}\left(\|{\bf v}\|_{H^{1}({\Omega }_{+})^3}+\|{\bf w}\|_{H^{1}({\Omega }_{+})^3}\right)
\|{\bf v}\!-\!{\bf w}\|_{H^{1}({\Omega }_{+})^3},
\ \forall\ {\bf v},{\bf w}\in H^{1}({\Omega }_{+})^3
\end{align}
with the same constant $c_1$ as in \eqref{Ibound}.
\end{lem}
\begin{proof}
Since $\Omega _{+}$ is a bounded Lipschitz domain in ${\mathbb R}^3$, the Sobolev {embedding theorem} implies that the following inclusion
\begin{align}
\label{Newtonian-D-B-F-new5-new1a-S}
H^{1}({\Omega }_{+})^3\hookrightarrow L^q({\Omega }_{+})^3
\end{align}
is continuous for any $2\leq q \leq 6$. Let $q'$ be such that $\dfrac{1}{q'}:=1-\dfrac{1}{q}$. Since {embedding} (\ref{Newtonian-D-B-F-new5-new1a-S}) has also dense range, a duality argument implies the continuity of the embedding
\begin{equation}
\label{Newtonian-D-B-F-Robin2}
L^{q'}({\Omega }_{+})^3\hookrightarrow \widetilde{H}^{-1}({\Omega }_{+})^3,
\end{equation}
in the sense that $\mathring E_+{\bf u}\in \widetilde{H}^{-1}({\Omega }_{+})^3$ for any ${\bf u}\in L^{q'}({\Omega }_{+})^3${, $\dfrac{6}{5}\le q'\le 2$}, and there exists a constant $C_q>0$ 
such that
\begin{align}
\label{Newtonian-D-B-F-Robin2-S}
\|\mathring E_+{\bf u}\|_ {\widetilde{H}^{-1}({\Omega }_{+})^3}\le C_q\|{\bf u}\|_{L^{q'}({\Omega }_{+})^3}.
\end{align}
By (\ref{Newtonian-D-B-F-new5-new1a-S}) with $q=4$ and by the H\"{o}lder inequality there exists a constant $c'\equiv c'({\Omega }_{+})>0$ such that
\begin{align}
\label{Newtonian-D-B-F-new5-new1-s}
\|\ |{\bf v}|{\bf w}\ \|_{L^2({\Omega }_{+})^3}&\leq \|{\bf v}\|_{L^4({\Omega }_{+})^3}\|{\bf w}\|_{L^4({\Omega }_{+})^3}
\leq c'\|{\bf v}\|_{H^{1}({\Omega }_{+})^3}\|{\bf w}\|_{H^{1}({\Omega }_{+})^3},\ \forall \ {\bf v},\, {\bf w}\in H^{1}({\Omega }_{+})^3.
\end{align}
Hence, we obtain that $|{\bf v}|{\bf w}\in L^2({\Omega }_{+})^3$ for all ${\bf v},\, {\bf w}\in H^{1}({\Omega }_{+})^3$. Moreover, (\ref{Newtonian-D-B-F-Robin2-S}), (\ref{Newtonian-D-B-F-new5-new1-s}) imply that the positively homogeneous operator $b$ of order $2$ defined by
$b({\bf v},{\bf w}):=\mathring E_+(|{{\bf v}|{\bf w}})$ maps $H^{1}({\Omega }_{+})^3\times H^{1}({\Omega }_{+})^3$ to $\widetilde{H}^{-1}({\Omega }_{+})^3$ and satisfies the inequality
\begin{align}
\label{Newtonian-D-B-F-new5-new1}
\|b({\bf v},{\bf w})\|_{{\widetilde{H}^{-1}({\Omega }_{+})^3}}&\leq c \|{\bf v}\|_{H^{1}({\Omega }_{+})^3}\|{\bf w}\|_{H^{1}({\Omega }_{+})^3},
\end{align}
for some constant $c\equiv c({\Omega }_{+})>0$. Thus, the following {operators are} bounded
$$b: H^{1}({\Omega }_{+})^3\times H^{1}({\Omega }_{+})^3\to L^{2}({\Omega }_{+})^3,\quad
b: H^{1}({\Omega }_{+})^3\times H^{1}({\Omega }_{+})^3\to \widetilde{H}^{-1}({\Omega }_{+})^3.$$

Moreover, the {embedding} (\ref{Newtonian-D-B-F-new5-new1a-S}) with $q=6$ and the H\"{o}lder inequality imply that there exists a constant $c_0'\equiv c_0'({\Omega }_{+})>0$ such that
\begin{align}
\label{Newtonian-D-B-F-Robin1-s}
\!\!\!\!\|({\bf v}\cdot \nabla ){\bf w}\|_{L^{\frac{3}{2}}({\Omega }_{+})^3}&\leq \|{\bf v}\|_{L^{6}({\Omega }_{+})^3}\|\nabla {\bf w}\|_{L^{2}({\Omega }_{+})^3}\leq c_0'\|{\bf v}\|_{H^{1}({\Omega }_{+})^3}
\|{\bf w}\|_{H^{1}({\Omega }_{+})^3} \ \forall \ {\bf v},\, {\bf w}\in H^{1}({\Omega }_{+})^3.
\end{align}
Hence, $({\bf v}\cdot \nabla ){\bf w}\in L^{\frac{3}{2}}({\Omega }_{+})^3$ for all ${\bf v},\, {\bf w}\in H^{1}({\Omega }_{+})^3$, and by (\ref{Newtonian-D-B-F-Robin2-S}) and (\ref{Newtonian-D-B-F-Robin1-s}) for the bi-linear operator
$n({\bf v},{\bf w}):={\mathring E_+({\bf v}\cdot \nabla ){\bf w}}$, there exists a constant $c_0\equiv c_0({\Omega }_{+})>0$ such that
\begin{align}
\label{Newtonian-D-B-F-Robin1}
\|n({\bf v},{\bf w})\|_{{\widetilde{H}^{-1}({\Omega }_{+})^3}}&\leq c_0\|{\bf v}\|_{H^{1}({\Omega }_{+})^3}\|{\bf w}\|_{H^{1}({\Omega }_{+})^3},\ \forall \ {\bf v},{\bf w}\in H^{1}({\Omega }_{+})^3.
\end{align}
Thus, the {operators
$$n: H^{1}({\Omega }_{+})^3\times H^{1}({\Omega }_{+})^3\to L^\frac{3}{2}({\Omega }_{+})^3,\ \
n: H^{1}({\Omega }_{+})^3\times H^{1}({\Omega }_{+})^3\to \widetilde{H}^{-1}({\Omega }_{+})^3$$
are bounded in the sense of \eqref{Newtonian-D-B-F-Robin1-s} and} (\ref{Newtonian-D-B-F-Robin1}).

By definition \eqref{Newtonian-oper-Brinkman1-s}, the operator ${\mathcal I}_{k;\beta ;{\Omega }_{+}}$ is positively homogeneous of the order 2 and since
${\mathcal I}_{k;\beta ;{\Omega }_{+}}({\bf v})={k b({\bf v},{\bf v})+\beta n({\bf v},{\bf v})}$,
it is also bounded in the sense of \eqref{Ibound} due to (\ref{Newtonian-D-B-F-new5-new1-s})-(\ref{Newtonian-D-B-F-Robin1}) with
\begin{align}
\label{const}
c'_1={|k|}c'+{|\beta |}c'_0 \text{ and } c_1={|k|}c+{|\beta |}c_0.
\end{align}

We now prove \eqref{estimate-IL} and thus the continuity of the operator ${\mathcal I}_{k;\beta ;{\Omega }_{+}}$. By (\ref{Newtonian-D-B-F-new5-new1}) and (\ref{Newtonian-D-B-F-Robin1}),
\begin{align}
\label{continuity}
&\|{\mathcal I}_{k;\beta ;{\Omega }_{+}}({\bf v})-{\mathcal I}_{k;\beta ;{\Omega }_{+}}({\bf w})\|_{\widetilde{H}^{-1}({\Omega }_{+})^3}
\nonumber\\
&\leq {|k|}\|\ |{\bf v}|{\bf v}-|{\bf w}|{\bf w}\ \|_{\widetilde{H}^{-1}({\Omega }_{+})^3}+{|\beta |}\|({\bf v}\cdot \nabla ){\bf v}-({\bf w}\cdot \nabla ){\bf w}\|_{\widetilde{H}^{-1}({\Omega }_{+})^3}\nonumber\\
&\leq {|k|}\|(|{\bf v}|-|{\bf w}|){\bf v}+|{\bf w}|({\bf v}-{\bf w})\|_{\widetilde{H}^{-1}({\Omega }_{+})^3}+{|\beta |}\|(({\bf v}-{\bf w})\cdot \nabla ){\bf v}+({\bf w}\cdot \nabla )({\bf v}-{\bf w})\|_{\widetilde{H}^{-1}({\Omega }_{+})^3}\nonumber\\
&\leq ({|k|}c+{|\beta |}c_0)\|{\bf v}-{\bf w}\|_{H^1({\Omega }_{+})^3}\left(\|{\bf v}\|_{H^1({\Omega }_{+})^3}+
\|{\bf w}\|_{H^1({\Omega }_{+})^3}\right),\ \forall \ {\bf v},{\bf w}\in H_{\rm{div}}^{1}({\Omega }_{+})^3,
\end{align}
which, due to the second equality in \eqref{const}, proves (\ref{estimate-IL}). Thus,
$\|{\mathcal I}_{k;\beta ;{\Omega }_{+}}({\bf v})-{\mathcal I}_{k;\beta ;{\Omega }_{+}}({\bf w})\|_{\widetilde{H}^{-1}({\Omega }_{+})^3}\to 0$ as $\|{\bf v}-{\bf w}\|_{H^1({\Omega }_{+})^3}\to 0$, and the continuity of the operator ${\mathcal I}_{k;\beta ;{\Omega }_{+}}:H_{\rm{div}}^{1}({\Omega }_{+})^3\to \widetilde{H}^{-1}({\Omega }_{+})^3$ follows.
\end{proof}

{\begin{thm}
\label{existence-N-S-D-B-F}
\label{existence-N-S-D-B-Fn} Let $\Omega _{+}:=\Omega\subset {\mathbb R}^3$ be a bounded Lipschitz
domain with connected boundary. Let ${\Omega }_{-}:={\mathbb R}^3\setminus \overline{\Omega }$. Let $\alpha >0, \mu >0$ and $k, \beta\in\mathbb R$ be given constants and ${\mathcal P}\in L^{\infty }(\partial\Omega )^{3\times 3}$ be a symmetric matrix-valued function which satisfies {nonnegativity condition} \eqref{positive-definite-lambda}. Then there exist two constants $\zeta>0$ and $\eta>0$ depending only on ${\Omega }_{+}$, ${\Omega }_{-}$, ${\mathcal P}$, $\alpha $, $k$, $\beta $ and $\mu $, with the property that for all given data $\big({\tilde{\bf f}}_{+},{\tilde{\bf f}}_{-},{\bf h}_0,{\bf g}_{0},{\bf u}_{\infty }\big)\in {\mathcal Y}_\infty$, which satisfy the condition
\begin{align}
\label{cond-smalln}
\big\|({\tilde{\bf f}}_{+},{\tilde{\bf f}}_{-},{\bf h}_0,{\bf g}_{0},{\bf u}_{\infty })\big\|_{\mathcal Y_\infty}\leq \zeta ,
\end{align}
{the transmission problem} for the Stokes and Darcy-Forchheimer-Brinkman systems \eqref{Poisson-Brinkman} has a unique solution $\left({\bf u}_{+},\pi _{+},{\bf u}_{-},\pi _{-}\right)$ {such that} $\left({\bf u}_{+},\pi _{+},{\bf u}_{-}-{\bf u}_{\infty },\pi _{-}\right)\in {\mathcal X}$ {and}
{\begin{equation}
\label{inequality-N-S-D-B-Fn}
\|({\bf u}_{+},\pi _{+},{\bf u}_{-}-{\bf u}_{\infty },\pi _{-})\|_{\mathcal X}\leq \eta .
\end{equation}}
Moreover, the solution depends continuously on the given data, and satisfies the estimate
\begin{align}
\label{estimate-D-B-F-new1-new1-D-Rn2}
\|({\bf u}_{+},\pi _{+},{\bf u}_{-}-{\bf u}_{\infty },\pi _{-})\|_{\mathcal X}\leq
C\big\|({\tilde{\bf f}}_{+},{\tilde{\bf f}}_{-},{\bf h}_0,{\bf g}_{0},{\bf u}_{\infty })\big\|_{\mathcal Y_\infty}
\end{align}
with some positive constant $C$ depending only on ${\Omega }_{+}$, ${\Omega }_{-}$, ${\mathcal P}$, $\alpha $ and $\mu $,
while ${\bf u}_{-}-{\bf u}_{\infty }$ vanishes at infinity in the sense of Leray \eqref{behavior-infty-sm}.
\end{thm}
\begin{proof}
{Let us introduce} the new variables
\begin{align}
\label{linear-new-smn}
{\bf v}_{+}={\bf u}_{+},\ \ {\bf v}_{-}={\bf u}_{-}-{\bf u}_{\infty},
\end{align}
and write {the nonlinear problem} (\ref{Poisson-Brinkman}) in the following equivalent form
\begin{equation}
\label{Poisson-Brinkman-smn}
\left\{\begin{array}{lll}
\triangle {\bf v}_{+}-\alpha {\bf v}_{+}-\nabla \pi _{+}=
{\tilde{\bf f}}_{+}|_{\Omega _{+}}+{\mathcal I}_{k;\beta ;{\Omega }_{+}}({\bf v}_+)|_{\Omega _{+}}\
\mbox{ in }\ {\Omega }_{+},\\
\triangle {\bf v}_{-}-\nabla \pi _{-}={\tilde{\bf f}}_{-}|_{\Omega _{-}}\
\mbox{ in }\ {\Omega }_{-},\\
\gamma_{+}{\bf v}_{+}-\gamma_{-}{\bf v}_{-}={\bf h}_0+{\bf u}_{\infty} \mbox{ on } \partial\Omega  ,\\
{\bf t}_{\alpha }^{+}\left({\bf v}_{+},\pi _{+};
{\tilde{\bf f}}_{+}+{\mathcal I}_{k;\beta ;{\Omega }_{+}}({\bf v}_+)\right)-
\mu {\bf t}_{0}^{-}\left({\bf v}_{-},\pi _{-};{\tilde{\bf f}}_{-}\right)
\\
\hspace{8em} +\frac{1}{2}{\mathcal P}\, \left(\gamma_{-}{\bf v}_{-}+\gamma_{+}{\bf v}_{+}\right)={\bf g}_0-\frac{1}{2}{\mathcal P}\, {\bf u}_{\infty } \mbox{ on } \partial\Omega .
\end{array}\right.
\end{equation}

Next we construct a nonlinear operator ${U}_{+}$ that maps a closed ball ${\bf B}_{\eta}$ of the space  $H_{\rm{div}}^{1}({\Omega }_{+})^3$ to ${\bf B}_{\eta}$ and is a contraction on ${\bf B}_{\eta}$. Then the unique fixed point of ${U}_{+}$ will determine a solution of the nonlinear problem (\ref{Poisson-Brinkman-smn}).

$\bullet $ {\bf Construction of a nonlinear operator ${U_{+}}$}

For a fixed $\mathbf v_+\in H^1_{\rm{div}}(\Omega _{+})^3$, let us consider the following linear Poisson problem of transmission type for the Stokes and Brinkman systems in the unknown $({\bf v}_{+}^{0},\pi_{+}^{0})$, $({\bf v}_{-}^{0},\pi_{-}^{0})$
\begin{equation}
\label{Newtonian-D-B-F-new2-D-Rn}
\left\{\begin{array}{lll}
\triangle {{\bf v}^0_+}-\alpha {{\bf v}^0_+}-\nabla {\pi^0_+}=
{\tilde{\bf f}}_{+}|_{\Omega _{+}}+{\mathcal I}_{k;\beta ;\Omega_+}({\bf v}_+)|_{\Omega_+}\
\mbox{ in }\ {\Omega }_{+},\\
\triangle {{\bf v}^0_-}-\nabla {\pi^0_-}={\tilde{\bf f}}_{-}|_{\Omega _{-}}\
\mbox{ in }\ {\Omega }_{-},\\
\gamma_{+}{{\bf v}^0_+}-\gamma_{-}{{\bf v}^0_-}
={\bf h}_0+{\bf u}_{\infty }\mbox{ on } \partial \Omega ,\\
{\bf t}_{\alpha }^{+}\left({{\bf v}^0_+},{\pi^0_+};
{\tilde{\bf f}}_{+}+{\mathcal I}_{k;\beta ;{\Omega }_{+}}({\bf v}_+)\right)-
\mu {\bf t}_{0}^{-}({{\bf v}^0_-},{\pi^0_-};{\tilde{\bf f}}_{-})\\
\hspace{8em} +\frac{1}{2}{\mathcal P}\, \left(\gamma_{-}{\bf v}^0_{-}+\gamma_{+}{\bf v}^0_{+}\right)=
{\bf g}_0-\frac{1}{2}{\mathcal P}\, {\bf u}_{\infty } \mbox{ on } \partial \Omega .
\end{array}\right.
\end{equation}
Note that ${\bf h}_0+{\bf u}_{\infty }\in H^{\frac{1}{2}}(\partial\Omega )^3$ and ${\bf g}_0-\frac{1}{2}{\mathcal P}\, {\bf u}_{\infty }\!\in H^{-\frac{1}{2}}(\partial\Omega )^3$,
while $\mathring E_+(k|{\bf v}_{+}|{\bf v}_{+}+ \beta ({\bf v}_{+}\cdot\nabla){\bf v}_{+})\in \widetilde{H}^{-1}({\Omega }_{+})^3$ by Lemma~\ref{embeddings}. Then
Theorem \ref{well-posedness-Poisson-Besov-Dirichlet} implies that problem (\ref{Newtonian-D-B-F-new2-D-Rn}) has a unique solution
\begin{align}
\label{solution-v0}
\left({{\bf v}^0_+},{\pi^0_+},{{\bf v}^0_-},{\pi^0_-}\right)&=
\left({U}_{+}({\bf v}_+),{P}_{+}({\bf v}_+),{U}_{-}({\bf v}_+),{P}_{-}({\bf v}_+)\right)\nonumber\\
&:={\mathcal T}\left({\tilde{\bf f}}_{+}|_{\Omega _{+}}+{\mathcal I}_{k;\beta ;\Omega_+}({\bf v}_+)|_{\Omega_+},\ {\tilde{\bf f}}_{-}|_{\Omega _{-}},\ {\bf h}_0+{\bf u}_{\infty },\ {\bf g}_0-\frac{1}{2}{\mathcal P}\, {\bf u}_{\infty }
\right)\in {\mathcal X},
\end{align}
where the linear continuous operator ${\mathcal T}:\mathcal Y\to\mathcal X$ has been introduced in Theorem \ref{well-posedness-Poisson-Besov-Dirichlet0} (see \eqref{T-1}), and ${\mathcal X}$ and ${\mathcal Y}$ are the spaces defined in \eqref{space}, (\ref{space-data}).
Hence, due to Lemma~\ref{embeddings} and Theorem~\ref{well-posedness-Poisson-Besov-Dirichlet}, for fixed ${\tilde{\bf f}}_\pm, {\bf h}_0, {\bf g}_0, {\bf u}_{\infty }$, the nonlinear operators defined by \eqref{solution-v0},
\begin{align}
\label{Newtonian-D-B-F-new3n}
&(U_+,P_+,U_-P_-):H_{\rm{div}}^{1}({\Omega }_{+})^3
\to {\mathcal X}
\end{align} are continuous and bounded, in the sense that there exists a constant $c_*\equiv c_*({\Omega }_{+},{\Omega }_{-},{\mathcal P},\alpha ,\mu )>0$ such that
\begin{align}
\label{estimate-D-B-F-new1-new1-D-Rn}
\big\|\big({U}_{+}({\bf v}_+),&{P}_{+}({\bf v}_+),{U}_{-}({\bf v}_+),{P}_{-}({\bf v}_+)\big)\big\|_{\mathcal X}
\leq c_*\Big\|\left({\tilde{\bf f}}_{+}+\mathcal I_{k;\beta;\Omega_+}({\bf v}_+)
,{\tilde{\bf f}}_{-},{\bf h}_0,{\bf g}_0,{\bf u}_{\infty }\right)\Big\|_{\mathcal Y_\infty}\nonumber\\
&\leq c_*\left(\big\|\big({\tilde{\bf f}}_{+},{\tilde{\bf f}}_{-},{\bf h}_0,{\bf g}_0,{\bf u}_{\infty }\big)\big\|_{\mathcal Y_\infty}+\|\mathcal I_{k;\beta;\Omega_+}({\bf v}_+)\|_{\widetilde{H}^{-1}({\Omega }_{+})^3}\right)\nonumber\\
&\leq c_*\big\|\big({\tilde{\bf f}}_{+},{\tilde{\bf f}}_{-},{\bf h}_0,{\bf g}_0,{\bf u}_{\infty }\big)\big\|_{\mathcal Y_\infty}
+c_{*}c_1{\|{\bf v}_{+}\|_{H^{1}({\Omega }_{+})^3}^2},
\ \ \forall\ {{\bf v}_{+}\in H_{\rm{div}}^{1}({\Omega }_{+})^3},
\end{align}
where $c_1\equiv c_1({\Omega }_{+}{,k,\beta})>0$ is the constant of Lemma~\ref{embeddings}. In addition, in view of (\ref{solution-v0}),  we obtain
\begin{equation}
\label{Newtonian-D-B-F-new4n}
\left\{\begin{array}{lll}
\triangle {U_+({\bf v}_{+})}\!-\!\alpha {U_+({\bf v}_{+})}\!-\!\nabla {P_+({\bf v}_{+})}
={\tilde{\bf f}}_{+}|_{\Omega _{+}}\!\!+{\mathcal I}_{k;\beta ;\Omega_+}({\bf v}_+)|_{\Omega_+}
\mbox{ in }\ {\Omega }_{+},\\
\triangle {U_-({\bf v}_{+})}-\nabla {P_-({\bf v}_{+})}
={\tilde{\bf f}}_{-}|_{\Omega _{-}}
\mbox{ in }\ {\Omega }_{-},\\
\gamma_{+}\left({U_+({\bf v}_{+})}\right)-\gamma_{-}\left({U_-({\bf v}_{+})}\right)={\bf h}_0+{\bf u}_{\infty }\in H^{\frac{1}{2}}(\partial\Omega )^3,\\
{\bf t}_{\alpha }^{+}\left({U_+({\bf v}_{+})},{P_+({\bf v}_{+})};{\tilde{\bf f}}_{+}+{{\mathcal I}_{k;\beta ;{\Omega }_{+}}({\bf v}_{+})}\right)
-\mu {\bf t}_{0}^{-}\left({U_-({\bf v}_{+})},
{P_-({\bf v}_{+})};{\tilde{\bf f}}_{-}\right)\\
\qquad\qquad +\frac{1}{2}{\mathcal P}\, \left(\gamma_{+}\left({U_+({\bf v}_{+})}\right)+\gamma_{-}{ U_-({\bf v}_{+})}\right)
={\bf g}_0-\frac{1}{2}{\mathcal P}\, {\bf u}_{\infty }\in H^{-\frac{1}{2}}(\partial\Omega )^3.
\end{array}\right.
\end{equation}
Therefore, if we can show that the nonlinear operator $U_{+}$ has a fixed point ${\bf v}_{+}\in H_{\rm{div}}^{1}({\Omega }_{+})^3$, then it solves the equation ${U_{+}({\bf v}_{+})={\bf v}_{+}}$
and together {with $\mathbf v_-=U_-({\bf v}_{+})$, $\pi_\pm =P_\pm({\bf v}_{+})$} it defines a solution of {nonlinear transmission problem} (\ref{Poisson-Brinkman-smn}) in the space ${\mathcal X}$.
In order to show that ${U}_{+}$ has a fixed point, we prove that $U_{+}$ maps some closed ball $\mathbf{B}_{\eta }$ in $H_{\rm{div}}^{1}({\Omega }_{+})^3$ to the same closed ball $\mathbf{B}_{\eta }$, and that ${U}_{+}$ is a contraction on $\mathbf{B}_{\eta }$. Let us introduce the constants
\begin{align}
\label{Newtonian-D-B-F-new9n}
&\zeta:=\frac{3}{16c_1c_*^2}>0,\ \ \eta :=\frac{1}{4c_1c_*}>0
\end{align}
(see also \cite[Lemma 29]{Choe-Kim}, \cite[Theorem 5.1]{Russo-Tartaglione-2}), and the closed ball
\begin{align}
\label{gamma-0n}
\mathbf{B}_{\eta }:=\left\{{\bf v}_{+}\in H_{\rm{div}}^{1}({\Omega }_{+})^3:\|{\bf v}_{+}\|_{H^{1}({\Omega }_{+})^3}\leq \eta\right\}.
\end{align}
Also we assume that the given data satisfy the condition
\begin{align}
\label{cond-small-sm-msn}
\big\|\big({\tilde{\bf f}}_{+},{\tilde{\bf f}}_{-},{\bf h}_0,{\bf g}_0,{\bf u}_{\infty }\big)\big\|_{\mathcal Y_\infty}\leq \zeta .
\end{align}
Then by (\ref{estimate-D-B-F-new1-new1-D-Rn}), \eqref{Newtonian-D-B-F-new9n}-(\ref{cond-small-sm-msn}) we deduce the inequality
\begin{align}
\label{Newtonian-D-B-F-new9-new-crackn}
\|\left({U}_{+}({\bf v}_+),
{P}_{+}({\bf v}_+),
{U}_{-}({\bf v}_+),
{P}_{-}({\bf v}_+)\right)\|_{\mathcal X}
\leq {\frac{1}{4c_1c_*}}
= \eta ,\ \forall \ {\bf v}_{+}\in \mathbf{B}_{\eta }.
\end{align}
Inequality (\ref{Newtonian-D-B-F-new9-new-crackn}) shows that ${\|{U}_{+}({\bf v}_{+})\|_{H^{1}({\Omega }_{+})^3}\leq \eta,\ \forall \ {\bf v}_{+}\in \mathbf{B}_{\eta }},$ i.e., ${U}_{+}$ maps $\mathbf{B}_{\eta }$ to $\mathbf{B}_{\eta }$, as asserted.

Next we show that the map ${U}_{+}$ is Lipschitz continuous on $\mathbf{B}_\eta$.
Indeed, by using expression (\ref{solution-v0}) for ${U_{+}}$, given in terms of the linear and continuous operator ${\mathcal T}$ and fixed data $\big({\tilde{\bf f}}_{+},{\tilde{\bf f}}_{-},{\bf h}_0,{\bf g}_0,{\bf u}_{\infty }\big)$, we have for two arbitrary functions ${\bf v}_+,{\bf w}_+\in \mathbf{B}_\eta$,
\begin{align}
\label{4.30n}
\|{U}_{+}({\bf v}_{+})-{U}_{+}({\bf w}_{+})
&\|_{{H}^{1}({\Omega }_{+})^3}
\leq c_*\|{{\mathcal I}_{k;\beta ;{\Omega }_{+}}({\bf v}_{+})}-{{\mathcal I}_{k;\beta ;{\Omega }_{+}}({\bf w}_{+})}\|_{\widetilde{H}^{-1}({\Omega }_{+})^3}\nonumber
\\
& \leq c_*c_1\left(\|{\bf v}_+\|_{H^{1}({\Omega }_{+})^3}+\|{\bf w}_+\|_{H^{1}({\Omega }_{+})^3}\right)
\|{\bf v}_+\!-\!{\bf w}_+\|_{H^{1}({\Omega }_{+})^3}\nonumber
\\
& \leq 2\eta c_*c_1\|{\bf v}_+\!-\!{\bf w}_+\|_{H^{1}({\Omega }_{+})^3}
=\frac{1}{2}\|{\bf v}_+\!-\!{\bf w}_+\|_{H^{1}({\Omega }_{+})^3},
\ \forall\ {\bf v}_+,{\bf w}_+\in \mathbf{B}_\eta,
\end{align}
where the first inequality is implied by the continuity of the operator ${\mathcal T}$, while the second one follows from inequality \eqref{estimate-IL} in Lemma~\ref{embeddings}, and $c_*$, $c_1$ are the constants from \eqref{estimate-D-B-F-new1-new1-D-Rn}.
Hence, ${U_{+}}:\mathbf{B}_{\eta }\to \mathbf{B}_{\eta }$ {is a contraction}.

{$\bullet $ {\bf The existence of a solution of the {nonlinear problem} (\ref{Poisson-Brinkman})}}

The Banach-Caccioppoli fixed point theorem implies that there exists a unique fixed point {${\bf v}_{+}\in \mathbf{B}_{\eta }$} of ${U_{+}}$, i.e., ${U_{+}({\bf v}_{+})={\bf v}_{+}}$.
Then ${\bf v}_{+}$ together with the functions ${\mathbf v}_-=U_-({\bf v}_{+})$ and $\pi_\pm =P_\pm({\bf v}_{+})$ given by (\ref{solution-v0}), determine a solution of the {nonlinear problem} (\ref{Poisson-Brinkman-smn}) in the space ${\mathcal X}=\big({H_{\rm{div}}^{1}({\Omega }_{+})^3}\times L^2({\Omega }_{+})\big)\times \big({{\mathcal H}_{\rm{div}}^{1}({\Omega }_{-})^3}\times L^2({\Omega }_{-})\big)$. In addition, since ${\bf v}_{-}\in {\mathcal H}_{\rm{div}}^{1}({\Omega }_{-})^3$, ${\bf v}_{-}$ vanishes at infinity in the sense of Leray (\ref{behavior-infty-sm-ms}). Further, the fields $\left({\bf u}_{+},\pi _+,{\bf u}_{-},\pi _{-}\right)$, where ${{\bf u}_{+}}$ and ${{\bf u}_{-}}$ are given by (\ref{linear-new-smn}), determine a solution of {the nonlinear transmission problem of Poisson type} (\ref{Poisson-Brinkman}) satisfying $\left({\bf u}_{+},\pi _{+},{\bf u}_{-}-{\bf u}_{\infty },\pi _{-}\right)\in {\mathcal X}$.
In addition, by (\ref{Newtonian-D-B-F-new9-new-crackn}) and by the expressions ${\bf u}_{-}-{\bf u}_{\infty }={\bf v}_{-}=U_{-}({\bf v}_{+})$ and $\pi_\pm =P_\pm({\bf v}_{+})$, we deduce that such a solution satisfies inequality (\ref{inequality-N-S-D-B-Fn}).
Also, ${\bf u}_{-}-{\bf u}_{\infty }$ vanishes at infinity in the sense of Leray, i.e., ${\bf u}_{-}-{\bf u}_{\infty }$ satisfies {condition} (\ref{behavior-infty-sm}).

Since the solution ${\bf v}_{+}$ belongs to the closed ball $\mathbf{B}_{\eta }$, we obtain that $c_{*}c_1\|{\bf v}_{+}\|_{H^{1}({\Omega }_{+})^3}\le c_{*}c_1\eta=\dfrac{1}{4}.$
Then inequality \eqref{estimate-D-B-F-new1-new1-D-Rn} gives
\begin{align}
\label{estimate-D-B-F-new1-new1-D-Rn1}
\|{\bf v}_+\|_{H^1(\Omega _{+})^3}&+\|\pi_{+}\|_{L^2(\Omega _{+})}+\|{\bf v}_-\|_{{\mathcal H}^1(\Omega _{-})^3}+\|\pi_-\|_{L^2(\Omega _{-})}\nonumber\\
&=\|({\bf v}_{+},\pi _{+},{\bf v}_{-},\pi _{-})\|_{\mathcal X}\leq c_*\big\|\big({\tilde{\bf f}}_{+},{\tilde{\bf f}}_{-},{\bf h}_0,{\bf g}_{0},{\bf u}_{\infty }\big)\big\|_{\mathcal Y_\infty}
+\frac{1}{4}\|{\bf v}_{+}\|_{H^{1}({\Omega }_{+})^3},
\end{align}
which implies
$
\|{\bf v}_+\|_{H^1(\Omega _{+})^3}\leq
\dfrac{4}{3}c_*\big\|\big({\tilde{\bf f}}_{+},{\tilde{\bf f}}_{-},{\bf h}_0,{\bf g}_{0},{\bf u}_{\infty }\big)\big\|_{\mathcal Y_\infty}.
$
Substituting this back to \eqref{estimate-D-B-F-new1-new1-D-Rn1} and using
\eqref{linear-new-smn}, we obtain the desired estimate \eqref{estimate-D-B-F-new1-new1-D-Rn2} with $C=\dfrac{4}{3}c_*$.

{$\bullet $ {\bf The uniqueness of solution of the {nonlinear problem} (\ref{Poisson-Brinkman})}}

We now show the uniqueness of the solution $\left({\bf u}_{+},\pi _{+},{\bf u}_{-},\pi _{-}\right)$ of the {nonlinear transmission problem} (\ref{Poisson-Brinkman}), satisfying $\left({\bf u}_{+},\pi _{+},{\bf u}_{-}-{\bf u}_{\infty },\pi _{-}\right)\in {\mathcal X}$ and {inequality} (\ref{inequality-N-S-D-B-Fn}).
Assume that $\left({\bf u}_{+}',\pi_{+}',{\bf u}_{-}',\pi_{-}'\right)$ is another solution of problem (\ref{Poisson-Brinkman}), such that $\left({\bf u}_{+}',{\pi}_{+}',{\bf u}_{-}'-{\bf u}_{\infty },{\pi}_{-}'\right)\in {\mathcal X}$, and that such a solution satisfies {inequality} (\ref{inequality-N-S-D-B-Fn}).
Hence for $({\bf v}'_{+},{\bf v}'_{-})=({\bf u}_{+}',{\bf u}_{-}'-{\bf u}_{\infty })$, we obtain ${{\bf v}'_{+}}\in \mathbf{B}_{\eta }$.
Since ${{\bf v}'_{+}}\in \mathbf{B}_{\eta }$, we obtain that ${U_{+}({\bf v}'_{+})}\in \mathbf{B}_{\eta },$ where $\left({U_+}({\bf v}'_{+}),P_+({\bf v}'_{+}),{U_-}({\bf v}'_{+}),P_-({\bf v}'_{+})\right)$ are given by (\ref{solution-v0}) and satisfy {problem}
(\ref{Newtonian-D-B-F-new4n}) {with ${\bf v}_{+}$ replaced by ${\bf v}'_{+}$}.
Then by (\ref{Poisson-Brinkman-smn}) (written in terms of $\left({\bf v}'_{+},{\pi}_{+}',{\bf v}'_{-},{\pi}_{-}'\right)$) and (\ref{Newtonian-D-B-F-new4n}) we obtain the linear transmission problem
\begin{equation}
\label{Newtonian-D-B-F-new4-unique1n}
\left\{\begin{array}{lll}
(\triangle -\alpha {\mathbb I})\left({U_+({\bf v}'_{+})}-{\bf v}'_{+}\right)
-\nabla ({P_+({\bf v}'_{+})}-{\pi}_{+}')
={\bf 0}\ \mbox{ in }\ {\Omega }_{+},\\
\triangle \left({U_{-}({\bf v}'_{+})}-{\bf v}'_{-}\right)
-\nabla ({P_-({\bf v}'_{+})}-{\pi}_{-}')={\bf 0}\
\mbox{ in }\ {\Omega }_{-},\\
{\rm{Tr}}^{+}\left({U_+({\bf v}'_{+})}
-{\bf v}'_{+}\right)-{\rm{Tr}}^{-}\left({U_-({\bf v}'_{+})}
-{\bf v}'_{-}\right)={\bf 0}\ \mbox{ on }\ \partial\Omega  ,\\
{\bf t}_{\alpha }^{+}\left({U_+({\bf v}'_{+})}-{\bf v}'_{+},{P_+({\bf v}'_{+})}
-{\pi}_{+}'\right)-\mu {\bf t}_{0}^{-}\left({U_-({\bf v}'_{+})}-{\bf v}'_{-},{P_-({\bf v}'_{+})}-{\pi}_{-}'\!\right)\\
\hspace{6em}  +\frac{1}{2}{\mathcal P}\, \left({\rm{Tr}}^{+}
\left({U_+({\bf v}'_{+})}-{\bf v}'_{+}\right)+{\rm{Tr}}^{-}
\left({U_-({\bf v}'_{+})}-{\bf v}'_{-}\right)\right)
={\bf 0}\ \mbox{ on }\ \partial \Omega , 
\end{array}\right.\nonumber
\end{equation}
which, in view of the well-posedness result in Theorem \ref{well-posedness-Poisson-Besov-Dirichlet}, has only the trivial solution in the space ${\mathcal X}$, i.e.,
$\left({U_+}({\bf v}'_{+}),{U_-}({\bf v}'_{+})\right)=({\bf v}'_{+},{\bf v}'_{-})$ and ${P_\pm({\bf v}'_{+})}={\pi}_{\pm }'$. Therefore, ${{\bf v}'_{+}}$ is a fixed point of ${U_{+}}$. However, ${ U_{+}}:\mathbf{B}_{\eta }\to \mathbf{B}_{\eta }$ given by (\ref{solution-v0}) is a contraction, and, thus, it has a unique fixed point {${\bf v}_{+}$} in $\mathbf{B}_{\eta }$. Consequently, ${{\bf v}'_{+}={\bf v}_{+}}$. In addition, ${{\bf v}'_{-}={\bf v}_{-}}$ and also $\pi'_\pm=\pi_\pm$.

$\bullet $ {\bf The continuity of the solution with respect to the given data}

Since the unique solution $\left({\bf u}_{+},\pi _{+},{\bf u}_{-},\pi _{-}\right)$ of {nonlinear problem} (\ref{Poisson-Brinkman}) is expressed in terms of the unique fixed point of the contraction $U_{+}:\mathbf{B}_{\eta }\to \mathbf{B}_{\eta }$, which is a continuous mapping with respect to the data $\big({\tilde{\bf f}}_{+},{\tilde{\bf f}}_{-},{\bf h}_0,{\bf g}_{0}\big)\in {\mathcal Y}$ due to the continuity of the operator ${\mathcal T}$, we deduce that
$\left({\bf u}_{+},\pi _{+},{\bf u}_{-},\pi _{-}\right)$ depends continuously on the given data (cf., e.g., \cite[Chapter XVI, \S 1, Theorem 3]{Ka-Ak}).
\end{proof}

{\begin{rem}
\label{transmission-infty-nonlinn}
Note that in Theorem~$\ref{existence-N-S-D-B-F}$ we do not need to assume the conditions $k>0$, $\beta>0$ on the coefficients of the nonlinear terms, as is normally done in a variational approach to such problems. 
\end{rem}}

\appendix
{
{ 
}
\section{Behaviour at infinity and properties of potentials}
}

{The following assertion} {characterizes} the behavior at infinity of functions in the space ${\mathcal H}^{1}(\Omega _{-})$ (cf. \cite[Lemma 2.1 {and} Definition 2.2]{Am-Ra}).
\begin{lem}
\label{behavior-infinity-s}
Let $\Omega \subset {\mathbb R}^3$ be a bounded Lipschitz domain with connected boundary. Let $\Omega _{-}:={\mathbb R}^3\setminus \overline{\Omega }$ be the corresponding exterior Lipschitz domain. Assume that $u\in {\mathcal D}'(\Omega _{-})$ and that $\nabla u\in L^2(\Omega _{-})^3$. Then there exists a unique constant $u_{\infty }\in {\mathbb R}$ such that $u-u_{\infty }\in {\mathcal H}^{1}(\Omega _{-})$. Moreover,
\begin{align*}
{u_{\infty }=\frac{1}{4\pi }\lim _{r\to \infty }\int _{S^2}u(r{\bf y})d\sigma _{\bf y},}
\end{align*}
where $S^2$ is the unit sphere in ${\mathbb R}^3$. In addition, $u-u_{\infty }\in L^6(\Omega _{-})$, and
\begin{align}
\label{sm-sm}
\|u-u_{\infty }\|_{L^6(\Omega _{-})}\leq C\|\nabla u\|_{L^2(\Omega _{-})^3},\ \ {\lim_{r\to \infty }\int _{S^2}|u(r{\bf y})-u_{\infty }|d\sigma _{\bf y}=0.}
\end{align}
\end{lem}

Let us show the behavior at infinity of a solution of the homogeneous Stokes system in the weighted Sobolev space ${\mathcal H}^{1}(\Omega _{-})^3\times L^{2}(\Omega _{-})$.
\begin{lem}
\label{infty}
Let $\Omega :=\Omega _{+}\subset {\mathbb R}^3$ be a bounded Lipschitz domain with connected boundary and
$\Omega _{-}:={\mathbb R}^3\setminus \overline{\Omega }$. {If the pair $({\bf v},p)\in {\mathcal H}^{1}(\Omega _{-})^3\times L^{2}(\Omega _{-})$ satisfies the equations
\begin{equation}
\triangle {\bf v}-\nabla p ={\bf 0},\ \
{\rm{div}}\ {\bf v}=0\ \mbox{ in }\ \Omega _{-},
\end{equation}
then }
\begin{align}
\label{behavior-Stokes}
{{\bf v}({\bf x})=O(|{\bf x}|^{-1}),\ \nabla {\bf v}({\bf x})=O(|{\bf x}|^{-2}),\ p({\bf x})=O(|{\bf x}|^{-2})\ \mbox{ as }\ |{\bf x}|\to \infty .}
\end{align}
\end{lem}
\begin{proof}
First, note that the exterior Dirichlet problem for the Stokes system
\begin{equation}
\label{Dirichlet-Stokes-ext}
\left\{\begin{array}{lll}
\triangle {\bf u}-\nabla \pi ={\bf 0},\ {\rm{div}}\ {\bf u}=0\ \mbox{ in }\ \Omega _{-},\\
\gamma_{-}{\bf u}=\gamma_{-}{\bf v}\in H^{\frac{1}{2}}(\partial \Omega )^3,\\
{\bf u}({\bf x})=O(|{\bf x}|^{-1}),\ \nabla {\bf u}({\bf x})=O(|{\bf x}|^{-2}),\ \pi ({\bf x})=O(|{\bf x}|^{-2})\ \mbox{ as }\ |{\bf x}|\to \infty ,
\end{array}\right.
\end{equation}
has a unique solution $({\bf u},\pi )\in H_{{\rm{loc}}}^{1}(\overline{\Omega _{-}})^3\times L_{{\rm{loc}}}^{2}(\overline{\Omega _{-}})$ (see, e.g., \cite[Theorem 9.2.4]{M-W}). We now show that $({\bf u},\pi )\in {\mathcal H}^{1}({\Omega _{-}})^3\times L^{2}(\Omega _{-})$. Indeed, by the first asymptotic condition in (\ref{Dirichlet-Stokes-ext}), there exist two constants $M,R_0>0$ such that $|{\bf x}|^2|{\bf u}({\bf x})|^2\leq M$ if $|{\bf x}|\geq R_0$. Moreover, we can assume that $\overline{\Omega }\subset B_{R_0}$, where $B_{R_0}$ is the ball in ${\mathbb R}^3$ of radius $R_0$ and center $0$ (assumed to be a point in $\Omega $). Then by exploiting the inequalities $\rho \geq 1$, $\rho \geq |{\bf x}|$ and $\displaystyle\int _{{\mathbb R}^3\setminus \overline{B_{R_0}}}\frac{1}{|{\bf x}|^4}d{\bf x}<\infty $, we obtain that
\begin{align}
\int _{\Omega _{-}}\frac{1}{\rho ^2}|{\bf u}({\bf x})|^2d{\bf x}&=\int _{\Omega _{-}\bigcap B_{R_0}}\frac{1}{\rho ^2}|{\bf u}({\bf x})|^2d{\bf x}+\int _{{\mathbb R}^3\setminus \overline{B_{R_0}}}\frac{1}{\rho ^2}|{\bf u}({\bf x})|^2d{\bf x},
\\
&\leq \int _{\Omega _{-}\bigcap B_{R_0}}|{\bf u}({\bf x})|^2d{\bf x}+
M\int _{{\mathbb R}^3\setminus \overline{B_{R_0}}}\frac{1}{|{\bf x}|^4}d{\bf x}<\infty .\nonumber
\end{align}
Hence, ${\bf u}\in L^2(\rho^{-1};{\Omega }_{-})^3$. By exploiting the last two asymptotic assumptions in (\ref{Dirichlet-Stokes-ext}) and a similar argument as above, we deduce that $\nabla {\bf u}\in L^2(\Omega _{-})^3$ and $\pi \in L^2(\Omega _{-})$. Therefore, $({\bf u},\pi )\in {\mathcal H}^{1}(\Omega _{-})^3\times L^2(\Omega _{-})$. In addition, $({\bf u},\pi )$ is a solution of the exterior Dirichlet problem
\begin{equation}
\label{Dirichlet-Stokes-ext-1}
\left\{\begin{array}{lll}
\triangle {\bf w}-\nabla q={\bf 0},\ {\rm{div}}\ {\bf w}=0\ \mbox{ in }\ \Omega _{-},\\
\gamma_{-}({\bf w})=\gamma_{-}({\bf v})\in H^{\frac{1}{2}}(\partial \Omega )^3
\end{array}\right.
\end{equation}
in the space ${\mathcal H}^{1}(\Omega _{-})^3\times L^2(\Omega _{-})$. The pair $({\bf v},p)\in {\mathcal H} ^{1}(\Omega _{-})^3\times L^2(\Omega _{-})$ satisfies the same problem (\ref{Dirichlet-Stokes-ext-1}). Then the uniqueness of the solution of the exterior Dirichlet problem for the Stokes system (\ref{Dirichlet-Stokes-ext-1}) in the space ${\mathcal H}^{1}(\Omega _{-})^3\times L^2(\Omega _{-})$ (cf. \cite[Theorem 3.4]{Gi-Se}) implies that ${\bf v}={\bf u}$ and $p=\pi $ in $\Omega _{-}$. Finally, the asymptotic relations in (\ref{Dirichlet-Stokes-ext}) satisfied by ${\bf u}$ and $\pi $ yield {formulas} (\ref{behavior-Stokes}), as asserted.
\end{proof}


{For $\alpha=0$, and an exterior domain $\Omega_-$ {(or the whole $\mathbb R^3$)}, the weighted Sobolev spaces are more suitable than the standard Sobolev spaces{, while in $\Omega_+$ the standard Sobolev spaces {coincide with the weighted ones}. In these cases} we have the following mapping properties similar to the ones for the potentials of the Laplace operator available, e.g., in \cite[Theorem 4.1]{Ch-Mi-Na-3} and \cite[Theorem 3.2]{Lang-Mendez}.}
\begin{lem}
\label{Newtonian-weight}
The following Newtonian {velocity and pressure potential operators} for the Stokes system given by \eqref{NoT} and \eqref{QoT} are linear and continuous,
\begin{align}
\label{Newtonian-Stokes-R}
{\boldsymbol{\mathcal N}}_{\mathbb R^3}:{\mathcal H}^{-1}(\mathbb R^3)^3\to {\mathcal H}^{1}(\mathbb R^3)^3,&
\quad {\mathcal Q}_{\mathbb R^3}:{\mathcal H}^{-1}(\mathbb R^3)^3\to L^2(\mathbb R^3),\\
\label{Newtonian-Stokes-S}
{\boldsymbol{\mathcal N}}_{{\Omega }_\pm}:\widetilde{\mathcal H}^{-1}({\Omega }_\pm)^3\to {\mathcal H}^{1}({\Omega }_\pm)^3,&
\quad {\mathcal Q}_{{\Omega }_\pm}:\widetilde{\mathcal H}^{-1}({\Omega }_\pm)^3\to L^2({\Omega }_\pm),\\
\label{Newtonian-Stokes+}
{\boldsymbol{\mathcal N}}_{\Omega_+}:\widetilde{H}^{-1}(\Omega_+)^3\to {H}^{1}(\Omega_+)^3,&\quad
{\mathcal Q}_{\Omega_+}:\widetilde{H}^{-1}(\Omega_+)^3\to L^{2}(\Omega_+).
\end{align}
\end{lem}
\begin{proof}
The continuity of the operators in \eqref{Newtonian-Stokes-R} follow from \cite[Propositions 4.6, 4.7]{Al-Am-1}. Since $\widetilde{\mathcal H}^{-1}(\Omega_\pm)^3$ are subspaces of ${\mathcal H}^{-1}({\mathbb R}^3)^3$ {and ${\boldsymbol{\mathcal N}}_{\Omega_\pm}=r_{\Omega_\pm}{\boldsymbol{\mathcal N}}_{{\mathbb R}^3}$,
${\mathcal Q}_{\Omega_\pm}=r_{\Omega_\pm}{\mathcal Q}_{{\mathbb R}^3}$, we deduce the continuity of the} operators in \eqref{Newtonian-Stokes-S}. On the other hand, since $\widetilde{\mathcal H}^{-1}(\Omega_+)^3$ coincides with $\widetilde{H}^{-1}(\Omega_+)^3$ and ${\mathcal H}^{1}(\Omega_+)^3$ coincides with ${H}^{1}(\Omega_+)^3$, with equivalent norms, {\eqref{Newtonian-Stokes-S} also} implies the well-known continuity of {the operators} in \eqref{Newtonian-Stokes+}.
\end{proof}

{In the following lemma we} collect the main properties of layer potentials for the Stokes system defined in \eqref{58}-\eqref{double-layer-principal-value}.
\begin{lem}
\label{layer-potential-properties-Stokes}
Let
$\Omega :=\Omega _{+}\subset {\mathbb R}^3$ be a bounded Lipschitz domain with connected boundary $\partial\Omega $. Let $\Omega _{-}:={\mathbb R}^3\setminus \overline{{\Omega }}$.
\begin{itemize}
\item[$(i)$] The following operators are linear and bounded,
\begin{align}
\label{ss-s1-Stokes}
&\left({\bf V}_{\partial\Omega }\right)|_{\Omega _{+}}\!:\!H^{-\frac{1}{2}}(\partial\Omega )^3\to H^{1}(\Omega _{+})^3,\
\left({\mathcal Q}_{\partial\Omega }^s\right)|_{\Omega _{+}}\!:\!H^{-\frac{1}{2}}(\partial\Omega )^3\!\to L^2(\Omega _{+}),
\\
\label{ds-s1-Stokes}
&\left({\bf W}_{\partial\Omega }\right)|_{\Omega _{+}}:H^{\frac{1}{2}}(\partial\Omega )^3\to H^{1}(\Omega _{+})^3,\
\big({\mathcal Q}_{\partial\Omega }^d\big)|_{\Omega _{+}}\!:\!H^{\frac{1}{2}}(\partial\Omega )^3\to L^2(\Omega _{+}),\\
\label{ss-s1-ms}
&{{\left({\bf V}_{\partial\Omega }\right)|_{\Omega_{-}}:H^{-\frac{1}{2}}(\partial\Omega )^3\to {\mathcal H}^{1}(\Omega _{-})^3,}\ {\left({\mathcal Q}_{\partial\Omega }^s\right)|_{\Omega _{-}}:H^{-\frac{1}{2}}(\partial\Omega )^3\to L^{2}(\Omega _{-})},}
\\
\label{ds-s1-ms}
&\left({\bf W}_{\partial\Omega }\right)|_{\Omega_{-}}:H^{\frac{1}{2}}(\partial\Omega )^3\to {\mathcal H}^{1}(\Omega _{-})^3,\
{\big({\mathcal Q}_{\partial\Omega }^d\big)|_{\Omega _{-}}:H^{\frac{1}{2}}(\partial\Omega )^3\to {L^{2}(\Omega _-)}}\\
\label{ds-s1-msH}
&\left({\bf W}_{\partial\Omega }\right)|_{\Omega_{-}}:H^{\frac{1}{2}}(\partial\Omega )^3\to {H}^{1}(\Omega _{-})^3.
\end{align}
\item[$(ii)$] Let ${\bf h}\in H^{\frac{1}{2}}(\partial\Omega )^3$ and ${\bf g}\in
H^{-\frac{1}{2}}(\partial\Omega )^3$.
Then the following relations hold a.e. on $\partial\Omega $,
\begin{align}
\label{68-Stokes}
&\gamma_+\big({\bf V}_{\partial\Omega }{\bf g}\big)=\gamma_{-}\big({\bf V}_{{\partial\Omega } }{\bf g}\big)=:{\boldsymbol{\mathcal V}}_{\partial\Omega }{\bf g},\\
\label{68-s1-Stokes}
&
\frac{1}{2}{\bf h}+\gamma_{+}({\bf W}_{\partial\Omega }{\bf h})=-\frac{1}{2}{\bf h}+\gamma_{-}({\bf W}_{\partial\Omega }{\bf h})=:{\bf K}_{{\partial\Omega }}{\bf h},
\\
\label{70aaa-Stokes}
&
-\frac{1}{2}{\bf g}+{{\bf t}_0^{+}}\left({\bf V}_{\partial\Omega }{\bf g},{\mathcal Q}_{\partial\Omega }^s{\bf g}\right)=\frac{1}{2}{\bf g}+{{\bf t}_0^{-}}\left({\bf V}_{\partial\Omega }{\bf g},{\mathcal Q}_{\partial\Omega }^s{\bf g}\right)=:{\bf K}_{{\partial\Omega } }^*{\bf g},
\\
\label{70aaaa-Stokes}
&{{\bf t}_0^{+}}\big({\bf W}_{\partial\Omega }{\bf h},{\mathcal Q}_{\partial\Omega }^d{\bf h}\big)
={{\bf t}_0^{-}}\big({\bf W}_{\partial\Omega }{\bf h},{\mathcal Q}_{\partial\Omega }^d{\bf h}\big)=:{\bf D}_{\partial\Omega }{\bf h},
\end{align}
where ${\bf K}_{\partial\Omega }^*$ is the transpose of ${\bf K}_{\partial\Omega }$, and the following operators are linear and bounded,
\begin{align}
\label{ss-s2-Stokes}
&{\boldsymbol{\mathcal V}}_{\partial\Omega }:H^{-\frac{1}{2}}(\partial\Omega )^3\to H^{\frac{1}{2}}(\partial\Omega )^3,\
{\bf K}_{\partial\Omega }:H^{\frac{1}{2}}(\partial\Omega )^3\to H^{\frac{1}{2}}(\partial\Omega )^3,\\
\label{ds-s2-Stokes}
&{\bf K}_{\partial\Omega }^*:H^{-\frac{1}{2}}(\partial\Omega )^3\to
H^{-\frac{1}{2}}(\partial\Omega )^3,\ {\bf D}_{{\partial\Omega }}:H^{\frac{1}{2}}(\partial\Omega )^3\to H^{-\frac{1}{2}}(\partial\Omega )^3.
\end{align}
\end{itemize}
\end{lem}
\begin{proof}
All the above mentioned mapping properties of the layer potential operators for the Stokes system in Sobolev spaces on bounded Lipschitz domains, as well as their jump relations across a Lipschitz boundary, are well known, and we refer the reader to, e.g., \cite{Fa-Ke-Ve}, \cite{H-W}, {\cite[Propositions 4.2.5, 4.2.9, Corollary 4.3.2, Theorems 10.5.1, 10.5.3]{M-W}, \cite[Propositions 9.1, 9.2]{Sa-Se}}.
{The continuity of the operators \eqref{ss-s1-ms} and \eqref{ds-s1-ms} in weighted spaces in the exterior domains follow, e.g., from Propositions 5.2 and 6.2 in \cite{Sa-Se} deduced from the corresponding transmission problem solutions analyzed by the variational approach (cf. also \cite{B-H}, \cite[Theorem 4.1]{Ch-Mi-Na-3}, \cite[Theorem 4.1]{Lang-Mendez} for the Laplace {and some scalar operators} in weighted Sobolev spaces).}
{
Let us consider \eqref{ds-s1-msH}. Due to continuity of the first operator in \eqref{ds-s1-ms}, we only need to prove continuity of the operator
\begin{equation}
\label{WCont}
\left({\bf W}_{\partial\Omega }\right)|_{\Omega_{-}}:H^{\frac{1}{2}}(\partial\Omega )^3\to L^{2}(\Omega _{-})^3.
\end{equation}
By \eqref{fundamental-solution-Brinkman2-new}, \eqref{58} and \eqref{59},
$$
\big(\!{\bf W}_{\partial\Omega }{\bf h}\big)_k({\bf x})=
\partial _k\big(V_{\triangle }(\nu \cdot {\bf h})\big)({\bf x})
-\partial_j\big({\bf V}_{{\partial\Omega }}(\nu_{_{{j}}}{\bf h})\big)_k({\bf x})
-\partial_j\big({\bf V}_{{\partial\Omega }}(h_j\nu)\big)_k({\bf x}),
$$
where $V_{\triangle }:L^2(\partial \Omega )\to {\mathcal H}^{1}(\Omega _{-})$ is the single-layer potential for the Laplace operator, i.e.,
\begin{equation}
\label{Laplace-single-layer}
(V_{\triangle }g)({\bf x}):=-\int _{\partial \Omega }\frac{1}{4\pi }\frac{1}{|{\bf x}-{\bf y}|}g({\bf y})d\sigma _{{\bf y}},\ \ {\bf x}\in {\mathbb R}^3\setminus \partial \Omega ,
\end{equation}
which is continuous (cf. \cite[Theorem 4.1]{Ch-Mi-Na-3}, \cite[Theorem 4.1]{Lang-Mendez}).
If $h_i\in H^{\frac{1}{2}}(\partial\Omega)$, then $h_i\nu_{_j}\in L^{2}(\partial\Omega)\subset H^{-\frac{1}{2}}(\partial\Omega)$, and continuity of operator (\ref{Laplace-single-layer}) and of the first operator in \eqref{ss-s1-ms} imply continuity of {operator} \eqref{WCont}.
}

{Note {also} that continuity of the single layer operators \eqref{ss-s1-Stokes} and (\ref{ss-s1-ms}) can be shown in a more direct way, similar to the proof of Theorem 1 in \cite{Co} for the scalar case. Indeed, let ${\gamma }:{\mathcal H}^{1}({\mathbb R}^3)^3\to H^{\frac{1}{2}}(\partial \Omega )^3$ and ${\gamma }':H^{-\frac{1}{2}}(\partial \Omega )^3\to {\mathcal H}^{-1}({\mathbb R}^3)^3$ be the continuous trace operator and its continuous transpose operator, respectively.} The volume potential {operator} ${\boldsymbol{\mathcal N}}_{{\mathbb R}^3}:{\mathcal H}^{-1}({\mathbb R}^3)^3\to {\mathcal H}^{1}({\mathbb R}^3)^3$ is also continuous (see \eqref{Newtonian-Stokes-R} {in Lemma~\ref{Newtonian-weight}}). Then the single-layer potential {operator} {can be presented as ${\bf V}_{{\partial\Omega }}={\boldsymbol{\mathcal N}}_{{\mathbb R}^3}{\gamma }'$ and hence the operators}
\begin{align}
{{\bf V}_{\partial\Omega }}
:H^{-\frac{1}{2}}(\partial \Omega )^3\to {\mathcal H}^{1}({\mathbb R}^3)^3,\
\left({\bf V}_{{\partial\Omega }}\right)|_{\Omega _{\pm }}:H^{-\frac{1}{2}}(\partial \Omega )^3\to {\mathcal H}^{1}(\Omega _{\pm })^3\nonumber
\end{align}
are continuous. Since ${\mathcal H}^{1}(\Omega _{+})^3={H}^{1}(\Omega _{+})^3$, continuity of the first operators in \eqref{ss-s1-Stokes} and \eqref{ss-s1-ms} also follows.
The pressure potential operator can be written as ${\mathcal Q}_{{\partial\Omega }}^s={\mathcal Q}_{{\mathbb R}^3}{\gamma }'$, where the operator ${\mathcal Q}_{{\mathbb R}^3}:{\mathcal H}^{-1}({\mathbb R}^3)^3\to L^2({\mathbb R}^3)$
is continuous by \eqref{Newtonian-Stokes-R} in Lemma~\ref{Newtonian-weight}.
This implies continuity of the second operators in \eqref{ss-s1-Stokes} and \eqref{ss-s1-ms}.
\end{proof}

Let $H_{\rm{div}}^{1}(\Omega_\pm)^3:=\{{\bf u}\in H^{1}(\Omega_\pm)^3:{\rm{div}}\ {\bf u}=0\mbox{ in } \Omega_\pm\}$.
Since the double layer potential for the Stokes system is divergence-{free,} {the jump property} \eqref{68-s1-Stokes} and continuity of the first operator in \eqref{ds-s1-Stokes} and of operator \eqref{ds-s1-msH} imply the following assertion.

\begin{lem}
\label{property-transm}
Let $\Omega _{+}\subset {\mathbb R}^3$ be a bounded Lipschitz domain with connected boundary and $\Omega _{-}:={\mathbb R}^3\setminus \overline{\Omega }_{+}$.
If $\mathbf h \in H^{\frac{1}{2}}(\partial \Omega )^3$, then there exists ${\bf u}\in H_{\rm{div}}^{1}(\Omega_\pm)^3$ such that
$
{\gamma_{+}{\bf u}-\gamma_{-}{\bf u}=-\mathbf h \mbox{ on } \partial \Omega.}
$
\end{lem}

{For {$\Omega'$ being $\Omega_+$, $\Omega_-$ or $\mathbb R^3$}, let} us introduce the following Hilbert space and its norm,
\begin{align}
\label{H-M}
\mathfrak M(\Omega')\!:=\!\{q\in L^2(\rho^{-1};\Omega'): \nabla q\in H^{-1}_{\rm curl}(\Omega')^3\},\
\|q\|^2_{\mathfrak M(\Omega')}\!:=\|\rho^{-1}q\|^2_{L^2(\Omega')}\!+\!\|\nabla q\|^2_{H^{-1}(\Omega')^3},
\end{align}
where
$H^{-1}_{\rm curl}(\Omega')^3:=\{\mathbf f\in H^{-1}(\Omega')^3: \mathrm{curl}\,\mathbf f=\mathbf 0\}$.
Let $\mathfrak M^*(\Omega')$ denote the space dual to $\mathfrak M(\Omega')$.
Then we have the following continuous embeddings chain:
\begin{align}\label{L2M}
L^2(\rho ;\Omega')\subset
{\mathfrak M^*(\Omega')}\subset
L^2(\Omega')\subset
{\mathfrak M}(\Omega')\subset
L^2(\rho^{-1} ;\Omega')
{\subset L^2_{\rm loc}(\Omega')}.
\end{align}
Note that the function $\rho^{-1}({\mathbf x})=(1+|{\mathbf x}|^2)^{-1/2}$ belongs to $\mathcal H^1(\mathbb R^3)\subset \mathfrak M(\mathbb R^3)$ but does not belong to $L^2(\mathbb R^3)$, thus proving that $L^2(\mathbb R^3)$ and $\mathfrak M(\mathbb R^3)$ {do not coincide} (and hence $L^2(\Omega_-)$ and $\mathfrak M(\Omega_-)$ {do not coincide either)}.
For ${\Omega'}=\Omega_+$ all the spaces in \eqref{L2M} evidently coincide with $L^2({\Omega_+})$ with equivalent norms.

\begin{lem}
\label{M-pr}
$($i$)$ If $q\in\mathfrak M(\mathbb R^3)$, then $q|_{\Omega_\pm}\in\mathfrak M(\Omega_\pm)$.

$($ii$)$ If $q_\pm\in\mathfrak M(\Omega_\pm)$, then
${\rm grad}\,\mathring E_\pm q_\pm\in\widetilde H^{-1}_{\rm curl}(\Omega_\pm)^3$ and
$\mathring E_+ q_+ + \mathring E_- q_-\in\mathfrak M(\mathbb R^3)$.
\end{lem}
\begin{proof}
Item (i) follows from definition of the spaces $\mathfrak M(\mathbb R^3)$ and $\mathfrak M(\Omega_\pm)$ in \eqref{H-M}.

(ii) If $q_+\in\mathfrak M(\Omega_+)$ then $q_+\in L^2(\Omega_+)$ and evidently
${\rm grad}\,\mathring E_+ q_+\in\widetilde H^{-1}(\Omega_+)^3$.
Suppose now that $q_-\in\mathfrak M(\Omega_-)$.
Let $B_r$ be an open ball of radius $r$ such that $\overline{\Omega}_+\subset B_r$, and let
$\chi \in \mathcal D({\mathbb R}^3)$ be such that ${\rm{supp}}\, \chi \subseteq B_{2r}$, $0\leq \chi \leq 1$ and $\chi=1$ in $B_r$.
Then $\mathring E_- q_-=q_1+q_2$, where
$q_1:=\chi\mathring E_- q_-\in \widetilde H^0(\Omega_-)$, and thus $\nabla q_1\in \widetilde H^{-1}(\Omega_-)^3$,
while
$q_2:=(1-\chi)\mathring E_- q_-=\mathring E_-((1-\chi) q_-)\in L^2(\rho^{-1};\mathbb R^3)$ and
moreover,
\begin{align}
\label{nq2}
\nabla q_2=\nabla((1-\chi)\mathring E_-q_-)=(1-\chi)\nabla\mathring E_- q_{-}-\mathring E_-q_-\nabla\chi .
\end{align}
Since {$\nabla\chi$ is zero outside $B_{2r}$ and $\chi=1$ inside $B_{r}$, then $\nabla\chi=0$ in the {set complementary to} $B_{2r}\setminus \overline{B}_r$}, and we have for the last term in \eqref{nq2},
$\mathring E_-q_-\nabla\chi\in \widetilde H^0(\Omega_-)^3\subset \widetilde H^{-1}(\Omega_-)^3$.
On the other hand, the membership $q_-\in\mathfrak M(\Omega_-)$ implies $\nabla q_-\in H^{-1}(\Omega_-)^3$ and since $(1-\chi) \nabla \mathring E_-q_-=\mathring E_-((1-\chi) \nabla q_-)=0$ in $B_r\supset\partial\Omega$, we obtain  $(1-\chi) \nabla \mathring E_-q_-\in\widetilde H^{-1}(\Omega_-)$. Thus $\nabla q_2\in \widetilde H^{-1}(\Omega_-)^3$ and hence ${\rm grad}\,\mathring E_- q_-\in\widetilde H^{-1}(\Omega_-)^3$. {Since ${\rm curl}\,{\rm grad}\,\mathring E_\pm q_\pm=\mathbf 0$}
we obtain that ${\rm grad}\,\mathring E_\pm q_\pm\in\widetilde H^{-1}_{\rm curl}(\Omega_\pm)^3$. {In addition,} the last membership of statement (ii) immediately follows as well.
\end{proof}


Let $J^t$ be the Bessel potential operator of order $t$ defined by $J^t u =\mathcal F^{-1} (\rho^t\hat u)$, where $\hat u=\mathcal F u$ is the Fourier transform and $\rho({\boldsymbol\xi})=(1+|{\boldsymbol\xi}|^2)^{\frac{1}{2}}$ as defined in Section~\ref{S2.2}. By definition of the Bessel-potential spaces (see, e.g., \cite[Section 3]{Lean}),
\begin{align}\label{hatrho}
\|g\|_{H^{t}({\mathbb R}^3)}=\| \rho^{t}\hat g\|_{L^2({\mathbb R}^3)}, \
\|\rho^{t}g\|_{L^2({\mathbb R}^3)}=\|\hat g\|_{H^{t}({\mathbb R}^3)},\ \  \forall \ t\in\mathbb R.
\end{align}

\begin{lem}
\label{Newtonian-B}
The following Newtonian potential operators for the Brinkman system, $\alpha>0$, are linear and continuous,
\begin{align}
\label{Newtonian-B-R}
{\boldsymbol{\mathcal N}}_{\alpha;\mathbb R^3}:{H}^{-1}(\mathbb R^3)^3\to {H}^{1}(\mathbb R^3)^3,&
\quad {\mathcal Q}_{\mathbb R^3}:{H}^{-1}(\mathbb R^3)^3\to {\mathfrak M(\mathbb R^3)},\\
\label{Newtonian-Brinkman-SM}
{\boldsymbol{\mathcal N}}_{\alpha;\Omega_\pm}:\widetilde{H}^{-1}(\Omega_\pm)^3\to {H}^{1}(\Omega_\pm)^3,&
\quad {\mathcal Q}_{{\Omega }_\pm}:\widetilde{H}^{-1}({\Omega }_\pm)^3\to {\mathfrak M({\Omega }_\pm)},\\
\label{Newtonian-Brinkman+}
&\quad
{\mathcal Q}_{\Omega_+}:\widetilde{H}^{-1}(\Omega_+)^3\to L^{2}(\Omega_+).
\end{align}
\end{lem}
\begin{proof} In view of {e.g., \cite[Theorem 3.10]{McCracken1981},} \cite[Lemma 1.3]{Deuring} the volume potential {operator}
$
{\boldsymbol{\mathcal N}}_{\alpha;{\mathbb R}^3}:L^2({\mathbb R}^3)^3\to H^{2}({\mathbb R}^3)^3
$
is continuous. Since ${\boldsymbol{\mathcal N}}_{\alpha ;{\mathbb R}^3}$ and its adjoint ${\boldsymbol{\mathcal N}}_{\alpha ;{\mathbb R}^3}'$ coincide on test functions in ${\mathbb R}^3$, we deduce by duality that the operator
$
{\boldsymbol{\mathcal N}}_{\alpha ;{\mathbb R}^3}:H^{-2}({\mathbb R}^3)^3\to L^{2}({\mathbb R}^3)^3
$
is continuous as well. Then by interpolation the operator
${\boldsymbol{\mathcal N}}_{\alpha ;{\mathbb R}^3}:H^{-1}({\mathbb R}^3)^3\to H^{1}({\mathbb R}^3)^3$
is also continuous. By the continuity of the {Newtonian velocity} potential operator \eqref{Newtonian-B-R}, as well as the continuity of the operators $r_{\Omega _{\pm}}:{H}^{1}({\mathbb R}^3)^3\to {H}^{1}(\Omega _{\pm})^3$ and since $\widetilde{H}^{-1}({\Omega }_{\pm})^3$ is a subspace of ${H}^{-1}({\mathbb R}^3)^3$, we deduce that
the Newtonian potential {operators} {in the left of} \eqref{Newtonian-Brinkman-SM} are also continuous.

Let $\boldsymbol{\varphi}\in {H}^{-1}(\mathbb R^3)^3$ and $q={\mathcal Q}_{\mathbb R^3}\boldsymbol{\varphi}$.
Due to \eqref{hatrho},
$
\|\rho^{-1}q\|_{{L^2}(\mathbb R^3)}
=\|\hat q\|_{H^{-1}(\mathbb R^3)}.
$

By \eqref{QoT} and \eqref{5.4.21}, we have (see, e.g., \cite[Equation (3.5)]{McCracken1981})
\begin{align}
\label{hatq}
\hat q({\boldsymbol\xi})=\widehat{\Pi}_j({\boldsymbol\xi}){\hat\varphi}_j({\boldsymbol\xi})
=\frac{-i{\boldsymbol\xi}\cdot\boldsymbol{\hat\varphi}({\boldsymbol\xi})}{(2\pi)^{\frac{3}{2}}|{\boldsymbol\xi}|^2}.
\end{align}
Then, to estimate {$\|\hat q\|_{H^{-1}(\mathbb R^3)}$} we have the inequalities
\begin{align}
\label{identity-2}
(2\pi)^{\frac{3}{2}}|\langle\hat q, \hat v\rangle_{\mathbb R^3}|
&=\big|\langle|{\boldsymbol\xi}|^{-2}{\boldsymbol\xi}\cdot\boldsymbol{\hat\varphi},\hat v\rangle_{{L^2}(\mathbb R^3),\ {L^2}(\mathbb R^3)}\big|
\le \|\rho^{-1}\boldsymbol{\hat\varphi}\|_{{L^2(\mathbb R^3)^3}}\ \|\ |{\boldsymbol\xi}|^{-1}\rho \hat v\ \|_{{L^2}(\mathbb R^3)}\nonumber\\
&\le \|\boldsymbol{\varphi}\|_{H^{-1}(\mathbb R^3)^3}\left\|(1+|{\boldsymbol\xi}|^{-2})^{{\frac{1}{2}}} \hat v\right \|_{L^2(\mathbb R^3)}
\le \|\boldsymbol{\varphi}\|_{H^{-1}(\mathbb R^3)^3} \left\|(1+|{\boldsymbol\xi}|^{-1})\hat v\right \|_{L^2(\mathbb R^3)}
\nonumber\\
&\le \|\boldsymbol{\varphi}\|_{H^{-1}(\mathbb R^3)^3} \left(\|\hat v\|_{L^2(\mathbb R^3)}+2\|\nabla \hat v\|_{L^2(\mathbb R^3)}\right)
\le 2\|\boldsymbol{\varphi}\|_{H^{-1}(\mathbb R^3)^3} \|\hat v\|_{H^1(\mathbb R^3)},
\end{align}
that hold for any $\hat v\in H^1(\mathbb R^3)$ by the Hardy inequality,
$\|\ |{\boldsymbol\xi}|^{-1} \hat v({\boldsymbol\xi})\ \|_{L^2({\mathbb R}^3)}\le 2\|\nabla \hat v\|_{L^2({\mathbb R}^3)^3}$
(see, e.g., 
\cite[Inequality (1.3.3)]{Maz'ya2011}).
This implies
$$
\|\rho^{-1}q\|_{L^2(\mathbb R^3)}
=\|\hat q\|_{H^{-1}(\mathbb R^3)}
\le\frac{2}{(2\pi)^{\frac{3}{2}}}\|\boldsymbol{\varphi}\|_{H^{-1}(\mathbb R^3)^3}
$$
and hence continuity of the operator
${\mathcal Q}_{\mathbb R^3}:{H}^{-1}(\mathbb R^3)^3\to {L^2(\rho^{-1};\mathbb R^3)}$.
Continuity of the operator
${\rm grad}{\mathcal Q}_{\mathbb R^3}:{H}^{-1}(\mathbb R^3)^3\to {H}^{-1}_{\rm curl}(\mathbb R^3)^3$
follows from \eqref{hatq}. Indeed, we have 
\begin{align}
\|{\rm{grad}}({\mathcal Q}_{\mathbb R^3}\boldsymbol{\varphi})\|_{H^{-1}(\mathbb R^3)^3}&=
\|{\rho }^{-1}{\mathcal F}({\rm{grad}}({\mathcal Q}_{\mathbb R^3}\boldsymbol{\varphi}))\|_{L^2(\mathbb R^3)^3}=\frac{2\pi }{(2\pi )^{\frac{3}{2}}}\left\|{\rho }^{-1}i\boldsymbol{\xi }\frac{\boldsymbol{\xi }}{|\boldsymbol{\xi }|^2}\cdot \hat{\boldsymbol{\varphi}}\right\|_{L^2(\mathbb R^3)^3}\nonumber\\
&\leq \frac{1}{(2\pi )^{\frac{1}{2}}}\|{\rho }^{-1}\hat{\boldsymbol{\varphi}}\|_{L^2(\mathbb R^3)^3}=
\frac{1}{(2\pi )^{\frac{1}{2}}}\|\boldsymbol{\varphi}\|_{H^{-1}(\mathbb R^3)^3},\nonumber
\end{align}
which shows the desired continuity. This, in turn, implies continuity of the operator in the right of \eqref{Newtonian-B-R} and, by Lemma~\ref{M-pr}(i), also continuity of the operator in the right of \eqref{Newtonian-Brinkman-SM} along with the well-known continuity of operator \eqref{Newtonian-Brinkman+}.
\end{proof}
\comment{Arguments for the continuity of the operator ${\rm grad}{\mathcal Q}_{\mathbb R^3}:{H}^{-1}(\mathbb R^3)^3\to {H}^{-1}_{\rm curl}(\mathbb R^3)^3$:
$|\mathcal F{\rm grad}{\mathcal Q}_{\mathbb R^3}\boldsymbol{\varphi}|
=(2\pi)^{1/2}|\boldsymbol\xi\hat q({\boldsymbol\xi})|
\leq (2\pi)^{-1}|\boldsymbol{\hat\varphi}|$. This implies $$\|{\rm grad}{\mathcal Q}_{\mathbb R^3}\boldsymbol{\varphi}\|_{H^{-1}(...)}=\|\rho ^{-1}\mathcal F{\rm grad}{\mathcal Q}_{\mathbb R^3}\boldsymbol{\varphi}\|_{L^2(...)}\leq (2\pi)^{-1}\|\rho ^{-1}\boldsymbol{\hat\varphi}\|_{L^2(...)}=(2\pi)^{-1}\|\varphi \|_{H^{-1}(...)}.$$}


If
$
({\bf u}_{\pm},\pi_{\pm} ,{\tilde{\bf f}}_{\pm})\in
H^{1}_{\rm{div}}({\Omega_{\pm}})^3\times
\mathfrak M({\Omega_{\pm}})\times
\widetilde{H}^{-1}({\Omega_{\pm}})^3$ is such that
${\boldsymbol{\mathcal L}}_{\alpha }({\bf u}_{\pm},\pi_{\pm} )={\tilde{\bf f}}_{\pm}|_{\Omega_{\pm}},
$
the conormal derivatives ${\bf t}_{\alpha}^{\pm}({\bf u}_{\pm},\pi_{\pm} ;{\tilde{\bf f}}_{\pm})$ are still well defined by \eqref{conormal-generalized-1-Robin}, e.g., due to Remark~\ref{CSD}. The first Green identity \eqref{conormal-generalized-2} also holds true by arguments similar to the ones in the proof of Lemma~\ref{conormal-derivative-generalized-Robin-S} and by the formula
\begin{align}
\langle \pi_{\pm},{\rm{div}}\ {{\bf w}_\pm}\rangle _{\Omega_{\pm}}
:=\langle \mathring E_{\pm}\pi_{\pm},{\rm{div}}\ {\bf W}_{\pm }\rangle _{{\mathbb R}^3}
=-\langle {\rm{grad}}(\mathring E_{\pm} \pi_{\pm}),{{\bf W}_\pm}\rangle _{{\mathbb R}^3}
=-\langle {\rm{grad}}(\mathring E_{\pm} \pi_{\pm}),{{\bf w}_\pm}\rangle _{\Omega_{\pm}},\nonumber
\end{align}
where ${{\bf W}_\pm}\in H^1({\mathbb R}^3)^3$ is such that $r_{\Omega _{\pm }}{\bf W}_\pm={\bf w}_\pm.$
Note that
${\rm grad}(\mathring E_{\pm} \pi_{\pm})\in \widetilde H^{-1}(\Omega_\pm)$
by the definition of the space $\mathfrak M(\Omega _{\pm })$ and by Lemma \ref{M-pr}.
Moreover, if
${\tilde{\bf f}}_{\pm}|_{\Omega_{\pm}}={\boldsymbol{\mathcal L}}_{\alpha }({\bf u}_{\pm},\pi_{\pm} )\in L^2(\Omega_{\pm})^3$, then we can take
${\tilde{\bf f}}_{\pm}=\mathring E_{\Omega_{\pm}}\boldsymbol{\mathcal L}_\alpha({\bf u}_{\pm},\pi_{\pm} )\in \widetilde H^0(\Omega_{\pm})^3\subset L^2(\mathbb R^3)^3$
in the definition of conormal derivative \eqref{conormal-generalized-1-Robin}, which makes it canonical,
${\bf t}_{\alpha}^{\pm}({\bf u},\pi)$, cf., e.g., \cite{Co}, \cite{Mikh}, \cite{Mikh-3}. In this case the identities \eqref{conormal-generalized-2} take the form
\begin{multline}
\label{conormal-generalized-2c}
{\pm}\left\langle {\bf t}_{\alpha}^{\pm}({\bf u}_{\pm},\pi_{\pm}),\gamma_{\pm}{{\bf w}_\pm}\right\rangle_{\partial\Omega  }=
2\langle {\mathbb E}({\bf u}_{\pm}),{\mathbb E}({{\bf w}_\pm})\rangle _{\Omega_{\pm}}
+\alpha \langle {\bf u}_{\pm},{{\bf w}_\pm}\rangle _{\Omega_{\pm}}\\
-\langle \pi_{\pm},{\rm{div}}\ {{\bf w}_\pm}\rangle _{\Omega_{\pm}}
+\langle
\boldsymbol{\mathcal L}_\alpha({\bf u}_{\pm},\pi_{\pm} ),{{\bf w}_\pm}\rangle_{\Omega_{\pm}},\
\forall\,{{\bf w}_\pm}\in H^{1}({\Omega_{\pm}})^3.
\end{multline}
Let now
{$({\bf u}_{\pm },\pi _{\pm }),({\bf w}_{\pm },q_{\pm })\in H^{1}_{\rm{div}}({\Omega_{\pm}})^3\times\mathfrak M({\Omega_{\pm}})$}
be such that
${\boldsymbol{\mathcal L}}_{\alpha }({\bf u}_{\pm},\pi_{\pm}), {\boldsymbol{\mathcal L}}_{\alpha }({\bf w}_{\pm},q_{\pm})\in L^2(\Omega_{\pm})^3$.
Then subtracting from \eqref{conormal-generalized-2c} its counterpart with swapped roles of $({\bf u}_{\pm},\pi_{\pm})$ and $({\bf w}_{\pm},q_{\pm})$, we arrive at the second Green identity for the Brinkman system,
{\begin{align}
\label{Green}
\!\!\!{\pm}\!\left\langle {\bf t}_{\alpha}^{\pm}({\bf u}_{\pm},\pi_{\pm}),\gamma_{\pm}{\bf w}_\pm \right\rangle _{\partial\Omega}
\!{\mp}\!\left\langle {\bf t}_{\alpha}^{\pm}({\bf w}_{\pm},q_{\pm}),\gamma_{\pm}{\bf u}_\pm\right\rangle_{\!\partial\Omega}
\!\!\!=\!\!\langle\mathcal L_\alpha({\bf u}_{\pm},\pi_{\pm} ),{\bf w}_\pm\rangle_{\Omega_{\pm}}
\!\!\!-\!\!\langle\mathcal L_\alpha({\bf w}_{\pm},q_{\pm} ),{\bf u}_\pm\rangle_{\Omega_{\pm}}.
\end{align}}
If, in addition, $\gamma_{+}{\bf w}_+=\gamma_{-}{\bf w}_-{=:\gamma{\bf w}}$ and
${\bf t}_{\alpha}^{+}({\bf w}_{+},q_{+})={\bf t}_{\alpha}^{-}({\bf w}_{-},q_{-}){=:{\bf t}_{\alpha}({\bf w},q)}$, then summing up equalities \eqref{Green} with upper and lower signs, the second Green identity reduces to
\begin{align}
\label{Green-s}
\left\langle [{\bf t}_{\alpha}({\bf u},\pi)],\gamma{\bf w}
\right\rangle _{\partial\Omega}
-\left\langle {\bf t}_{\alpha}({\bf w},q),[\gamma{\bf u}]\right\rangle_{\!\partial\Omega}
=\langle\boldsymbol{\mathcal L}_\alpha({\bf u},\pi),{\bf w}\rangle_{\Omega_+\cup\Omega_-}
-\langle\boldsymbol{\mathcal L}_\alpha({\bf w},q),{\bf u}\rangle_{\Omega_+\cup\Omega_-},
\end{align}
where {${\bf u},{\bf w}\in H^{1}({\mathbb R}^3\setminus \partial \Omega )^3$ are defined by ${\bf u}|_{\Omega _{\pm }}:={\bf u}_{\pm }$ and ${\bf w}|_{\Omega _{\pm }}:={\bf w}_{\pm }$}, and
$[\gamma{\bf u}]:={\gamma }_{+}({\bf u}_+)-{\gamma }_{-}({\bf u}_-)$,
$\left[{\bf t}_{\alpha }({\bf u},\pi)\right]:=
{\bf t}_{\alpha }^{+}({\bf u}_+,\pi_+)-{\bf t}_{\alpha }^{-}({\bf u}_-,\pi_-)$.
For any fixed point ${\bf y}\in {\mathbb R}^3\setminus \partial \Omega $ and any index $j\in \{1,2,3\}$, the fundamental solution $({\mathcal G}_{j.}^{\alpha }(\cdot -{\bf y}),\ \Pi _j^{\alpha }(\cdot -{\bf y}))$, given by \eqref{5.4.21}, satisfies all the above conditions {imposed to the couple} $({\bf w},q)$ in $\Omega _{+}\cup \Omega _{-}\setminus \overline B_\epsilon(\bf y)$, where $B_\epsilon(\bf y)$ is a small ball centered at the point ${\bf y}$. Modifying \eqref{Green-s} to the domains without $B_\epsilon(\bf y)$, substituting there $({\mathcal G}_{j.}^{\alpha }(\cdot-{\bf y}),\ \Pi _j^{\alpha }(\cdot-{\bf y}))$ for $({\bf w},q)$ and taking appropriate limits as $\epsilon\to 0$, we arrive at the following {\it third Green identity} ({\it representation formula}),
\begin{equation}
\label{dl-Green}
{\bf u}
={\bf V}_{\alpha ;\partial \Omega }\left[{\bf t}_{\alpha }({\bf u},p)\right]
-{\bf W}_{\alpha ;\partial \Omega }\left[{\gamma }{\bf u}\right]
+{\boldsymbol{\mathcal N}}_{\alpha;\Omega_+\cup\Omega_-}{\boldsymbol{\mathcal L}}_{\alpha }({\bf u},p)
\ \mbox{ in } \ {\mathbb R}^3\setminus \partial \Omega ,
\end{equation}
for any $({\bf u},\pi )\in H^{1}_{\rm{div}}({\Omega_{\pm}})^3\times\mathfrak M({\Omega_{\pm}})$
such that
${\boldsymbol{\mathcal L}}_{\alpha }({\bf u}_{\pm},\pi_{\pm})\in L^2(\Omega_{\pm})^3$.

{Let us give} the main properties of layer potentials for the Brinkman system{, which are also partly available in \cite[Lemma 3.1]{K-L-W} for interior and in} \cite[Proposition 2.3]{B-H} for exterior Lipschitz domains.
\begin{lem}
\label{layer-potential-properties} Let
$\Omega :=\Omega _{+}\subset {\mathbb R}^3$ be a bounded Lipschitz domain with connected boundary $\partial\Omega $. Let $\alpha >0$ be a given constant.
\begin{itemize}
\item[$(i)$] Then the following operators are linear and bounded,
\begin{align}
\label{ss-s1}
&\left({\bf V}_{\alpha ;\partial\Omega }\right)|_{\Omega _{+}}:H^{-\frac{1}{2}}(\partial\Omega )^3\to H^{1}(\Omega _{+})^3,\
\left({\mathcal Q}_{\alpha ;\partial\Omega }^s\right)|_{\Omega _{+}}:H^{-\frac{1}{2}}(\partial\Omega )^3\!\to L^2(\Omega _{+}),\\
\label{ds-s1}
&\left({\bf W}_{\alpha ;\partial\Omega }\right)|_{\Omega _{+}}:H^{\frac{1}{2}}(\partial\Omega )^3\to H^{1}(\Omega _{+})^3,\
\big({\mathcal Q}_{\alpha ;\partial\Omega }^d\big)|_{\Omega _{+}}\!:\!H^{\frac{1}{2}}(\partial\Omega )^3\to L^2(\Omega _{+}).
\end{align}
\item[$(ii)$] Let $\Omega _{-}:={\mathbb R}^3\setminus \overline{{\Omega }}$. Then the following
operators are linear and bounded,
\begin{align}
\label{exterior-weight}
&{\left({\bf V}_{\alpha ;\partial\Omega }\right)|_{\Omega _{-}}\!:\!H^{-\frac{1}{2}}(\partial\Omega )^3\!\to H^{1}({\Omega }_{-})^3,}\ \left({\mathcal Q}_{\alpha ;\partial\Omega }^s\right)|_{\Omega _{-}}:H^{-\frac{1}{2}}(\partial\Omega )^3\to L^{2}(\Omega _{-}),\\
\label{exterior-weight-s}
&{\left({\bf W}_{\alpha ;\partial\Omega }\right)|_{\Omega _{-}}\!:\!H^{\frac{1}{2}}(\partial\Omega )^3\!\to \! H^{1}({\Omega }_{-})^3,}\ {\big({\mathcal Q}_{\alpha ;\partial\Omega }^d\big)|_{\Omega _{-}}:H^{\frac{1}{2}}(\partial\Omega )^3\to \mathfrak M(\Omega _{-})}.
\end{align}
\item[$(iii)$]
Let ${\bf h}\in H^{\frac{1}{2}}(\partial\Omega )^3$ and ${\bf g}\in
H^{-\frac{1}{2}}(\partial\Omega )^3$. Then the following relations hold a.e. on $\partial\Omega $,
\begin{align}
\label{68}
&\gamma_+\big({\bf V}_{\alpha ;\partial\Omega }{\bf g}\big)
=\gamma_{-}\big({\bf V}_{\alpha
;{ \partial\Omega } }{\bf g}\big)
=:{\boldsymbol{\mathcal V}}_{\alpha ; \partial\Omega }{\bf g},
\\
\label{68-s1}
&
\frac{1}{2}{\bf h}+\gamma_{+}({\bf W}_{\alpha ;\partial\Omega }{\bf h})=
-\frac{1}{2}{\bf h}+\gamma_{-}({\bf W}_{\alpha ;\partial\Omega }{\bf h})
=:{\bf K}_{\alpha ;{ \partial\Omega }}{\bf h},
\\
\label{70aaa}
&-\frac{1}{2}{\bf g}+{\bf t}_{\alpha }^{+}\left({\bf V}_{\alpha ;\partial\Omega }{\bf g},{\mathcal Q}_{\alpha
;\partial\Omega }^s{\bf g}\right)
=\frac{1}{2}{\bf g}+{\bf t}_{\alpha }^{-}\left({\bf V}_{\alpha ;\partial\Omega }{\bf g},{\mathcal Q}_{\alpha
;\partial\Omega }^s{\bf g}\right)
=:{\bf K}_{\alpha ;{ \partial\Omega } }^*{\bf g},
\\
\label{70aaaa}
&{\bf t}_{\alpha }^{+}\big({\bf W}_{\alpha ;\partial\Omega }{\bf h},{\mathcal Q}_{\alpha ; \partial\Omega }^d{\bf h}\big)
={\bf t}_{\alpha }^{-}\big({\bf W}_{\alpha ;\partial\Omega }{\bf h},{\mathcal Q}_{\alpha ;\partial\Omega }^d{\bf h}\big)
=:{\bf D}_{\alpha ;\partial\Omega }{\bf h},
\end{align}
where ${\bf K}_{\alpha ;\partial\Omega }^*$ is the transpose of ${\bf K}_{\alpha ;\partial\Omega }$, and the following operators are linear and bounded,
\begin{align}
\label{ss-s2}
&{\boldsymbol{\mathcal V}}_{\alpha ;\partial\Omega }:H^{-\frac{1}{2}}(\partial\Omega )^3\to H^{\frac{1}{2}}(\partial\Omega )^3,\
{\bf K}_{\alpha ;\partial\Omega }:H^{\frac{1}{2}}(\partial\Omega )^3\to H^{\frac{1}{2}}(\partial\Omega )^3,\\
\label{ds-s2}
&{\bf K}_{\alpha ;\partial\Omega }^*:H^{-\frac{1}{2}}(\partial\Omega )^3\to
H^{-\frac{1}{2}}(\partial\Omega )^3,\ {\bf D}_{\alpha ;\partial\Omega }:H^{\frac{1}{2}}(\partial\Omega )^3\to H^{-\frac{1}{2}}(\partial\Omega )^3.
\end{align}
\end{itemize}
\end{lem}
\begin{proof}
{Let ${\gamma }:{H}^{1}({\mathbb R}^3)^3\to H^{\frac{1}{2}}(\partial \Omega )^3$ and ${\gamma }':H^{-\frac{1}{2}}(\partial \Omega )^3\to {H}^{-1}({\mathbb R}^3)^3$ be, respectively, the trace operator and its transpose operator}. Both operators are continuous. The volume potential {operator} ${\boldsymbol{\mathcal N}}_{\alpha ;{\mathbb R}^3}:H^{-1}({\mathbb R}^3)^3\to H^{1}({\mathbb R}^3)^3$ is also continuous (see (\ref{Newtonian-B-R})). Then, similar to the argument in the proof of Lemma~\ref{layer-potential-properties-Stokes}, the single-layer {operator} {can be represented as
${\bf V}_{{\alpha ; \partial\Omega }}={\boldsymbol{\mathcal N}}_{\alpha ;{\mathbb R}^3} {\gamma }'$
and hence}
\begin{align}
\label{sl-1-ms}
{{\bf V}_{\alpha ; \partial\Omega }}
:H^{-\frac{1}{2}}(\partial \Omega )^3\to H^{1}({\mathbb R}^3)^3,\
\left({\bf V}_{{\alpha ; \partial\Omega }}\right)|_{\Omega _{\pm }}:H^{-\frac{1}{2}}(\partial \Omega )^3\to H^{1}(\Omega _{\pm })^3,
\end{align}
are continuous operators (see \cite[Corollary 2.5]{B-H} for a different argument). This property shows the {continuity} of the first operators in (\ref{ss-s1}) and (\ref{exterior-weight}), {as well as continuity of traces \eqref{68}}.
{Continuity} of the second operators in (\ref{ss-s1}) and (\ref{exterior-weight}), ${\mathcal Q}_{{\alpha ; \partial\Omega }}^s:H^{-\frac{1}{2}}(\partial \Omega )^3\to L^{2}(\Omega _{\pm })$, follows from the equality of the {single layer pressure potentials} for the Brinkman and Stokes systems,
{${\mathcal Q}_{{\alpha ; \partial\Omega}}^s={\mathcal Q}_{{\partial\Omega}}^s$,
see \eqref{58},} and continuity of the second operators in (\ref{ss-s1-Stokes}), \eqref{ss-s1-ms}.
{The jump property \eqref{70aaa} follows by a straightforward adaptation of the corresponding arguments of \cite[Lemma 4.1]{Co}.}

Next we show the continuity of second operators in \eqref{ds-s1} and (\ref{exterior-weight-s}).
{First we use formula \eqref{pressure-tensor-B} and the second formula in \eqref{fundamental-solution-Brinkman2}, and deduce that the fundamental pressure tensor for the Brinkman system, $\Lambda ^{\alpha }$, can be written as
${\Lambda }^{\alpha }_{jk}({\bf x},{\bf y})={\Lambda }_{jk}({\bf x},{\bf y})+ \dfrac{\alpha}{{4\pi}|{\bf x}-{\bf y}|}\delta _{jk}.$
Therefore,
the Brinkman double layer pressure {potential}
is given by
\begin{align}
\label{double-layer-pressure-B}
{\mathcal Q}^d_{\alpha ;\partial \Omega }\mathbf h
={\mathcal Q}^d_{\partial \Omega }\mathbf h -\alpha V_{\triangle }({\bf h} \cdot \nu )
\ \mbox{ in } \ {\mathbb R}^3\setminus \partial \Omega,\ \forall \ \mathbf h\in H^{\frac{1}{2}}(\partial\Omega )^3,
\end{align}
where $V_{\triangle }$ is the Laplace single-layer {potential \eqref{Laplace-single-layer}. Since ${\bf h}\cdot\nu\in L^2(\partial\Omega)$, the continuity of the operator
$V_{\triangle }:L^2(\partial \Omega )\subset H^{-\frac{1}{2}}(\partial\Omega)\to {\mathcal H}^{1}(\Omega _{-})\subset \mathfrak M(\Omega _{-})$ and continuity of the second operator in \eqref{ds-s1-ms}}
{along with the inclusion $L^2(\Omega _{-})\subset \mathfrak M(\Omega _{-})$ imply continuity of the
second operator in (\ref{exterior-weight-s}).}
The continuity of the second operator in \eqref{ds-s1} follows with a similar argument.}

Let us show continuity of the first operators in \eqref{ds-s1} and \eqref{exterior-weight-s}, by using arguments similar to those in the proof of \cite[Theorem 1]{Co}.
For a function $\mathbf h\in H^{\frac{1}{2}}(\partial\Omega )^3$ let us consider the couple
$({\bf u}_{\mathbf h},p_{\mathbf h})
=\big({\bf W}_{\partial\Omega }{\bf h},{\mathcal Q}_{\partial\Omega }^d{\bf h}\big)$
given by the Stokes double layer velocity and pressure potentials.
Due to statement (i) of Lemma~\ref{layer-potential-properties-Stokes},
$({\bf u}_{\mathbf h},p_{\mathbf h})\in {H}^{1}(\Omega _\pm)^3\times L^2(\Omega _\pm)\subset {H}^{1}(\Omega _\pm)^3\times \mathfrak M(\Omega _\pm)$
and, moreover,
${\boldsymbol{\mathcal L}}_{\alpha }({\bf u}_{\mathbf h},p_{\mathbf h})
={\boldsymbol{\mathcal L}}_{0}({\bf u}_{\mathbf h},p_{\mathbf h})-\alpha {\bf u}_{\mathbf h}
={-\alpha}{\bf u}_{\mathbf h}\in {H}^{1}(\Omega _\pm)^3\subset L^2(\Omega _\pm),$
where ${\boldsymbol{\mathcal L}}_{\alpha }$ is given by \eqref{matrix-pseudo-diff1}.
Then $({\bf u}_{\mathbf h},p_{\mathbf h})$ satisfy the representation formula
(the third Green identity) \eqref{dl-Green}.
Due to {\eqref{68-s1-Stokes}} and \eqref{70aaaa-Stokes}, we have $[{\gamma }{\bf u}_{\mathbf h}]=-\mathbf h$, and moreover, $\left[{\bf t}_{\alpha }({\bf W}_{\partial\Omega }{\bf h},{\mathcal Q}_{\partial\Omega }^d{\bf h})\right]={\bf 0}$.
Then the Brinkman double layer velocity potential can be also written as
\begin{equation}
\label{dl-GreenW}
{\bf W}_{\alpha ;\partial \Omega }\mathbf h={\bf W}_{\partial\Omega }{\bf h}
+\alpha{\boldsymbol{\mathcal N}}_{\alpha;\mathbb R^3}{\bf W}_{\partial\Omega }{\bf h}
\ \mbox{ in } \ {\mathbb R}^3\setminus \partial \Omega,\ \forall \ \mathbf h\in H^{\frac{1}{2}}(\partial\Omega )^3 .
\end{equation}
{In addition, the last term in \eqref{double-layer-pressure-B} can be expressed as
$
V_{\triangle }({\bf h} \cdot \nu )=-\mathcal Q_{\mathbb R^3}{\bf W}_{\partial\Omega }{\bf h}.
$
}
Indeed, by using again the property that $[{\gamma }{\bf u}_{\mathbf h}]=-\mathbf h$, where
${\bf u}_{\mathbf h}={\bf W}_{\partial\Omega }{\bf h}$, we obtain
\begin{align}
\left(\mathcal Q_{\mathbb R^3}{\bf u}_{\mathbf h}\right)({\bf x})
&=-\int_{\Omega_-\cup\Omega_+}\frac{{\bf x}-{\bf y}}{4\pi|{\bf x}-{\bf y}|^3}\cdot{\bf u}_{\mathbf h}({\bf y})d{\bf y}
=-\frac{1}{4\pi}\int_{\Omega_-\cup\Omega_+}\left\{\nabla_y\frac{1}{|{\bf x}-{\bf y}|}\right\}
\cdot{\bf u}_{\mathbf h}({\bf y})d{\bf y}\nonumber\\
&=-\frac{1}{4\pi}\int_{\partial\Omega}\frac{1}{|{\bf x}-{\bf y}|}
\nu_y\cdot[\gamma{\bf u}_{\mathbf h}({\bf y})]d{\bf y}
+\frac{1}{4\pi}\int_{\Omega_-\cup\Omega_+}\frac{1}{|{\bf x}-{\bf y}|}{{\rm div}}\,{\bf u}_{\mathbf h}({\bf y})d{\bf y}\nonumber\\
&=V_{\triangle }(\nu\cdot[\gamma{\bf u}_{\mathbf h}])({\bf x})
=-V_{\triangle }({\bf h} \cdot \nu)({\bf x}).\nonumber
\end{align}
\comment{\rd Then by \eqref{double-layer-pressure-B} and \eqref{dlp} the Brinkman double layer pressure potential can be expressed as
\begin{align}
\label{double-layer-pressure-B-1}
{\mathcal Q}^d_{\alpha ;\partial \Omega }\mathbf h
={\mathcal Q}^d_{\partial \Omega }\mathbf h +\mathcal Q_{\mathbb R^3}{\bf W}_{\partial\Omega }{\bf h}
\ \mbox{ in } \ {\mathbb R}^3\setminus \partial \Omega .
\end{align}
} 
Now, the continuity of the operators involved in \eqref{dl-GreenW} leads to the continuity of the first operators in (\ref{ds-s1}) and \eqref{exterior-weight-s}.
Jump formulas (\ref{68-s1}) for the double-layer potential ${\bf W}_{\alpha ;\partial\Omega }{\bf h}$ follow from the formula \eqref{dl-GreenW}, combined with the jump relations \eqref{68-s1-Stokes} satisfied by the Stokes double layer potential ${\bf W}_{\partial\Omega }{\bf h}$, as well as the continuity of the Brinkman Newtonian potential ${\boldsymbol{\mathcal N}}_{\alpha;\mathbb R^3}$ {across $\partial \Omega $} (see \eqref{Newtonian-B-R}).
Continuity of the conormal derivatives \eqref{70aaaa} is implied by the equality
${\bf t}_{\alpha }^{\pm}\big({\bf W}_{\partial\Omega }{\bf h},{\mathcal Q}_{\partial\Omega }^d{\bf h}\big)
={\bf t}_{0}^{\pm}\big({\bf W}_{\partial\Omega }{\bf h},{\mathcal Q}_{\partial\Omega }^d{\bf h}\big)$,
{relations \eqref{double-layer-pressure-B} and \eqref{dl-GreenW},
the embedding ${\boldsymbol{\mathcal N}}_{\alpha;\mathbb R^3}{\bf W}_{\partial\Omega}{\bf h}\in {H}^{2}_{\rm loc}(\mathbb R^3)^3$
and continuity of the potential $V_{\triangle }({\bf h} \cdot \nu )$ across $\partial\Omega$.}
Jump formulas \eqref{68-s1}-\eqref{70aaaa} can be also obtained by exploiting arguments similar to those {for} $\alpha =0$ (see, e.g., \cite{M-W} for further details).
\end{proof}

\section*{\bf Acknowledgements}
The research has been partially supported by the grant EP/M013545/1: "Mathematical Analysis of Boundary-Domain Integral Equations for Nonlinear PDEs" from the EPSRC, UK. The first author acknowledges also the support of the grant PN-II-ID-PCE-2011-3-0994 of the Romanian National Authority for Scientific Research, CNCS - UEFISCDI. The second author acknowledges also the support of INdAM-GNAMPA. 

\end{document}